\documentclass[a4paper]{amsart}
\usepackage[utf8]{inputenc}
\usepackage{amsfonts,amsmath,amsthm,graphicx}
\usepackage{verbatim}
\usepackage[hidelinks]{hyperref}
\usepackage{array,float}
\usepackage[toc,page]{appendix}
\usepackage{caption, subcaption}
\usepackage[a4paper]{geometry}
\numberwithin{equation}{section}
\usepackage[foot]{amsaddr}

\newcommand{\matdev}{\partial^{\bullet}}

\newcommand{\gradg}{\nabla_{\Gamma}}
\newcommand{\gradgh}{\nabla_{\Gamma_h}}
\newcommand{\lapg}{\Delta_{\Gamma}}

\newcommand{\invh}{\mathcal{G}_h}
\newcommand{\invsh}{\mathcal{G}_{S_h}}
\newcommand{\intinvsh}{\bar{\mathcal{G}}_{S_h}}
\newcommand{\normh}[2]{\left \| #2 \right\|_{h,#1}}
\newcommand{\normsh}[1]{\left\| #1 \right\|_{S_h}}
\newcommand{\inormh}[1]{\left\| #1 \right\|_{-h}}

\newcommand{\matdevtau}{\partial^{\bullet}_\tau}

\newcommand{\Xint}[1]{\mathchoice
	{\XXint\displaystyle\textstyle{#1}}%
	{\XXint\textstyle\scriptstyle{#1}}%
	{\XXint\scriptstyle\scriptscriptstyle{#1}}%
	{\XXint\scriptscriptstyle\scriptscriptstyle{#1}}%
	\!\int}
\newcommand{\XXint}[3]{{\setbox0=\hbox{$#1{#2#3}{\int}$ }
		\vcenter{\hbox{$#2#3$ }}\kern-.6\wd0}}

\newcommand{\dashint}{\Xint-}
\newcommand{\mval}[2]{\dashint_{#2} #1}
\newcommand{\esssup}[1]{\underset{#1}{\operatorname{ess \ sup \ }}}

\newcommand{\Udh}{U_h^{\delta}}
\newcommand{\Wdh}{W_h^{\delta}}
\newcommand{\Udhn}[1]{U_h^{#1,\delta}}
\newcommand{\Wdhn}[1]{W_h^{#1,\delta}}
\newcommand{\fd}{f^{\delta}}
\newcommand{\Ech}{E^{\text{CH}}}

\newcommand{\Echdh}{E^{\text{CH},\delta}_h}
\newcommand{\intm}{\bar{m}_h}

\theoremstyle{plain}
\newtheorem{theorem}{Theorem}[section]
\newtheorem{corollary}[theorem]{Corollary}
\newtheorem{proposition}[theorem]{Proposition}
\newtheorem{lemma}[theorem]{Lemma}
\newtheorem{definition}[theorem]{Definition}
\newtheorem{notation}[theorem]{Notation}
\newtheorem{assumption}[theorem]{Assumption}
\newtheorem{remark}[theorem]{Remark}

\begin{document}

\begin{abstract}
In this paper we study semi-discrete and fully discrete evolving surface finite element schemes for the Cahn-Hilliard equation with a logarithmic potential. Specifically we consider linear finite elements discretising space and backward Euler time discretisation. Our analysis relies on a specific geometric assumption on the evolution of the surface. Our main results are $L^2_{H^1}$ error bounds for both the semi-discrete and fully discrete schemes, and we provide some numerical results.
\end{abstract}

\author{Charles M. Elliott \and Thomas Sales}
\title[ESFEM for the Cahn-Hilliard equation with a logarithmic potential]{An evolving surface finite element method for the Cahn-Hilliard equation with a logarithmic potential}
\address{Mathematics Institute, Zeeman Building, University of Warwick, Coventry CV4 7AL, UK}
\email{\href{mailto:c.m.elliott@warwick.ac.uk}{c.m.elliott@warwick.ac.uk}, \href{mailto:tom.sales@warwick.ac.uk}{tom.sales@warwick.ac.uk}}
\subjclass[2020]{65M60, 65M16, 35K67}
\keywords{evolving surface finite element method, Cahn-Hilliard equation, error bounds}
\date{}

\maketitle

\section{Introduction}
\label{section:logch intro}
We are interested in evolving surface finite element (ESFEM) schemes for the Cahn-Hilliard equation posed on a sufficiently smooth, closed, orientable evolving surface, $\Gamma(t) \subset \mathbb{R}^3$.
The evolving surface Cahn-Hilliard equation, as formulated in \cite{caetano2021cahn,caetano2023regularization}, is given by
\begin{align}
	\begin{gathered}
	\matdev u + (\gradg \cdot V)u = \lapg w,\\
	w = -\varepsilon \lapg u + \frac{1}{\varepsilon}F'(u),
	\end{gathered} \label{cahnhilliard}
\end{align}
subject to the initial condition $u(0) = u_0$ for suitable initial data.
Our focus is on the (singular) logarithmic potential function
\[F(r) := \frac{\theta}{2 \theta_c} F_{\log}(r) + \frac{1-r^2}{2},\]
where $F_{\log}(r) := (1-r) \log(1-r) + (1+r) \log(1+r)$, and $0 < \theta < \theta_c$ corresponds to the (assumed constant) temperature of the system, with $\theta_c$ being some critical temperature.
The condition $0 < \theta < \theta_c$ ensures that the potential takes a double-well form, which is seen to have two minima of equal size but differing sign.
For ease of presentation we set $\theta_c = 1$ and hence $\theta \in (0,1)$.
We will also use the notation $f(r) := F_{\log}'(r) = \log \left(\frac{1+r}{1-r} \right)$ throughout.
We expand on the appropriate differential operators in the next section.

The system \eqref{cahnhilliard} is often studied by use of the Ginzburg-Landau functional,
\begin{align}
	\label{glfunctional}
	\Ech[u;t] := \int_{\Gamma(t)} \frac{\varepsilon |\gradg u|^2}{2} + \frac{1}{\varepsilon} F(u),
\end{align}
where the constant $\varepsilon > 0$ is often a small coefficient associated with the width of a transition layer connecting  two  phase domains in which the solution takes values close to the minima of $F(\cdot)$.
On a stationary domain this is natural, as \eqref{cahnhilliard} is the $H^{-1}$ gradient flow of the Ginzburg-Landau functional (see \cite{bansch2023interfaces} for example).
It is noted in \cite{elliott2015evolving} that this is not the case for an arbitrary evolving surface, and on an evolving domain this functional is known to be bounded, but not necessarily monotonic.
This has been observed numerically in \cite{beschle2022stability,elliott2015evolving,elliott2024fully}, where the Ginzburg-Landau functional appears to converge to a periodic function on domains with periodic evolution --- to the authors knowledge there are no analytic results on this phenomena.
The Ginzburg-Landau functional remains useful in the analysis nonetheless.

The Cahn-Hilliard equation originates from the work of Cahn and Hilliard, \cite{cahn1958free}, in modelling phase separation in a binary alloy.
The Cahn-Hilliard equation originally was applied to metallurgy, for example in the studying phenomenon of spinodal decomposition \cite{cahn1961spinodal,elliott1989cahn}, but also has found application outside of this field, for example in modelling the dynamics of lipid biomembranes \cite{zhiliakov2021experimental}.

This has been extensively studied on a Euclidean domain, see for example \cite{ miranville2019cahn}, and the logarithmic potential has been studied in \cite{barrett1995error,barrett1997finite,copetti1992numerical,cherfils2011cahn,elliott1991generalized} for example.
Recently there has been interest in the equation when posed on a (possibly evolving) surface as motivated by applications such as those in \cite{eilks2008numerical,zhiliakov2021experimental}.
We refer to \cite{abels2024diffuse,abels2024diffuse2,caetano2021cahn, caetano2023regularization, elliott2024navier} for recent results on the analysis.
Likewise, from a numerical perspective we refer the reader to \cite{beschle2022stability, du2011finite, eilks2008numerical,elliott2015evolving,elliott2024fully, li2018direct}.

The breakdown of this paper is as follows.
In Section \ref{section: logch prelim} we introduce some preliminary material which will be necessary for our ESFEM schemes.
Then in Section \ref{section: logch variational form} we introduce the weak formulation of the Cahn-Hilliard equation and prove an error bound for a related regularisation.
In Section \ref{section: logch semi discrete} and Section \ref{section: logch fully discrete} we introduce, and analyse, a semi-discrete ESFEM scheme, and a fully discrete ESFEM scheme.
We provide some numerical results in Section \ref{section: logch numerics}.
We have focussed only on linear finite elements in discretising space, and backward Euler time discretisation as it is known that the logarithmic potential limits the regularity properties of the solution.
For higher order (in space and/or time) ESFEM schemes for more regular problems we refer the reader to \cite{elliott2024fully,kovacs2016error,lubich2013backward}.
Finally we include some material on inverse Laplacian operators in Appendix \ref{invlaps}.

\section{Preliminaries}
\label{section: logch prelim}
\subsection{Some geometric analysis}
Throughout we will consider closed, connected, oriented $C^2$ surfaces.
Given such a surface, $\Gamma$, we denote its normal vector by $\boldsymbol{\nu}(x,t)$.
We recall from \cite{deckelnick2005computation} that $\Gamma$ partitions $\mathbb{R}^{3}$ into two regions, an interior region which we denote $\Omega$, and an exterior region $\mathbb{R}^{3} \backslash \bar{\Omega}$.

It is known (see \cite{foote1984regularity, gilbarg1977elliptic}) that for a $C^2$ surface, $\Gamma$, the oriented distance function $d: \mathcal{N}(\Gamma) \rightarrow \mathbb{R}$ is a $C^2$ function, where $\mathcal{N}(\Gamma)$ is some open, tubular neighbourhood of $\Gamma$, $\mathcal{N}(\Gamma):= \{ x \in \mathbb{R}^3 \mid |d(x)| < \epsilon_0 \}$ for some $\epsilon_0$.
Given a point $x \in \mathcal{N}(\Gamma)$, we may express $x$ in Fermi coordinates as
\begin{align}
	\label{normalprojection}
	x = p(x) + d(x) \boldsymbol{\nu} (p(x)),
\end{align}
for some unique $p(x) \in \Gamma$.
We call $p : \mathcal{N}(\Gamma) \rightarrow \Gamma$ the normal projection operator.
This will be used later in the triangulation of the surface as a way of relating functions on a triangulated surface to the exact surface.

\begin{definition}
	Let $f : \Gamma \rightarrow \mathbb{R}$ be such that we have a differentiable extension of $f$, say $\tilde{f}$, defined on an open neighbourhood of $\Gamma$.
	We define the tangential gradient of $f$ at $x \in \Gamma$ to be
	$$ \gradg f(x) = \nabla \tilde{f}(x) - \left( \nabla \tilde{f}(x) \cdot \boldsymbol{\nu}(x) \right) \boldsymbol{\nu}(x) .$$
	We may express this componentwise as
	$$ \gradg f = \left( \underline{D}_1 f,\underline{D}_2 f,\underline{D}_3 f \right).$$
	It can be shown that this expression is independent of the choice of extension, $\tilde{f}$.
	
	We then define the Laplace-Beltrami operator of $f$to be
	$$ \Delta_{\Gamma} f = \gradg \cdot \gradg f = \sum_{i=1}^{3} \underline{D}_i \underline{D}_i f.$$
\end{definition}

Given these definitions, we define the mean curvature, $H$, of $\Gamma$ to be the tangential divergence of $\boldsymbol{\nu}$, that is for $x \in \Gamma$, $H(x) := \gradg \cdot \boldsymbol{\nu}(x) = \sum_{i=1}^{3} \underline{D}_i \nu_i (x)$.

\subsubsection{Sobolev spaces}

\begin{definition}
	For $p \in [1, \infty]$, the Sobolev space $H^{1,p}(\Gamma)$ is then defined by
	$$ H^{1,p}(\Gamma) := \{ f \in L^p(\Gamma) \mid \underline{D}_i f \in L^p(\Gamma), \ i=1,...,n+1 \},$$
    where $\underline{D}_i$ is understood in the weak sense.
	Higher order spaces ($k \in \mathbb{N}$) are defined recursively by
	$$ H^{k,p}(\Gamma) := \{ f \in H^{k-1,p}(\Gamma) \mid \underline{D}_i f \in H^{k-1,p}(\Gamma), \ i = 1,...,n+1  \}, $$
	where $H^{0,p}(\Gamma) := L^p(\Gamma)$.
	These Sobolev spaces are known to be Banach spaces when equipped with norm,
	$$
	\| f \|_{H^{k,p}(\Gamma)} := \begin{cases}
		\left( \sum_{|\alpha| = 0}^{k}  \| \underline{D}^{\alpha} f \|_{L^p(\Gamma)}^p  \right)^{\frac{1}{p}}, & p \in [1, \infty),\\
		\max_{|\alpha|=1,...,k} \|\underline{D}^\alpha f \|_{L^\infty(\Gamma)}, & p = \infty,
	\end{cases}$$
	where we consider all weak derivatives of order $|\alpha|$.
	We use shorthand notation, $H^{k}(\Gamma) := H^{k,2}(\Gamma)$, for the case $p = 2$.
\end{definition}

Next we introduce some notation which will be used throughout.
\begin{notation}
	For a $\mathcal{H}^2 -$measurable set, $X \subset \mathbb{R}^{3}$, we denote the $\mathcal{H}^2$ measure of $X$ by
	$$ |X| := \mathcal{H}^2(X).$$
	
	For a function $f \in L^1(X)$ we denote the mean value of $f$ on $X$ by
	$$ \mval{f}{X} := \frac{1}{|X|} \int_X f .$$
\end{notation}

We refer the reader to \cite{aubin2012nonlinear,hebey2000nonlinear} for further details on Sobolev spaces defined on Riemannian manifolds.

\subsection{Evolving surfaces}
\begin{definition}[$C^{2}$ evolving surface]
	Let $\Gamma_0 \subset \mathbb{R}^3$ be a closed, connected, orientable $C^{2}$ surface and let $\Phi :\Gamma_0 \times [0,T] \rightarrow \mathbb{R}^3$ be a smooth map such that
	\begin{enumerate}
		\item For all $t \in [0,T]$,
		$$\Phi(\cdot,t): \Gamma_0 \rightarrow \Phi(\Gamma_0,t) =: \Gamma(t)$$
		is a $C^{2}$ diffeomorphism.
		\item $\Phi(\cdot,0) = \mathrm{id}_{\Gamma_0}$.
	\end{enumerate}
	Then we call the family $(\Gamma(t))_{t \in [0,T]}$ a $C^{2}$ evolving surface.
\end{definition}
It follows that $\Gamma(t)$ is closed, connected and orientable for all $t$.
We define the spacetime surface to be the set
$$\mathcal{S}_T = \bigcup_{t \in [0,T]} \Gamma(t) \times \{ t \},$$
and given $(x,t) \in \mathcal{S}_T$ we denote the unit normal to $\Gamma(t)$  by $\boldsymbol{\nu}(x,t)$.
We will assume throughout that there is a velocity field $V \in C^1([0,T];C^{2}(\mathbb{R}^3;\mathbb{R}^3))$ such that for $t \in [0,T]$ and $x \in \Gamma_0$,
\begin{align*}
	\frac{d}{dt} \Phi(x,t) &= V(\Phi(x,t),t),\\
	\Phi(x,0) &= x.
\end{align*}
By compactness of $\mathcal{S}_T$, and assumed smoothness of $V$, there is a constant $C_V$ independent of $t$ such that
\begin{align*}
	\|V(t)\|_{C^{2}(\Gamma(t))} \leq C_V ,
\end{align*}
for all $t \in [0,T]$.

\subsubsection{Time-dependent Lebesgue/Bochner spaces}

Next we introduce a way of relating functions on the evolving surface back to the initial surface, which will be necessary for defining the evolving function spaces.
Let $t \in [0,T]$, $\eta \in H^{m,p}(\Gamma_0)$ and $\zeta \in H^{m,p}(\Gamma(t))$ for some $m = 0,1,2$, and $p \in [1,\infty]$.
We define the pushforward of $\eta$ by
$$ \Phi_t \eta = \eta(\Phi(\cdot,t)) \in H^{m,p}(\Gamma(t)),$$
and the pullback of $\zeta$ by
$$ \Phi_{-t} \zeta = \zeta(\Phi(\cdot,t)^{-1}) \in H^{m,p}(\Gamma_0).$$
Under our assumptions on the smoothness of $\Gamma(t)$ it can be shown that the pairs $(H^{m,p}(\Gamma(t)), \Phi_t)$ are compatible in the sense of \cite{alphonse2023function,alphonse2015abstract} for $m = 0,1,2$ and $p \in [1,\infty]$.
Compatibility of these spaces allows to one obtain Sobolev inequalities on $\Gamma(t)$ independent of $t$.

With these definitions, we can define time-dependent Bochner spaces.

\begin{definition}
	In the following we let $X(t)$ denote a Banach space dependent on $t$, for instance $H^{m,p}(\Gamma(t))$.
	The space $L^2_X$ consists of (equivalence classes of) functions
	\begin{gather*}
		u:[0,T] \rightarrow \bigcup_{t \in [0,T]} X(t) \times \{ t \},\\
		t \mapsto (\bar{u}(t), t),
	\end{gather*}
	such that $ \Phi_{-(\cdot)} \bar{u}(\cdot) \in L^2(0,T;X(0))$.
	We identify $u$ with $\bar{u}$.
    This space is equipped with a norm
    \[ \|u\|_{L^2_X} = \left( \int_0^T \|u(t)\|^2_{X(t)} \right)^{\frac{1}{2}}. \]
	If the family $X(t)$ are in fact Hilbert spaces then this norm is induced by the inner product
	$$ (u,v)_{L^2_X} = \int_0^T (u(t),v(t))_{X(t)},$$
	for $u,v \in L^2_X$.
	In this case, as justified in \cite{alphonse2015abstract}, we make the identification $(L^2_{X})^* \cong L^2_{X^*}$, and for $X = H^1$ we write $L^2_{H^{-1}} := (L^2_{H^1})^*$.
	
	One can similarly define $L^p_X$ for $p \in [1, \infty]$, which is equipped with a norm
	\[ \|u\|_{L^p_X} := \begin{cases}
		\left( \int_0^T \|u(t)\|^p_{X(t)} \right)^{\frac{1}{p}}, & p \in [1,\infty),\\
		\esssup{t \in [0,T]} \|u(t)\|_{X(t)}, & p = \infty.
	\end{cases}
	\]
	We refer the reader to \cite{alphonse2023function} for further details.
\end{definition}

\begin{definition}[Strong material derivative]
	Let $u : \mathcal{S}_T \rightarrow \mathbb{R}$ be sufficiently smooth, then we define the strong material time derivative as
	\[\matdev u = \Phi_t \left( \frac{d}{dt} \Phi_{-t} u \right).\]
\end{definition}
As in the stationary setting, this can be generalised to define a weak material derivative.

\begin{definition}[Weak material derivative]
	Let $u \in L^2_{H^1}$.
	A function $v \in L^2_{H^{-1}}$ is said to be the weak material time derivative of $u$ if for all $\eta \in \mathcal{D}_{H^1}(0,T)$ we have
	$$ \int_0^T \langle v(t), \eta(t) \rangle_{H^{-1}(\Gamma(t)) \times H^1(\Gamma(t))} = - \int_0^T \left(u(t), \matdev \eta(t) \right)_{L^2(\Gamma(t))} - \int_0^T \int_{\Gamma(t)} u(t) \eta(t) \gradg \cdot V(t),$$
	where
	$$ \mathcal{D}_{H^1}(0,T) := \left\{ u \in L^2_{H^1} \mid \Phi_{-t} u(t) \in C^{\infty}_c(0,T;H^1(\Gamma_0)) \right\}.$$
	We abuse notation and write $v = \matdev u$.
\end{definition}

We introduce shorthand notation for a function space of weakly differentiable functions to be
\[ H^1_{H^{-1}} := \{ \eta \in L^2_{H^{-1}} \mid \matdev \eta \in L^2_{H^{-1}} \}, \]
and more generally we may consider the space
\[ H^1_{H^{k}} := \{ \eta \in L^2_{H^{k}} \mid \matdev \eta \in L^2_{H^{k}} \}, \]
for $k \geq 0$.
Clearly if $u \in L^2_{H^1}$ has a strong material time derivative it has a weak material time derivative, and the two coincide.

We now state a transport theorem for quantities defined on an evolving surface.
Firstly, we define the following notation for bilinear forms to be used throughout:
\begin{align*}
	m_*(t;\hat{\eta}, \zeta) &:= \langle \hat{\eta}, \zeta \rangle_{H^{-1}(\Gamma(t)) \times H^1(\Gamma(t))},\\
	m(t;\eta,\zeta) &:=  \int_{\Gamma(t)} \eta \zeta, \\
	g(t;\eta,\zeta)  &:=  \int_{\Gamma(t)} \eta \zeta \gradg \cdot V(t),\\
	a(t;\eta,\zeta) &:= \int_{\Gamma(t)} \gradg \eta \cdot \gradg \zeta,
\end{align*}
where $\eta, \zeta \in H^1(\Gamma(t))$, $\hat{\eta} \in H^{-1}(\Gamma(t))$.
The argument in $t$ will often be omitted, as above.
For weakly differentiable functions we have the following result.
\begin{proposition}[{\cite[Lemma 5.2]{dziuk2013finite}}]
	\label{transport2}
	Let $\eta, \zeta \in L^2_{H^1}\cap H^1_{H^{-1}}$.
	Then $t \mapsto m(\eta(t), \zeta(t))$ is absolutely continuous and such that
	$$ \frac{d}{dt} m(\eta, \zeta) = m_*(\matdev \eta, \zeta) + m_*(\matdev \zeta,\eta) + g(\eta, \zeta).$$
	Moreover, if $\eta, \zeta \in H^1_{H^1}$ then $t \mapsto a_S(\eta(t), \zeta(t))$ is absolutely continuous and such that
	$$ \frac{d}{dt} a(\eta, \zeta) = a(\matdev \eta, \zeta) + a(\eta, \matdev \zeta) + b(\eta, \zeta).$$
	where
	\[b(\eta, \zeta) := \int_{\Gamma(t)} \mathbf{B}(V) \gradg \eta \cdot \gradg \zeta ,\]
	and $\mathbf{B}(V)$ is a tensor given by
	\[\mathbf{B}(V) = \left( (\gradg \cdot V)\mathrm{id} - (\gradg V + (\gradg V)^T) \right).\]
\end{proposition}

\subsection{Triangulated surfaces}
In this subsection we briefly discuss discretisation of (evolving) surfaces.
Much of this presentation is the same as in \cite{elliott2024fully}.
\subsubsection{Construction and lifts}
Let $\Gamma \subset \mathbb{R}^3$ be a closed, oriented $C^2$ surface.
We introduce a discretised version of this surface, denoted $\Gamma_h$, which we call an triangulated (or interpolated) surface.

\begin{definition}
	We let $(x_i)_{i=1,...,N_h} \subset \Gamma$ be a collection of nodes used to define a set of triangles $\mathcal{T}_h$.
	The triangulated surface, $\Gamma_h$, is defined by an admissible subdivision of triangles, $\mathcal{T}_h$, such that
	$$\bigcup_{K \in \mathcal{T}_h} K = \Gamma_h.$$
	If $K_1, K_2 \in \mathcal{T}_h$ are distinct, then we have $K_1^{\circ} \cap K_2^{\circ} = \emptyset$, and if $ \bar{K}_1 \cap \bar{K}_2 \neq \emptyset$ then this intersection is either a node of the triangulation, or a line segment connecting two adjacent nodes.
	
	For $K \in \mathcal{T}_h$ we define following quantities
	$$h_K := \mathrm{diam}(K), \quad \rho_K := \sup\{ \mathrm{diam}(B) \mid B \text{ is a } 2-\text{dimensional ball contained in } K \}.$$
	We assume that the subdivision $\mathcal{T}_h$ is quasi-uniform, meaning there exists $\rho > 0$ such that for all $h \in (0,h_0)$
	$$ \min \{ \rho_K \mid K \in \mathcal{T}_h \} \geq \rho \max_{K \in \mathcal{T}_h} h_{K} .$$
\end{definition}

Throughout this paper we work with linear Lagrange finite elements --- that is our degrees of freedom are given by the point evaluations at the nodes $(x_i)_{i=1,...,N_h}$.
We will denote the set of shape functions as
$$ S_h := \left\{ \phi_h \in C(\Gamma_h) \mid \phi_h |_{K} \text{ is affine linear for } K \in \mathcal{T}_h \right\}.$$
The normal $\boldsymbol{\nu}_h$ is defined piecewise on each element of $\Gamma_h$ which gives rise to a discrete tangential gradient, $\gradgh$, defined element-wise on $\Gamma_h$.

Next we relate functions on $\Gamma_h$ and $\Gamma$ by defining lifts.
We will assume that our triangulated surface $\Gamma_h$ is such that $\Gamma_h \subset \mathcal{N}(\Gamma)$ for $\mathcal{N}(\Gamma)$ a tubular neighbourhood as described above.
This is possible in practice by considering a sufficiently fine triangulation.
This allows us to define lifts of functions.
\begin{definition}
	For a function $\eta_h : \Gamma_h \rightarrow \mathbb{R}$ we implicitly define the lift operation on $\eta_h$ by
	$$\eta_h^{\ell}(p(x)) := \eta_h(x),$$
	where $p$ is the normal projection operator \eqref{normalprojection}.
	
	Similarly, for $\eta : \Gamma \rightarrow \mathbb{R}$ we define the inverse lift by
	$$ \eta^{-\ell}(x) = \eta(p(x)) . $$
	
\end{definition}

In \cite{elliott2021unified} the following result concerning lifts of functions is proven.
\begin{lemma}
	There exists constants $C_1, C_2$, independent of $h$ such that for $\eta_h \in H^1(\Gamma_h)$
	\begin{gather}
		C_1 \| \eta_h^{\ell} \|_{L^2(\Gamma)} \leq \| \eta_h \|_{L^2(\Gamma_h)} \leq C_2 \| \eta_h^{\ell} \|_{L^2(\Gamma)}, \label{lift1}\\
		C_1 \| \gradg \eta_h^{\ell} \|_{L^2(\Gamma)} \leq \| \gradgh \eta_h \|_{L^2(\Gamma_h)} \leq C_2 \| \gradg \eta_h^{\ell} \|_{L^2(\Gamma)}. \label{lift2}
	\end{gather}
\end{lemma}

This shows that there exist constants $C_1, C_2$ independent of $h$ such that
$$ C_1 |\Gamma| \leq |\Gamma_h| \leq C_2 |\Gamma| .$$ 
Another useful consequence of the stability of the lift is that it allows one to obtain a Poincar\'e inequality independent of $h$.

\begin{definition}
	We say that our triangulation is exact if the lifted triangles $K^{\ell} := \{ x^{\ell} \mid x \in \Gamma_h \}$ form a conforming subdivision of $\Gamma$.
\end{definition}

\subsubsection{Evolving triangulated surfaces}
Given an evolving surface, $(\Gamma(t))_{t \in [0,T]}$ we construct an evolving triangulated surface as follows.
Firstly, we construct an admissible triangulation,  $\mathcal{T}_h(0)$, of $\Gamma_0$, with nodes $(x_{i,0})_{i=1,...,N_h}$ as above.
We denote this triangulated surface as $\Gamma_h(0)$.
The nodes of $\Gamma_h(0)$ then evolve in time according to the ODE,
$$ \frac{d}{dt} x_i(t) = V(x_i(t),t), \qquad x_i(0) = x_{i,0},$$
where $V$ is the velocity field associated with the evolution of $\Gamma(t)$.
This induces a triangulation $\mathcal{T}_h(t)$, where $K(0) \in \mathcal{T}_h(0)$ gives rise to a triangle $ K(t) \in \mathcal{T}_h(t)$ by evolving the nodes as above.
The triangulated surface $\Gamma_h(t)$ is then defined as
$$ \Gamma_h(t) := \bigcup_{K(t) \in \mathcal{T}_h(t)} K(t),$$
which will be admissible by construction of $\Gamma_h(0)$.
Here the $h$ parameter is defined to be
$$ h:= \sup_{t \in [0,T]} \max_{K(t) \in \mathcal{T}_h(t)} h_{K(t)}. $$
We denote the discrete spacetime surface as
$$ \mathcal{S}_{h,T} := \bigcup_{t \in [0,T]} \Gamma_h(t) \times \{t\}.$$

We note that as the domain is evolving, the set of basis functions also evolves in time.
As such, we denote the set of basis functions at time $t$ to be
$$ S_h(t) = \left\{ \phi_h \in C(\Gamma_h(t)) \mid \phi_h |_{K(t)} \text{ is affine linear },  K(t) \in \mathcal{T}_h(t) \right\}.$$
This definition allows one to characterise the velocity of the surface $\Gamma_h(t)$, as an arbitrary point $x(t) \in \Gamma_h(t)$ will evolve according to the discrete velocity, $V_h$, given by
$$ \frac{d}{dt} x(t) = V_h(x(t),t) := \sum_{i=1}^{N_h} \dot{x}_i(t) \phi_i(x(t),t) = \sum_{i=1}^{N_h} V(x_i(t),t) \phi_i(x(t),t),$$
where $\phi_i(t)$ is the `$i$'th nodal basis function of $\Gamma_h(t)$.
From this we observe that $V_h$ is the Lagrange interpolant of $V$.

The evolution of $\Gamma(t)$ induces a discrete flow map $\Phi^h: \Gamma_{h}(0) \rightarrow \Gamma_h(t)$ (see \cite{elliott2021unified}) which one can use to define discrete pushforwards/pullbacks $\Phi_t^h \eta_h, \Phi_{-t}^h \eta_h$, for which one defines a strong material time derivative by
\[ \matdev_h \eta_h = \Phi_t^h \left( \frac{d}{dt} \Phi_{-t}^h \eta_h \right). \]
One can similarly define a weak discrete material derivative in the standard way.

An important consequence of this is the transport property of basis functions --- if $\phi_i : \mathcal{S}_{h,T} \rightarrow \mathbb{R}$ is one of the nodal basis functions of $\Gamma_h(t)$, then
\[\matdev_h \phi_i = 0.\]
This is an important property, which is exploited in implementing evolving surface finite element schemes as it eliminates any velocity terms in the fully discrete formulation.
\begin{definition}
	The evolving triangulated surface, $\Gamma_h(t)$, is said to be uniformly quasi-uniform if there exists $\rho > 0$ such that for all $t \in [0,T]$, and $h \in (0, h_0)$ we have
	$$ \min\{ \rho_{K(t)} \mid K(t) \in \mathcal{T}_h(t) \} \geq \rho h .$$
\end{definition}

A useful property of uniformly quasi-uniform meshes is the following discrete Sobolev inequality
\begin{align}
	\|\phi_h\|_{L^\infty(\Gamma(t))} \leq C \log\left(\frac{1}{h}\right)^\frac{1}{2} \|\phi_h \|_{H^1(\Gamma(t))}, \label{discrete sobolev}
\end{align}
where $\phi_h \in S_h(t)$, and $C$ is independent of $h,t$.
We do not prove this for an evolving surface, but we do refer to \cite{thomee2007galerkin} Lemma 6.4 for a proof in the stationary, Euclidean case.

\begin{definition}
	We say that the triangulation, $\mathcal{T}_h$, of $\Gamma_h$ is acute if for all $K \in \mathcal{T}_h$ the angles of $K$ are less than or equal to $\frac{\pi}{2}$.
	We say that the triangulation for an evolving surface is evolving acute if $\mathcal{T}_h(t)$ is acute for all $t \in [0,T]$.
	
\end{definition}
We say the evolving triangulation is exact if for all $t \in [0,T]$
$$ \bigcup_{K(t) \in \mathcal{T}_h(t)} K^{\ell}(t) = \Gamma(t).$$
In our analysis we make the following assumptions:
\begin{assumption}[Evolving triangulation]
~~

\begin{enumerate}
\item
For all $t \in [0,T]$ the  evolving triangulated surface, $\Gamma_h(t)$  is uniformly quasi-uniform.
\item
For all $t \in [0,T]$ the  evolving triangulated surface, $\Gamma_h(t)$  is exact.
\item
For all $t \in [0,T]$ the  evolving triangulated surface, $\Gamma_h(t)$ is acute.
\item
For each $t \in [0,T]$ one has $\Gamma_h(t) \subset \mathcal{N}(\Gamma(t))$ so that we may define the lift at all times $t \in [0,T]$.
\end{enumerate}
\end{assumption}
\begin{remark}
The condition that the mesh remains acute for all time is somewhat problematic as it is known that an initially acute mesh may lose the acute property after evolution in time.
However, if the initial mesh is strictly acute then for some small time interval (dependent on the velocity field $V$) the mesh will remain strictly acute.
More generally one may want to use a remeshing procedure, such as the algorithm in \cite{elliott2016algorithms} --- where an initially acute mesh will remain acute under evolution as harmonic map flow yields conformal maps.
It is likely that one can extend the analysis in this paper to meshes which satisfy weaker conditions on the mesh, for example the Xu-Zikatanov condition \cite{xu1999monotone}, but we shall not consider this here.
\end{remark}
We now state a discrete analogue of the transport theorem, Proposition \ref{transport2}.
Here we denote the (time-dependent) bilinear forms by
\begin{align*}
	m_h(t;\eta_h, \zeta_h) &:= \int_{\Gamma_h(t)} \eta_h \zeta_h,\\
	a_h(t;\eta_h, \zeta_h) &:= \int_{\Gamma_h(t)} \gradgh \eta_h \cdot \gradgh \zeta_h,\\
	g_h(t;\eta_h, \zeta_h) &:= \int_{\Gamma_h(t)} \eta_h \zeta_h \gradgh \cdot V_h,
\end{align*}
where we typically omit the $t$ argument.
We then have a discrete transport theorem for these bilinear forms.
\begin{proposition}
	\label{transport3}
	Let $\eta_h, \zeta_h \in S_h(t)$ be such that $\matdev_h\eta_h, \matdev_h\zeta_h \in S_h(t)$ exist.
	Then we have
	\begin{gather*}
		\frac{d}{dt} m_h(\eta_h, \zeta_h) = m_h(\matdev_h \eta_h, \zeta_h) + m_h(\eta_h, \matdev_h \zeta_h) + g_h(\eta_h, \zeta_h),\\
		\frac{d}{dt} a_h(\eta_h, \zeta_h) = a_h(\matdev_h \eta_h, \zeta_h) + a_h(\eta_h, \matdev_h \zeta_h) + b_h(\eta_h, \zeta_h).
	\end{gather*}
	Here
	$$ b_h(\eta_h, \zeta_h) = \int_{\Gamma_h(t)} \mathbf{B}_h(V_h) \gradgh \eta_h \cdot \gradgh \zeta_h , $$
	where
	$$ \mathbf{B}_h(V_h) = \left( (\gradgh \cdot V_h)\mathrm{id} - (\gradgh V_h + (\gradgh V_h)^T) \right).$$
\end{proposition}

\subsubsection{Geometric perturbation estimates and the Ritz projection}
Here we state several results which are crucial to the numerical analysis of surface PDEs.
Firstly, as noted in \cite{elliott2021unified}, one can obtain another material derivative by (inverse) lifting a function onto $\Gamma_h(t)$, differentiating, and lifting back onto $\Gamma(t)$.
\begin{definition}
	For a sufficiently smooth function, $\eta : \mathcal{S}_T \rightarrow \mathbf{R}$, we define the (strong) lifted material derivative as
	\[\matdev_\ell \eta = \left( \matdev_h \eta^{-\ell} \right)^\ell.\]
\end{definition}
This can also be expressed as
\[\matdev_\ell \eta = \Phi_t^\ell \left( \frac{d}{dt} \Phi_{-t}^\ell \eta \right),\]
where $\Phi_{t}^\ell, \Phi_{-t}^\ell$ are the pushforward/pullback respectively associated to the map $\Phi^\ell : \Gamma_0 \rightarrow \Gamma(t)$ defined by
\[\Phi^\ell(x,t) = \Phi^h(x^{-\ell},t)^\ell.\]
This is discussed in detail in \cite{elliott2021unified}.
From this we obtain an alternate version of Proposition \ref{transport2}.
\begin{proposition}
	\label{transport4}
	Let $\eta, \zeta \in C^1_{L^2}$, then
	\begin{align*}
		\frac{d}{dt} m(t; \eta, \zeta) = m(t; \matdev_\ell \eta, \zeta) + m(t; \eta, \matdev_\ell \zeta) + g_\ell (t; \eta, \zeta),
	\end{align*}
	where
	$$g_\ell (t; \eta, \zeta) = \int_{\Gamma(t)} \eta \zeta (\gradg \cdot V_h^\ell).$$
	Similarly, if $\eta, \zeta \in C^1_{H^1}$, then
	\begin{align*}
		\frac{d}{dt} a(t; \eta, \zeta) = a(t; \matdev_\ell, \zeta) + a(t; \eta, \matdev_\ell \zeta) + b_\ell (t; \eta, \zeta),
	\end{align*}
	where
	$$ b_\ell(\eta, \zeta) = \int_{\Gamma(t)} \mathbf{B}(V_h^\ell) \gradg \eta \cdot \gradg \zeta , $$
	and $\mathbf{B}(V_h^\ell)$ is as in Proposition \ref{transport2}.
\end{proposition}
This allows one to define a weak lifted material derivative in the usual way.
We can then related $\matdev_\ell$ and $\matdev$ through the following result.
\begin{lemma}[{\cite[Lemma 9.25]{elliott2021unified}}]
	Let $\eta \in H^1_{H^1}$.
	Then we have
	\begin{align}
		\|\matdev \eta - \matdev_\ell \eta\|_{L^2(\Gamma(t))} \leq C h^{2} \|\eta\|_{H^1(\Gamma(t))},\label{derivativedifference1}
	\end{align}
	and if we have further that $\eta \in H^1_{H^2}$, then
	\begin{align}
		\|\gradg(\matdev \eta - \matdev_\ell \eta)\|_{L^2(\Gamma(t))} \leq C h \|\eta\|_{H^2(\Gamma(t))}.\label{derivativedifference2}
	\end{align}
\end{lemma}
We now state some results which allow us to compare bilinear forms on $\Gamma(t)$ and $\Gamma_h(t)$.
The following results are proven in \cite{elliott2021unified}.
\begin{lemma}
	Let $\eta_h, \zeta_h \in H^1(\Gamma_h(t))$, and $h$ be sufficiently small.
	Then there exists a constant $C > 0$, independent of $t,h$, such that
	\begin{align}
		\left| m \left(t; \eta_h^{\ell}, \zeta_h^{\ell} \right) - m_h \left(t; \eta_h, \zeta_h \right) \right| &\leq C h^{2} \| \eta_h \|_{L^2(\Gamma_h(t))} \| \zeta_h \|_{L^2(\Gamma_h(t))},\label{perturb1}\\
		\left| a \left(t; \eta_h^{\ell}, \zeta_h^{\ell} \right) - a_h \left(t; \eta_h, \zeta_h \right) \right| &\leq C h^{2} \| \gradgh \eta_h \|_{L^2(\Gamma_h(t))} \| \gradgh \zeta_h \|_{L^2(\Gamma_h(t))},\label{perturb2}\\
		\left| g_\ell \left(t; \eta_h^{\ell}, \zeta_h^{\ell} \right) - g_h \left(t; \eta_h, \zeta_h \right) \right| &\leq C h^{2} \| \eta_h \|_{L^2(\Gamma_h(t))} \| \zeta_h \|_{L^2(\Gamma_h(t))},\label{perturb3}\\
		\left| b_\ell \left(t; \eta_h^{\ell}, \zeta_h^{\ell} \right) - b_h \left(t; \eta_h, \zeta_h \right) \right| &\leq C h^{2} \| \gradgh \eta_h \|_{L^2(\Gamma_h(t))} \| \gradgh \zeta_h \|_{L^2(\Gamma_h(t))},\label{perturb4},\\
		\left| g \left(t; \eta_h^{\ell}, \zeta_h^{\ell} \right) - g_\ell \left(t; \eta_h^{\ell}, \zeta_h^{\ell} \right) \right| &\leq C h \| \eta_h \|_{H^1(\Gamma_h(t))} \| \zeta_h \|_{H^1(\Gamma_h(t))},\label{perturb5}\\
		\left| b \left(t; \eta_h^{\ell}, \zeta_h^{\ell} \right) - b_\ell \left(t; \eta_h^{\ell}, \zeta_h^{\ell} \right) \right| &\leq C h \| \eta_h \|_{H^1(\Gamma_h(t))} \| \zeta_h \|_{H^1(\Gamma_h(t))},\label{perturb6}.
	\end{align}
\end{lemma}

Next we introduce a projection onto the shape functions which is useful in the error analysis for surface finite elements.
\begin{definition}
	For $z \in H^1(\Gamma(t))$ we define the Ritz projection\footnote{Some authors define the Ritz projection differently, see \cite[Remark 3.4]{elliott2015evolving}.}, $\Pi_h z \in S_h(t)$, to be the unique solution of
	\begin{align}
            \begin{aligned}
                a_h(\Pi_h z , \phi_h) &= a(z , \phi_h^{\ell}),\\
        \int_{\Gamma_h(t)} \Pi_h z &= \int_{\Gamma(t)} z,
            \end{aligned}
        \label{ritz}
	\end{align}
	for all $\phi_h \in S_h(t)$.
	We denote the lift of the Ritz projection by
	$\pi_h z = (\Pi_h z)^{\ell}.$
\end{definition}
One has the following bounds for the Ritz projection, for which we refer the reader to \cite{elliott2015evolving,elliott2021unified}

\begin{lemma}
	For $z \in H^1(\Gamma(t))$ we have the following,
	\begin{gather}
		\label{ritz1}
		\| \pi_h z \|_{H^1(\Gamma(t))} \leq C \| z \|_{H^1(\Gamma(t))},\\
		\label{ritz2}
		\| \pi_h z - z \|_{L^2(\Gamma(t))} \leq C h \|z\|_{H^1(\Gamma(t))}.
	\end{gather}
	Moreover, if $z \in H^{2}(\Gamma(t))$ then
	\begin{gather}
		\|\Pi_h z\|_{L^\infty(\Gamma_h(t))} = \|\pi_h z\|_{L^\infty(\Gamma(t))} \leq C \|z\|_{H^2(\Gamma(t))}, \label{ritz4}\\
		\label{ritz3}
		\| \pi_h z - z \|_{L^2(\Gamma(t))} + h \| \gradg (\pi_h z - z) \|_{L^2(\Gamma(t))} \leq C h^{2} \|z \|_{H^{2}(\Gamma(t))}.
	\end{gather}
\end{lemma}

We also have the following lemma regarding the time derivative of the Ritz projection.

\begin{lemma}
	For $z : \mathcal{S}_T \rightarrow \mathbb{R}$ with $z, \matdev{z} \in H^2(\Gamma(t))$ then $\matdev_h \Pi_h z \in S_h(t)$ exists and
	\begin{equation}
		\label{ritzddtnorm}
		\| \matdev_h \Pi_h z\|_{H^1(\Gamma_h(t))} \leq C\left( \|z\|_{H^2(\Gamma(t))} + \|\matdev z\|_{H^2(\Gamma(t))} \right),
	\end{equation}
	and
	\begin{equation}
		\label{ritzddt}
		\left \|  \matdev_\ell(\pi_h z - z) \right\|_{L^2(\Gamma(t))} + h \left \| \gradg  \matdev_\ell(\pi_h z - z)\right\|_{L^2(\Gamma(t))} \leq C h^{2} \left( \|z\|_{H^{2}(\Gamma(t))} + \left\| \matdev z \right\|_{H^{2}(\Gamma(t))} \right),
	\end{equation}
	where $C$ is a constant independent of $h,t$.
\end{lemma}

\subsubsection{Interpolation and numerical integration}
\begin{definition}
	We denote the linear nodal basis functions, on $\Gamma_{h}(t)$, by $\phi_i$.
	Then the (Lagrange) interpolation operator $I_h : C(\Gamma_h(t)) \rightarrow S_h(t)$ is defined by
	$$ I_h z_h (t;x) = \sum_{i=1}^{N_h} z_h(t;x_i) \phi_i(t;x) ,$$
	where $z_h \in C(\Gamma_h(t))$.
	Similarly, we define the lifted interpolation operator, $I_h^{\ell} : C(\Gamma(t)) \rightarrow S_h^{\ell}(t)$, by
	$$ I_h^{\ell} z (t;x) = \sum_{i=1}^{N_h} z(t;x_i) \phi_i^{\ell}(t;x), $$
	where $z \in C(\Gamma(t))$.
	We will omit the $t$ argument as usual.
\end{definition}

In \cite{elliott2021unified} the following result is shown for the lifted interpolation operator, where we note that as $H^2(\Gamma(t)) \hookrightarrow C^0(\Gamma(t))$ the interpolant of a $H^2$ function is well-defined.
\begin{lemma}
	The lifted interpolation operator defined above satisfies, for $z \in H^2(\Gamma(t))$
	\begin{align}
		\label{interpolation}
		\left\| z - I_h^{\ell}z \right\|_{L^2(\Gamma(t))} + h \left\| \gradg(z - I_h^{\ell}z) \right\|_{L^2(\Gamma(t))} \leq C h^2 \| z \|_{H^2(\Gamma(t))},
	\end{align}
	and more generally for $z \in H^{k,p}(\Gamma(t))\cap C(\Gamma(t))$
	\begin{align}
		\left\| z - I_h^{\ell}z \right\|_{L^q(\Gamma(t))} + h\left\|\gradg (z - I_h^{\ell}z) \right\|_{L^q(\Gamma(t))} \leq C h^{k + \frac{2}{q} - \frac{2}{p}} \|z\|_{H^{k,p}(\Gamma(t))}, \label{interpolation2}
	\end{align}
	where $C$ is a constant independent of $t, h$, and $k = 1,2$, $p,q \in [1,\infty]$.
\end{lemma}

The interpolation operator also has the following property which will be used throughout.
\begin{lemma}
	For a monotonically increasing function $\lambda \in C^1(\mathbb{R})$, and a function $\phi_h \in S_h(t)$ we have that
	\begin{align}
		\|I_h \lambda(\phi_h) - \lambda(\phi_h)\|_{L^2(\Gamma_h(t))} \leq Ch\|\gradgh I_h \lambda(\phi_h)\|_{L^2(\Gamma_h(t))}, \label{enthalpybound}
	\end{align}
	where $C$ is independent of $t,h$.
\end{lemma}
\begin{proof}
	This proof is largely the same as that of \cite{elliott1987error}.
	To begin, we restrict to a single element $K(t) \in \mathcal{T}_h(t)$, and note that a linear function on $K(t)$ attains its extrema at the nodes.
	Hence there are nodes $x_m(t), x_M(t) \in \partial K(t)$ such that $\phi_h(x_m(t)) \leq \phi_h(x) \leq \phi_h(x_M(t))$ for all $x \in K(t)$.
	By the monotonicity assumption on $\lambda$ we find
	$$ \lambda(\phi_h(x_m(t))) \leq \lambda(\phi_h(x)) \leq \lambda(\phi_h(x_M(t))),$$
	for all $x \in K(t)$.
	From this is it clear that
	$$ \| I_h \lambda(\phi_h) - \lambda(\phi_h) \|_{L^{\infty}(K(t))} \leq \lambda(\phi_h(x_M)) - \lambda(\phi_h(x_m)) = |(x_M - x_m) \cdot \gradgh I_h\lambda(\phi_h)|,$$
	where the equality follows since $I_h \lambda$ is linear on $K(t)$.
	The desired inequality follows from this.
\end{proof}

The use of mass lumping in surface finite elements has previously been considered in \cite{eilks2008numerical,frittelli2017lumped,frittelli2018numerical} and so we do not prove every property of mass lumped finite elements that we use.
It is a well known property of acute triangulations (see for example \cite{ciarlet1973maximum, frittelli2019preserving}) that the stiffness matrix, $A(t) = (a_{ij}(t))_{ij}$, where
$$ a_{ij}(t) = \int_{\Gamma_h(t)} \gradgh \phi_i \cdot \gradgh \phi_j,$$
is such that
$$a_{ii}(t) > 0,\qquad a_{ij}(t) \leq 0 \ (i \neq j). $$
In particular this means that we have the following result.
\begin{lemma}
	Let $\lambda \in C(\mathbb{R})$, with $\lambda(0) = 0$, and $\lambda' \in L^{\infty}(\mathbb{R})$ such that $ 0 \leq \lambda'(s) \leq M_{\lambda} < \infty$ for almost all $s \in \mathbb{R}$.
	Then for an evolving acute triangulation of $\Gamma_h(t)$ one has
	\begin{align}
		\label{acuteineq}
		\| \gradgh I_h \lambda(\phi_h(t)) \|_{L^2(\Gamma_h(t))}^2 \leq M_{\lambda} \int_{\Gamma_h(t)} \gradgh \phi_h \cdot \gradgh I_h \lambda(\phi_h(t)),
	\end{align}
	for all $\phi_h \in S_h(t)$.
\end{lemma}
We refer the reader to \cite{ciavaldini1975analyse,nochetto1991finite} for the proof, noting that the evolving surface changes nothing (provided that one can ensure the triangulation remains acute for $t \in [0,T]$).

We introduce the following shorthand notation
\begin{align*}
	\intm(t;\eta_h, \zeta_h) &= \int_{\Gamma_h(t)} I_h(\eta_h \zeta_h),
\end{align*}
for the lumped mass $L^2$ inner product on $\Gamma_h(t)$.
The bilinear form $\intm(\cdot,\cdot)$ corresponds to the use of numerical integration in a finite element scheme.
In particular, it defines a new inner product on $S_h(t)$, and hence we have a new norm given by
$$ \normh{t}{\eta_h}^2 := \intm (t;\eta_h, \eta_h),$$
where we will omit $t$ as before.
Differentiating this bilinear form in time yields the following.
\begin{lemma}
\label{transport5}
	Let $\eta_h, \zeta_h \in S_h(t)$ be such that $\matdev_h \eta_h, \matdev_h \zeta_h \in S_h(t)$ exist.
    Then
	\[ \frac{d}{dt} \intm(\eta_h, \zeta_h) = \intm(\matdev_h\eta_h, \zeta_h) + \intm(\eta_h, \matdev_h \zeta_h) + g_h(I_h(\eta_h\zeta_h),1).\]
\end{lemma}
\begin{proof}
	By definition we have
	\[ \intm(\eta_h, \zeta_h) = \int_{\Gamma_h(t)} I_h(\eta_h\zeta_h),\]
	and so the transport theorem yields
	\[ \frac{d}{dt} \intm(\eta_h, \zeta_h) = m_h(\matdev_h I_h(\eta_h \zeta_h),1) + g_h(I_h(\eta_h \zeta_h), 1). \]
	Now writing $I_h(\eta_h \zeta_h) = \sum_{i=1}^{N_h} \eta_h(t;x_i(t))\zeta_h(t;x_i(t)) \phi_i(t)$, and using the transport property $\matdev_h \phi_i = 0$, one finds that
	\[ \matdev_h I_h(\eta_h \zeta_h) = I_h(\matdev_h\eta_h \zeta_h) + I_h(\eta_h \matdev_h \zeta_h), \]
	from which the result follows.
\end{proof}
Next we adapt a result of Ciavaldini \cite{ciavaldini1975analyse} to the setting of evolving surface finite elements.

\begin{lemma}
	\label{intmbounds}
	Let $\eta_h, \zeta_h \in S_h(t)$. Then we have
	\begin{align}
		\|\eta_h \|_{L^2(\Gamma_h(t))} \leq & \normh{t}{\eta_h} \leq C \|\eta_h \|_{L^2(\Gamma_h(t))}, \label{intmbound0} \\
		| \intm(\eta_h, \zeta_h) - m_h(\eta_h, \zeta_h)| \leq &C h \| \eta_h \|_{L^2(\Gamma_h(t))} \| \gradgh \zeta_h \|_{L^2(\Gamma_h(t))}, \label{intmbound1}\\
		| \intm(\eta_h, \zeta_h) - m_h(\eta_h, \zeta_h)| \leq &C h^2 \| \gradgh \eta_h \|_{L^2(\Gamma_h(t))} \| \gradgh \zeta_h \|_{L^2(\Gamma_h(t))}, \label{intmbound2}
	\end{align}
	where $C$ is some constant independent of $h, t$.
\end{lemma}
\begin{proof}
	By the piecewise nature of the functions, it is sufficient to restrict to a single $K(t) \in \mathcal{T}_h(t)$, and we denote the nodes of $K(t)$ by $x_1, x_2, x_3$.
	To show the first result, we follow the presentation of \cite{ciavaldini1975analyse}.
	Firstly, we note that
	$$ \normh{t}{\eta_h}^2 = \int_{K(t)} \sum_{i=1}^3 \eta_h(x_i)^2 \phi_i = \frac{|K(t)|}{3} \sum_{i=1}^3 \eta_h(x_i)^2,$$
	where $\phi_i$ is the nodal basis function associated to $x_i$, and we have used a quadrature rule which is exact for a linear function.
	Similarly, by noting that $\eta_h^2$ is piecewise quadratic, we find that
	$$ \| \eta_h \|_{L^2(K(t))}^2 = \int_{K(t)} \left( \sum_{i=1}^3 \eta_h(x_i) \phi_i \right)^2 = \frac{|K(t)|}{6} \left( \sum_{i=1}^3 \eta_h(x_i)^2 + \sum_{1\leq i < j \leq 3} \eta_h(x_i) \eta_h(x_j) \right),$$
	where we have used a quadrature rule which is exact for quadratics.
	One can readily show that
	$$ \sum_{1\leq i < j \leq 3} \eta_h(x_i) \eta_h(x_j) \leq \sum_{i=1}^3 \eta_h(x_i)^2, $$
	and hence it follows that $\| \eta_h \|_{L^2(\Gamma_h(t))} \leq \normh{t}{\eta_h}$.
	To show the second part of \eqref{intmbound0}, we use a similar argument except now we use a bound of the form
	$$2  \sum_{i=1}^3 \eta_h(x_i)^2 \leq M \sum_{i=1}^3 \eta_h(x_i)^2 + M \sum_{1\leq i < j \leq 3} \eta_h(x_i) \eta_h(x_j),$$
	which one can show\footnote{For example one can use the identity $(M-2)(a^2 + b^2 + c^2) + M(ab + ac +bc) = (\frac{M}{2} - 2)(a^2 + b^2 + c^2) + \frac{M}{2}(a+b+c)^2$ for $a,b,c \in \mathbb{R}.$} holds for $M > 4$.\\
	
	To show \eqref{intmbound2}, we firstly recall from from \cite[Theorem 6.13]{elliott2021unified} that
	\[\|\chi - I_K \chi\|_{L^1(K(t))} \leq C h_K^{2} |\chi|_{H^{2,1}(K(t))},\]
	where $|\cdot|_{H^{2,1}(K(t))}$ denotes the $H^{2,1}(K(t))$ seminorm, and $I_K$ is the local Lagrange interpolation operator on $K(t)$, and $C$ is independent of $h,t$.
	This yields
	\begin{align}\label{ciavaldini1}
		\left| \int_{K(t)} \left(\eta_h \zeta_h  -  I_K(\eta_h \zeta_h) \right) \right| \leq \int_{K(t)} \left|\eta_h \zeta_h  -  I_K(\eta_h \zeta_h) \right| \leq C h_K^2 |\eta_h \zeta_h|_{H^{2,1}(K(t))}.
	\end{align}
	On $K(t)$ $\eta_h, \zeta_h$ are in fact smooth, and so the necessary derivatives exist.
	Moreover, as the functions are linear on $K(t)$ we observe that $ \underline{D}_i \underline{D}_j \eta_h = \underline{D}_i \underline{D}_j \zeta_h = 0$ for $i,j = 1,...,3$.
	Computing the $H^{2,1}(K(t))$ seminorm, and using the above observation, we find
	\begin{align*}
		| \eta_h \zeta_h |_{H^{2,1}(K(t))} = \sum_{i=1}^3\sum_{j=1}^3  \|\underline{D}_i \underline{D}_j(\eta_h \zeta_h)\|_{L^1(K(t))} &= \sum_{i=1}^3\sum_{j=1}^3 \int_{K(t)} |\underline{D}_i\eta_h| |\underline{D}_j \zeta_h|\\
		&\leq C \| \gradgh \eta_h \|_{L^2(K(t))} \| \gradgh \zeta_h \|_{L^2(K(t))}.
	\end{align*}
	Combining this with \eqref{ciavaldini1} gives a local version of \eqref{intmbound2}, which one uses to deduce the global result (since our mesh is uniformly quasi-uniform).
	\eqref{intmbound1} then follows from \eqref{intmbound2} and an inverse inequality.
\end{proof}

\begin{remark}
    By arguing along the same lines one can show that
    \begin{align}
        |g_h(\eta_h, \zeta_h) - g_h(I_h(\eta_h \zeta_h),1)| \leq Ch^2 \| \gradgh \eta_h \|_{L^2(\Gamma_h(t))} \| \gradgh \zeta_h \|_{L^2(\Gamma_h(t))}, \label{intmbound3}
    \end{align}
    and hence one can use an inverse inequality to see that
    \[ |g_h(I_h(\eta_h \zeta_h),1)| \leq C\|\eta_h\|_{L^2(\Gamma_h(t))}\|\zeta_h\|_{L^2(\Gamma_h(t))}. \]
    We shall use this bound repeatedly throughout.
\end{remark}

\section{Variational formulation and well-posedness}
\label{section: logch variational form}
The weak formulation of the Cahn-Hilliard equation of \cite{caetano2021cahn,caetano2023regularization} is to find a solution pair $(u,w) \in L^\infty_{H^1} \cap H^1_{H^{-1}} \times L^2_{H^1}$ such that
\begin{gather}
	\label{cheqn1}
	m_*\left(\matdev{u} , \phi \right) + g(u, \phi)+ a (w, \phi) = 0,\\ 
	m(w, \eta) = \varepsilon a(u,\phi) + \frac{\theta}{2 \varepsilon} m (f(u), \phi) - \frac{1}{\varepsilon} m(u,\phi), \label{cheqn2}
\end{gather}
for all $\phi \in H^1(\Gamma(t))$ and a.e. $t \in [0,T]$, where the initial condition $u(0) = u_0$ holds a.e. on $\Gamma_0$.
The well-posedness of this weak formulation has been studied in \cite{caetano2021cahn,caetano2023regularization}.
Due to the logarithmic potential the function $u$ must be such that $|u|<1$ a.e. on $\Gamma(t)$ for almost all $t \in [0,T]$.
This restricts the class of admissible initial data.
One would at least expect that we require the initial data to be such that $|u_0| \leq 1$ almost everywhere on $\Gamma(t)$ for almost all $t \in [0,T]$.
In fact we require a stronger condition --- that $u_0$ is such that
\[m_{u_0}(t) := \frac{1}{|\Gamma(t)|} \left| \int_{\Gamma_0} u_0\right| < 1,\]
for all $t \in [0,T]$.
We refer to \cite[Proposition 5.1]{caetano2021cahn} for the explanation for how this condition arises, and to \cite{caetano2023regularization} for a physical interpretation.
As such, one considers the set of admissible initial conditions to be
\[\mathcal{I}_0 := \left\{ \eta \in H^1(\Gamma_0) \mid |\eta| \leq 1\  \text{a.e. on } \Gamma_0, \ \Ech[\eta;0] < \infty , \ m_{\eta}(t) < 1 \ \forall t \in [0,T] \right\} .\]

In \cite{caetano2021cahn, caetano2023regularization} it is shown that there exists a solution pair solving \eqref{cheqn1}, \eqref{cheqn2}, with $u_0 \in \mathcal{I}_0$, by considering the regularised potential function defined for $\delta \in (0,1)$ by

\begin{equation}\label{dpot}
	F_{\log}^{\delta}(r) := 
	\begin{cases}
		(1-r) \log(\delta) + (1+r) \log(2-\delta) + \frac{(1-r)^2}{2 \delta} + \frac{(1+r)^2}{2 (2 -  \delta)} - 1, \ & r \geq 1 - \delta \\
		(1+r) \log(1+r) + (1-r)\log(1-r), \ & |r| < 1- \delta\\
		(1+r)\log(\delta) + (1-r)\log(2-\delta) + \frac{(1+r)^2}{2 \delta} + \frac{(1-r)^2}{2(2 - \delta)} - 1, \ & r \leq -1 + \delta
	\end{cases}.
\end{equation}
It can be shown that $F^{\delta}_{\log} \in C^2(\mathbb{R})$.
This approach has been used in several papers on the logarithmic potential, see for example \cite{barrett1995error,copetti1992numerical, elliott1991generalized}.
We will adopt the shorthand notation
\[F^\delta(r) = \frac{\theta}{2}F^\delta_{\text{log}}(r) + \frac{1-r^2}{2}, \quad \fd(r) := (F^{\delta}_{\log})'(r).\]
It can be shown that the function $\fd$ is such that for $r,s \in \mathbb{R}$,
\begin{gather}\
	\label{phidbound1}
	(r-s)^2 \leq (\fd(r) - \fd(s))(r-s),\\
	\label{phidbound2}|\fd(r) - \fd(s)| \leq \frac{1}{\delta} |r-s|,
\end{gather}
for sufficiently small $\delta$.
If also $|r|, |s| > 1-\delta$ then we have
\begin{align}
	\label{phidbound3}
	\frac{1}{2\delta} (r-s)^2 \leq (\fd(r) - \fd(s))(r-s).
\end{align}

We end this section by proving an error bound between the exact solution, and the solution of the regularised equation.
We also recall the definition of the inverse Laplacian, $\mathcal{G}$, from Appendix \ref{invlaps}.
This proof is adapted from \cite[Theorem 2.1]{barrett1995error}.

\begin{theorem}
	\label{u delta error theorem}
	Let $(u,w)$ denote the solution of \eqref{cheqn1}, \eqref{cheqn2}, and $(u^\delta, w^{\delta})$ denote the solution of the regularised problem for $\delta \in (0,1)$.
	Then for sufficiently small $\delta > 0$ we have that
	\begin{align}
		\varepsilon \int_0^T \| \gradg (u - u^\delta ) \|_{L^2(\Gamma(t))}^2 + \sup_{t \in [0,T]} \| u - u^{\delta} \|_{-1}^2 \leq C \delta. \label{deltaerror1}
	\end{align}
\end{theorem}

\begin{proof}
	To begin, we define $E_u:= u - u^\delta, E_w:= w - w^\delta$, which one readily finds satisfies
	\begin{gather}
		m_*\left(\matdev{E_u} , \phi \right) + g(E_u, \phi)+ a (E_w, \phi) = 0,\label{regdif1}\\ 
		m(E_w, \phi) = \varepsilon a(E_u,\phi) + \frac{\theta}{2 \varepsilon} m (f(u)- \fd(u^\delta), \phi) - \frac{1}{\varepsilon} m(E_u,\phi),\label{regdif2}
	\end{gather}
	for all $\phi \in H^1(\Gamma(t))$ and almost all $t \in [0,T]$.
	We note $\mval{E_u}{\Gamma(t)} = 0$, and so $\mathcal{G} E_u$ is well defined.
	Hence we test \eqref{regdif1} with $\mathcal{G} E_u$ and \eqref{regdif2} with $E_u$, from which one readily finds that
	\begin{align}
		m_*\left(\matdev{E_u} , \mathcal{G} E_u \right) + g(E_u, \mathcal{G} E_u)+ \varepsilon a(E_u,E_u) + \frac{\theta}{2 \varepsilon} m (f(u)- \fd(u^\delta), E_u) - \frac{1}{\varepsilon} m(E_u,E_u) = 0. \label{regdif3}
	\end{align}
	We now want to rewrite the first two terms appropriately.
	Firstly we observe that
	\[ m_*\left(\matdev{E_u} , \mathcal{G} E_u \right) + g(E_u, \mathcal{G}E_u) = \frac{d}{dt} m(E_u, \mathcal{G} E_u) - m(E_u, \matdev\mathcal{G} E_u),  \]
	where we note that $\matdev \mathcal{G} E_u \in H^1(\Gamma(t))$ (see Appendix \ref{invlaps}).
	Then we see that $m(E_u, \matdev\mathcal{G} E_u) = a(\mathcal{G}E_u, \matdev\mathcal{G} E_u)$, and from the transport theorem
	\[ a(\mathcal{G}E_u, \matdev\mathcal{G} E_u) = \frac{1}{2} \frac{d}{dt} a(\mathcal{G} E_u, \mathcal{G} E_u) - \frac{1}{2} b(\mathcal{G} E_u,\mathcal{G} E_u). \]
	Hence we find that \eqref{regdif3} becomes
	\begin{align}
		\frac{1}{2} \frac{d}{dt}\|E_u\|_{-1}^2 + \frac{1}{2} b(\mathcal{G} E_u, \mathcal{G} E_u) + \varepsilon a(E_u,E_u) + \frac{\theta}{2 \varepsilon} m (f(u)- \fd(u^\delta), E_u) = \frac{1}{\varepsilon} \|E_u\|_{L^2(\Gamma(t))}^2. \label{regdif4}
	\end{align}
	Now from the definition of $\mathcal{G}$ and Young's inequality
	\[ \frac{1}{\varepsilon} \|E_u\|_{L^2(\Gamma(t))}^2 \leq \frac{1}{2 \varepsilon^3} \|E_u\|_{-1}^2 + \frac{\varepsilon}{2} \|\gradg E_u \|_{L^2(\Gamma(t))}^2,  \]
	and likewise using the smoothness of $V$, one finds
	\[ b(\mathcal{G} E_u, \mathcal{G} E_u) \leq C \|E_u\|_{-1}^2. \]
	Next we express $m (f(u)- \fd(u^\delta), E_u)$ as
	\[ m (f(u)- \fd(u^\delta), E_u) = m (\fd(u)- \fd(u^\delta), E_u) - m (\fd(u) - f(u), E_u). \]
	Combining these in \eqref{regdif4} yields
	\begin{align}
		\frac{d}{dt}\|E_u\|_{-1}^2 + \varepsilon a(E_u,E_u) + \frac{\theta}{\varepsilon} m (\fd(u)- \fd(u^\delta), E_u) \leq C \|E_u\|_{-1}^2 + \frac{\theta}{\varepsilon} m (\fd(u)- f(u), E_u). \label{regdif5}
	\end{align}
	Hence we aim to bound the potential terms suitably to obtain the result.
	Firstly, by defining the sets
	\begin{gather*}
		\Gamma^+_{\delta}(t) := \left\{ x \in \Gamma(t) \ | \ 1 - \delta \leq u(x,t) \leq u^\delta(x,t) \right\},\\
		\Gamma^-_{\delta}(t) := \left\{ x \in \Gamma(t) \ | \  - 1 + \delta \geq u(x,t) \geq u^\delta(x,t) \right\},
	\end{gather*}
	and using the monotonicity of $\fd$ and \eqref{phidbound3} we find that 
	\[m(\fd(u) - \fd(u^{\delta}), E_u) \geq \frac{1}{2 \delta} \int_{\Gamma_\delta^+(t) \cup \Gamma_\delta^-(t) } E_u^2\]
	Secondly, we can consider three cases:
	\begin{enumerate}
		\item For $|r| \leq 1 - \delta$ we have that $\fd(r) = f(r)$,
		\item For $ r \geq 1 - \delta$ and $s \leq r$ we have $(\fd(r) - f(r))(r-s) \leq 0$,
		\item For $ r \leq  - 1 + \delta$ and $s \geq r$ we have $(\fd(r) - f(r))(r-s) \leq 0$,
	\end{enumerate} 
	hence we find that
	$$  m(\fd(u) - f(u), E_u) \leq - \int_{\Gamma_\delta^+(t) \cup \Gamma_\delta^-(t) } f(u)E_u ,$$
	as $\fd(u) E_u \leq 0$ on $\Gamma_\delta^+(t) \cup \Gamma_\delta^-(t)$.
	and by applying Young's inequality one finds that
	$$ \left| \int_{\Gamma_\delta^+(t) \cup \Gamma_\delta^-(t) } f(u)E_u \right| \leq \frac{1}{2 \delta} \int_{\Gamma_\delta^+(t) \cup \Gamma_\delta^-(t) } E_u^2 + \frac{ \delta}{2} \int_{\Gamma(t)} f(u)^2.$$
	Combining these bounds in \eqref{regdif5} one finds
	$$ \varepsilon \| \gradg E_u \|_{L^2(\Gamma(t))}^2 + \frac{d}{dt} \|E_u\|_{-1}^2 \leq C\|E_u\|_{-1}^2 + \frac{\delta \theta}{2 \varepsilon} \| f(u) \|_{L^2(\Gamma(t))}^2.$$
	Thus integrating in time, noting that $f(u) \in L^2_{L^2}$ and applying Gr\"onwall's inequality yields the result.
\end{proof}

\section{Semi-discretisation of the problem}
\label{section: logch semi discrete}
\subsection{Finite element problem}
We now propose and analyse a spatially discrete numerical scheme for \eqref{cheqn1}, \eqref{cheqn2}.
In order to show stability of the numerical scheme we make the following assumption on the evolution of $\Gamma(t)$.
\begin{assumption}
	We assume that the velocity, $V$, of $\Gamma(t)$ is such that $\gradg \cdot V \geq 0$ for all $(x,t) \in \mathcal{S}_T$.
\end{assumption}
This condition has also been considered in recent work on the degenerate Cahn-Hilliard equation on an evolving surface \cite{elliott2024degenerate}, in which the following geometric formulation of this assumption is shown.
\begin{lemma}[{\cite[Lemma 4.1]{elliott2024degenerate}}]
	$\gradg \cdot V \geq 0$ if, and only if, for every $\mathcal{H}^2$ measurable region $\Sigma(t) \subset \Gamma(t)$,
	\[ \frac{d}{dt} |\Sigma(t)| \geq 0. \]
\end{lemma}

As remarked in \cite{elliott2024degenerate} if one considers a velocity which is exclusively in the normal direction then $\gradg \cdot V = HV_N$.
As such that above condition holds for a normal velocity is of the form $V_N = g(H)$, as a function of the mean curvature, where $g(\cdot)$ is sufficiently smooth, and such that $xg(x) \geq 0$.
An explicit example of such an evolution is inverse mean curvature flow, $g(H) = \frac{1}{H}$.
We refer the reader to \cite{bethuel1999geometric} for a discussion of geometric evolution equations of form $V_N = g(H)$.

It is not clear that $\gradg \cdot V \geq 0$ would imply that $\gradgh \cdot V_h \geq 0$.
However, if instead we assume a strict inequality, $\gradg \cdot V > 0$, then as $V_h = I_h V^{-\ell}$, using \eqref{lift2} and \eqref{interpolation} one finds that for sufficiently small $h$, 
\[ \gradg \cdot V > 0 \Rightarrow \gradgh \cdot V_h > 0. \]

As a slightly weaker assumption we assume that $\gradgh \cdot V_h \geq 0$, which holds for a class of expanding surfaces (from the explanation above), but this also allows the possibility that $\Gamma(t)$ is a stationary surface.

As seen in the previous section, we consider a set of admissible initial conditions, given by
$$ \mathcal{I}_{h,0} := \left\{ \eta_h \in S_h(0) \mid |\eta_h| \leq 1 \ \text{a.e. on }\Gamma_{h}(0), \ \Ech_h[\eta_h;0] < \infty, \ \left|\mval{\eta_h}{\Gamma_h(0)} \right| < 1 \right\},$$
where
\begin{align}\label{discglfunctional}
	\Ech_h[U_h;t] = \int_{\Gamma_h(t)} \frac{\varepsilon |\gradgh U_h|^2}{2} + \frac{1}{\varepsilon} I_h F(U_h)
\end{align}
is a discrete version of the Ginzburg-Landau functional.
We note that the assumption $\gradgh \cdot V_h \geq 0$ implies that $\frac{1}{|\Gamma_h(0)|}\geq\frac{1}{|\Gamma_h(t)|}$ and so the condition on the mean value is simpler than that for $\mathcal{I}_0$ defined in Section \ref{section: logch variational form}.

Discretising in space gives rise to our semi-discrete scheme, where the solution spaces in which we pose the problem are
\begin{align*}
	\tilde{S}_h^T := \left\{ z_h \in C(\mathcal{S}_{h,T}) \mid z_h(\cdot, t) \in S_h(t) \right\}, \quad S_h^T := \left\{ z_h \in \tilde{S}_h^T \mid \matdev_h{z_h} \in C(\mathcal{S}_{h,T} ) \right\}.
\end{align*}

The semi-discrete problem is as follows.
Given initial data $U_{h,0} \in \mathcal{I}_{h,0}$, approximating some $u_0 \in \mathcal{I}_0$, we want to find $(U_h, W_h) \in S_h^T \times \tilde{S}_h^T$ such that
\begin{gather}\label{fecheqn1}
	\intm\left(\matdev_h U_h , \phi_h \right) + g_h(I_h(U_h \phi_h),1) + a_{h}(W_h,\phi_h) = 0,\\ 
	\intm(W_h ,\phi_h ) = \varepsilon a_{h}(U_h,\phi_h) + \frac{\theta}{2 \varepsilon} \intm (I_h f(U_h), \phi_h)- \frac{1}{\varepsilon} \intm(U_h,\phi_h),\label{fecheqn2}
\end{gather}
for all $\phi_h(t) \in S_h(t)$ and a.e. $t \in [0,T]$, satisfying the initial condition $U_h(0) = U_{h,0}$ a.e. in $\Gamma_h(0)$.
We note that unlike the schemes of \cite{beschle2022stability,elliott2015evolving, elliott2024fully} we have used mass lumping to control the nonlinearity, as in \cite{barrett1995error, copetti1992numerical}.
This allows us to establish $\delta$-independent bounds on the nonlinearity, which otherwise proves difficult.
Lastly we observe that the $g_h$ term in \eqref{fecheqn1} is motivated by Proposition \ref{transport5}, and will allow us to retain a mass conservation property as taking $\phi_h = 1$ above yields
\[0 = \intm\left(\matdev_h U_h , 1 \right) + g_h(U_h,1) = \frac{d}{dt}\intm(U_h,1) = \frac{d}{dt}\int_{\Gamma_h(t)}U_h.\]

\subsection{Well-posedness of finite element problem}
\subsubsection{Regularisation}
To show well-posedness of this problem we consider a regularisation as in \cite{caetano2021cahn,caetano2023regularization}.
The regularised semi-discrete problem is to find a pair $(\Udh, \Wdh) \in S_h^T \times \tilde{S}_h^T$ solving
\begin{gather}\label{discregch}
	\intm\left(\matdev_h{ \Udh}, \phi_h\right) +g_h(I_h(\Udh \phi_h),1) + a_{h}(\Wdh,\phi_h) = 0,\\
	\intm(\Wdh ,\phi_h ) = \varepsilon a_{h}(\Udh,\phi_h) + \frac{\theta}{2 \varepsilon} \intm (I_h \fd(\Udh),\phi_h) - \frac{1}{\varepsilon} \intm (\Udh,\phi_h), \label{discregch2}
\end{gather}
for all $\phi_h \in S_h$ and a.e. $t \in [0,T]$, along with the initial condition $\Udh(0) = U_{h,0}$.
As before, we require that $U_{h,0} \in \mathcal{I}_{h,0}$.

As part of the existence proof we will make use of a discrete and regularised version of the Ginzburg-Landau functional \eqref{glfunctional},
$$\Echdh\left[ \Udh ;t\right] := \int_{\Gamma_{h}(t)} \varepsilon \frac{|\gradgh \Udh|^2}{2} + \frac{1}{\varepsilon } I_h F^{\delta}(\Udh). $$

The proof of well posedness is quite long, and hence separated into several parts.
First we prove global existence of a solution to the regularised problem, \eqref{discregch}. Then show uniform bounds, independent of $\delta$ so we may pass to the limit $\delta \rightarrow 0$.
We then verify that the solution is unique by showing a stability result.
As part of this well posedness, we must show that the solution is such that $|U_h| < 1$ almost everywhere on $\Gamma_h(t)$ for almost all $t \in [0,T]$, due to the poles of $f$ at $\pm 1$.

Throughout we consider initial data $U_{h,0} \in \mathcal{I}_{h,0}$ approximating some $u_0 \in \mathcal{I}_0$.
Namely, for a given $u_0$ we consider $U_{h,0}$ such that
$$ \left\| u_0 - U_{h,0}^\ell \right \|_{L^2(\Gamma_0)} + h\left\| u_0 - U_{h,0}^\ell \right \|_{H^1(\Gamma_0)} \leq Ch^2,$$
where $C$ is independent of $h$ but may depend on $u_0$.
This allows us to make bounds involving $U_{h,0}$ independent of $h$ by lifting onto $\Gamma_0$.
Suitable examples, for sufficiently smooth $u_0$, are the Lagrange interpolant and the Ritz projection.

\begin{lemma}
	\label{regfeexist}
	Given $h \in (0,h_0)$, $\delta \in (0,1)$ and $U_{h,0} \in \mathcal{I}_{h,0}$ there exists a unique solution pair $(\Udh,\Wdh) \in S_h^T \times \tilde{S}_h^T$, solving \eqref{fecheqn1}, \eqref{fecheqn2} for all $\phi_h \in S_h(t)$ and $t \in [0,T]$, along with the initial condition $\Udh(0) = U_{h,0}$.\\
	
	Moreover, this pair is such that
	\begin{equation}\label{discregbound}
		\sup_{t \in [0,T]} \Echdh[\Udh(t)] + \int_0^{T} \| \gradgh \Wdh (t) \|_{L^2(\Gamma_h(t))}^2 \, dt \leq C,
	\end{equation}
	where $C$ is a constant independent of $\delta$ and $h$.
\end{lemma}
\begin{proof}
	We firstly show the short time existence of the functions $\Udh, \Wdh$.
	We enumerate the finite element basis functions as $(\phi_j(t))_{j=1}^{N_h}$ and express the functions $\Udh, \Wdh$ as
	$$ \Udh(t) = \sum_{j=1}^{N_h} \alpha_j^{\delta}(t) \phi_j(t), \quad \Wdh(t) = \sum_{j=1}^{N_h} \beta_j^{\delta}(t) \phi_j(t).$$
	Then the system of ODEs above can be written as
	\begin{gather*}
		\bar{M}(t)\frac{d}{dt} \alpha^{\delta}(t) + \bar{G}(t)\alpha^\delta(t) + A(t) \beta^{\delta}(t) = 0,\\
		\bar{M}(t) \beta^\delta(t) = \varepsilon A(t)\alpha^{\delta}(t) + \frac{\theta}{2 \varepsilon} \mathcal{F}^{\delta}(t;\alpha^{\delta}(t)) - \frac{1}{\varepsilon} \bar{M}(t)\alpha^{\delta}(t).
	\end{gather*}
	Here $\alpha^{\delta}, \beta^{\delta}$ are the vectors of coefficients, and the other terms are given by
	\begin{gather*}
		\bar{M}_{ij}(t) = \intm(\phi_i(t),\phi_j(t)), \quad \bar{G}_{ij}(t) = g_h(I_h(\phi_i(t)\phi_j(t)),1), \quad A_{ij}(t) = a_{h}(\phi_i(t),\phi_j(t)),\\
		\mathcal{F}^{\delta}(t;\alpha^{\delta}(t))_j = \intm \left(I_h\fd \left( \sum_{i=1}^{N_h} \alpha_i^{\delta}(t) \phi_i(t) \right), \phi_j(t)\right) = \intm \left( \sum_{i=1}^{N_h} \fd(\alpha_i^{\delta}(t)) \phi_i(t), \phi_j(t)\right).
	\end{gather*}
	This system is then subject to the initial condition, $ \alpha^{\delta}(0) = \alpha_0,$
	where $U_{h,0} = \sum_{j=1}^{N_h} \alpha_{0,j}\phi_j(0)$.
    We note that the use of mass-lumping here means that both the matrices $\bar{M}(t), \bar{G}(t)$ are diagonal.
	Noting that $\bar{M}(t)$ is positive definite, we can decouple the equations for a single ODE for $\alpha^{\delta}(t)$.
	As $\fd$ is $C^1$ by construction we observe that $\mathcal{F}^\delta(t;\cdot)$ is locally Lipschitz, and we can apply standard ODE theory to obtain short time existence of a unique $\alpha^{\delta}$ and $\beta^{\delta}$.
	
	Next we show that $\Udh, \Wdh$ exist on $[0,T]$ by showing the bound \eqref{discregbound}. 
	We start by differentiating the regularised energy for
	\begin{multline*}
		\frac{d}{dt} \Echdh[\Udh] = \varepsilon a_{h}\left(\matdev_h{\Udh},\Udh \right) + \frac{\varepsilon}{2} b_h(\Udh,\Udh) + \frac{\theta}{2\varepsilon} \intm \left(I_h \fd(\Udh),\matdev_h{\Udh} \right) - \frac{1}{\varepsilon} \intm \left(\Udh,\matdev_h{\Udh} \right)\\
        + \frac{1}{\varepsilon} g_h \left(I_h F^\delta(\Udh),1 \right),
	\end{multline*}
    where we have used Proposition \ref{transport3} and Lemma \ref{transport5}.
	Now testing \eqref{discregch} with $\Wdh$, and \eqref{discregch2} with $\matdev_h{\Udh}$ we see
	\begin{gather*}
		\intm \left(\matdev_h \Udh, \Wdh \right) = -g_h(I_h(\Udh\Wdh),1) - a_{h}(\Wdh,\Wdh),\\
		\intm \left(\Wdh,\matdev_h{\Udh}\right) = \varepsilon a_{h}\left(\matdev_h{\Udh}, \Udh \right) + \frac{\theta}{2\varepsilon} \intm \left(I_h \fd(\Udh), \matdev_h{\Udh} \right) - \frac{1}{\varepsilon} \intm \left(\Udh, \matdev_h{\Udh} \right),
	\end{gather*}
	from which we obtain
	\begin{align}
		\frac{d}{dt} \Echdh[\Udh] + \| \gradgh \Wdh \|_{L^2(\Gamma_h(t))}^2 + g_h(I_h(\Udh\Wdh),1) = \frac{\varepsilon}{2} b_h(\Udh,\Udh) +\frac{1}{\varepsilon}g_h \left(I_h F^\delta(\Udh),1 \right). \label{energypf1}
	\end{align}
	It is straightforward to see from the smoothness assumption on $V$ that
	\[ \frac{\varepsilon}{2} b_h(\Udh,\Udh) + \frac{1}{\varepsilon}g_h \left(I_h F^\delta(\Udh),1 \right) \leq C + C \Echdh[\Udh]. \]
	All that remains is to deal with the $g_h(I_h(\Udh\Wdh),1)$ term, and this is where we use the assumption that $\gradgh \cdot V_h \geq 0$.
	Returning to the matrix formulation, we see that $g_h(I_h(\Udh\Wdh),1) = \alpha^\delta \cdot \bar{G} \beta^\delta$, which we can write independent of $\beta^\delta$ as
	\[ \alpha^\delta \cdot \bar{G} \beta^\delta = \alpha^\delta \cdot \bar{G} \bar{M}^{-1} \left[ \varepsilon A\alpha^{\delta} + \frac{\theta}{2\varepsilon} \mathcal{F}^{\delta}(\alpha^{\delta}) - \frac{1}{\varepsilon} \bar{M}\alpha^{\delta}  \right],\]
    where the plan is to show the first and third terms are bounded above (in absolute value) and the second term is bounded below.
	Under the assumption that $\gradgh \cdot V_h  \geq 0$ one finds that $\alpha^\delta \cdot \bar{G} \bar{M}^{-1} \mathcal{F}^\delta(\alpha^\delta) \geq 0$.
	To see this is true we first observe, by definition of $\mathcal{F}^\delta$, that we may write $\mathcal{F}^\delta(\alpha^\delta) = \bar{M} f^\delta(\alpha^\delta)$, where we abuse notation and understand $f^\delta(\alpha^\delta)$ as a vector with entries $f^\delta(\alpha^\delta_i)$.
    This is justified as
    \[\mathcal{F}^\delta(t;\alpha^\delta(t))_j = \intm \left( \sum_{i=1}^{N_h} \fd(\alpha_i^{\delta}(t)) \phi_i(t), \phi_j(t)\right) = \intm \left(\fd(\alpha_j^{\delta}(t)) \phi_j(t), \phi_j(t)\right) = (\bar{M}(t) \fd(\alpha^\delta(t)))_j,\]
    since $\intm(\phi_i(t),\phi_j(t)) = 0$ for $i \neq j$.
	Hence we find that
    \[ \alpha^\delta \cdot \bar{G} \bar{M}^{-1} \mathcal{F}^\delta(\alpha^\delta) = \alpha^\delta \cdot \bar{G} \fd(\alpha^\delta) = g_h(I_h(\Udh \fd(\Udh)),1) \geq 0 \]
	where we have used the fact that $r f^\delta(r) \geq 0$ and the assumption $\gradgh \cdot V_h \geq 0$.
    It remains to show upper bounds on the other terms, namely $\varepsilon \alpha^\delta \cdot \bar{G} \bar{M}^{-1} A\alpha^{\delta}$ and $\alpha^\delta \cdot \bar{G} \alpha^\delta$.
    We firstly consider the second of these since it is more obvious.
    By definition
    \[\alpha^\delta \cdot \bar{G}{\alpha}^\delta = g_h(I_h((\Udh)^2),1),\]
    and using the smoothness of $V$ one finds 
    \[ |g_h(I_h((\Udh)^2),1)| \leq C \|\Udh\|_{h,t}^2 \leq C + C \Echdh[\Udh],\]
    where we have used the Poincar\'e inequality\footnote{As is typical for the Cahn-Hilliard equation, we also have used the mass conservation property $\int_{\Gamma_h(t)} \Udh(t) = \int_{\Gamma_h(0)} U_{h,0}$, which follows from testing \eqref{discregch} with $\phi_h = 1$.}.

    The more involved term is $\varepsilon \alpha^\delta \cdot \bar{G} \bar{M}^{-1} A\alpha^{\delta}$.
    For this we require two new objects.
    Firstly, we define $\widetilde{\Udh} \in S_h(t)$ to be the unique solution of
    \[ \intm(\widetilde{\Udh}, \phi_h) = a_h(\Udh, \phi_h), \]
    for all $\phi_h \in S_h(t)$.
    We note that such a function is bounded with
    \begin{align}
    \|\widetilde{\Udh}\|_{h,t} \leq \frac{C}{{h}}\|\gradgh\Udh\|_{L^2(\Gamma_h(t))},
        \label{laplacian inverse inequality}
    \end{align}
    where we have used an inverse inequality, and Young's inequality.
    If we denote the vector of nodal values of $\widetilde{\Udh}$ by $\widetilde{\alpha^\delta}$, one finds that $\widetilde{\alpha^\delta} = \bar{M}^{-1} {A} {\alpha^\delta}$, and hence we may write
    \[ \varepsilon \alpha^\delta \cdot \bar{G} \bar{M}^{-1} A\alpha^{\delta} = \varepsilon g_h(I_h(\widetilde{\Udh}\Udh),1). \]

    Next we introduce the $L^2$ projection onto $S_h(t)$.
    For $\psi \in L^2(\Gamma_h(t))$ we define $\Lambda_h (\psi) \in S_h(t)$ to be the unique solution of
    \[ m_h(\Lambda_h (\psi), \phi_h) = m_h(\psi, \phi_h), \]
    for all $\phi_h \in S_h(t)$.
    If one has that $\psi \in H^1(\Gamma_h(t))$ then it is known (see \cite{ern2021finite,elliott2015evolving}) that since our triangulation is uniformly quasi-uniform
    \begin{gather}
        \|\Lambda_h \psi \|_{H^1(\Gamma_h(t))} \leq C \|\psi\|_{H^1(\Gamma_h(t))}, \label{l2 projection bound1}\\
        \|\Lambda_h \psi - \psi \|_{L^2(\Gamma_h(t))} \leq C h\|\psi\|_{H^1(\Gamma_h(t))} \label{l2 projection bound2},
    \end{gather}
    for constants independent of $t$.
    To bound $g_h(I_h(\widetilde{\Udh}\Udh),1)$ we write
    \begin{multline*}
        |g_h(I_h(\widetilde{\Udh}\Udh),1)| \leq |g_h(I_h(\widetilde{\Udh}\Udh),1) - g_h(\widetilde{\Udh}, \Udh)| + |g_h(\widetilde{\Udh}, \Udh) - m_h(\widetilde{\Udh}, \Udh (\gradg \cdot V)^{-\ell})|\\
        + |m_h(\widetilde{\Udh}, \Udh (\gradg \cdot V)^{-\ell}) - m_h(\widetilde{\Udh}, \Lambda_h(\Udh (\gradg \cdot V)^{-\ell}))| + |a_h({\Udh}, \Lambda_h(\Udh (\gradg \cdot V)^{-\ell}))|,
    \end{multline*}
    where we have used the definition of $\widetilde{\Udh}$ in the last term.
    Now we use \eqref{intmbound3} and an inverse inequality to see that
    \[ |g_h(I_h(\widetilde{\Udh}\Udh),1) - g_h(\widetilde{\Udh}, \Udh)| \leq Ch \|\widetilde{\Udh}\|_{L^2(\Gamma_h(t))} \|\gradgh \Udh \|_{L^2(\Gamma_h(t))} \leq C \|\gradgh \Udh \|_{L^2(\Gamma_h(t))}^2.\]
    Next we use the fact that $V_h = I_h V$, and \eqref{interpolation2} (as well as \eqref{lift2} where needed) to see that
    \begin{align*}
    |g_h(\widetilde{\Udh}, \Udh) - m_h(\widetilde{\Udh}, \Udh (\gradg \cdot V)^{-\ell})| & \leq Ch \|\widetilde{\Udh} \|_{L^2(\Gamma_h(t))} \|\Udh\|_{L^2(\Gamma_h(t))}\\
    & \leq C \|\gradgh \Udh \|_{L^2(\Gamma_h(t))} \|\Udh\|_{L^2(\Gamma_h(t))},
    \end{align*}
    where we have again used \eqref{laplacian inverse inequality} for the final inequality.
    Finally for the last two terms we use \eqref{l2 projection bound1} and \eqref{l2 projection bound2}, as well as the smoothness of $V$, respectively to see that
    \begin{gather*}
        |m_h(\widetilde{\Udh}, \Udh (\gradg \cdot V)^{-\ell}) - m_h(\widetilde{\Udh}, \Lambda_h(\Udh (\gradg \cdot V)^{-\ell}))| \leq C \|\Udh\|_{H^1(\Gamma_h(t))}^2 \\
        |a_h({\Udh}, \Lambda_h(\Udh (\gradg \cdot V)^{-\ell})| \leq C \| \Udh\|_{H^1(\Gamma_h(t))}^2.
    \end{gather*}
    All in all, using these bounds in conjunction with the Poincar\'e inequality one obtains a bound of the form
    \[ |g_h(I_h(\widetilde{\Udh}\Udh),1)| \leq C + C \Echdh[\Udh]. \]
	Combining these facts together, and using the Poincar\'e inequality as appropriate, in \eqref{energypf1} yields
	\begin{align}
		\frac{d}{dt} \Echdh[\Udh] + \| \gradgh \Wdh \|_{L^2(\Gamma_h(t))}^2 \leq C + C \Echdh[\Udh], \label{energypf2}
	\end{align}
	where one concludes by using a Gr\"onwall inequality.
	
	Lastly note $\Echdh[U_{h,0}]$ can be bounded independent of $h$ by taking lifts onto the surface $\Gamma_0$ and using $\left\| u_0 - U_{h,0}^\ell \right \|_{H^1(\Gamma_0)} \leq Ch$, for some function $u_0 \in \mathcal{I}_0$.
	It remains to see that the bound is independent of $\delta$.
	For this we note that, by the definition of $F^\delta$
	\begin{multline*}
		\int_{\Gamma_h(0)} I_h F^\delta(U_{h,0}) = \int_{\{|U_{h,0}|< 1 - \delta\}} I_h F_{\log}(U_{h,0})
		+ \int_{\{U_{h,0} \geq 1- \delta\}} \sum_{k=0}^2 \frac{F_{\log}^{(k)}(1-\delta)}{k!}I_h\left((U_{h,0}-1+\delta)^k\right)\\
		+ \int_{\{U_{h,0} \leq -1 + \delta\}} \sum_{k=0}^2 \frac{F_{\log}^{(k)}(-1+\delta)}{k!}I_h\left((U_{h,0}+1-\delta)^k\right) + \int_{\Gamma_h(0)} \frac{1 - I_h\left((U_{h,0})^2\right)}{2}
	\end{multline*}
	and so we bound these first three terms.
	Since $U_{h,0} \in \mathcal{I}_{h,0}$, we see $|U_{h,0}| \leq 1$, and one clearly has 
	\[ \int_{\{|U_{h,0}|< 1 - \delta\}} I_hF_{\log}(U_{h,0}) \leq |\Gamma_h(0)|\sup_{r \in [-1,1]} F_{\log}(r) < \infty, \]
	and for the terms involving derivatives one finds that
	\begin{multline*}
	\int_{\{U_{h,0} \geq 1- \delta\}} \sum_{k=0}^2 \frac{F_{\log}^{(k)}(1-\delta)}{k!}I_h(U_{h,0}-1+\delta)^k\\
 \leq \frac{|\Gamma_h(0)|}{2} \left( \sup_{r \in [-1,1]} F_{\log}(r) + \delta \log\left(\frac{2 - \delta}{ \delta}\right) + \frac{ 4 \delta^2}{2\delta - \delta^2} \right),
	\end{multline*}
	which is finite uniformly bounded for $\delta \in (0,1)$.
	The bound for the integral over $\{U_{h,0} \leq -1 + \delta\}$ is similar.
\end{proof}
\begin{remark}
	The assumption $\gradgh \cdot V_h \geq 0$ was required here as adapting the argument of \cite{caetano2023regularization} would be analogous to taking the limit $h\rightarrow0$ then $\delta \rightarrow 0$.
	We expect that one can drop this assumption and expand the analysis presented here to a wider class of evolving surfaces.
\end{remark}

We use this bound to show a uniform bound for $\matdev_h{\Udh}$ in $L^2_{H^{-1}}$.

\begin{lemma}
	\label{semi discrete derivative lemma}
	For $\delta \in (0,1)$ we have that
	$$ \int_0^T \left\| \matdev_h{\Udh} \right\|_{H^{-1}(\Gamma_h(t))}^2 \leq C$$
	for a constant $C$ independent of $\delta$ and $h$.
\end{lemma}
\begin{proof}
	Firstly we note that $\matdev_h{\Udh}$ does not have mean value 0, and hence we cannot immediately use an inverse Laplacian.
	We do know however that it has a bounded mean value, as taking $\phi_h = 1$ in \eqref{discregch} one finds
	\[ \int_{\Gamma_h(t)} \matdev_h \Udh = g_h(\Udh,1) \leq C \|\Udh\|_{L^1(\Gamma_h(t))}, \]
	from which one can use \eqref{discregbound} to see that
	\begin{align}
		\int_{0}^T \left(\mval{\matdev_h \Udh}{\Gamma_h(t)}\right)^2 \leq C. \label{regderivative1}
	\end{align}
	Now by the definition of $\intinvsh$ from Appendix \ref{invlaps} one tests \eqref{discregch} with $\intinvsh\left( \matdev_h{\Udh} - \mval{\matdev_h \Udh}{\Gamma_h(t)} \right)$ for
	\begin{multline*}
		\inormh{\matdev_h{\Udh} - \mval{\matdev_h \Udh}{\Gamma_h(t)}}^2 = -a_h\left(\Wdh, \intinvsh \left(\matdev_h{\Udh} - \mval{\matdev_h \Udh}{\Gamma_h(t)}\right)\right)\\
		- g_h\left(I_h\left(\Udh\intinvsh \left(\matdev_h{\Udh} - \mval{\matdev_h \Udh}{\Gamma_h(t)}\right)\right),1\right).
	\end{multline*}
	Now using Young's inequality, Poincar\'e's inequality and \eqref{discregbound} it follows that
	\begin{align}
		\int_0^T \inormh{\matdev_h{\Udh} - \mval{\matdev_h \Udh}{\Gamma_h(t)}}^2 \leq C. \label{regderivative2}
	\end{align}
	Next, by using \eqref{invlapduality}, \eqref{invlapineq2}, \eqref{invlapineq3} one finds that
	\[\int_0^T \left\| \matdev_h{\Udh} \right\|_{H^{-1}(\Gamma_h(t))}^2 \leq C \int_{0}^T \left(\mval{\matdev_h \Udh}{\Gamma_h(t)}\right)^2 + C\int_0^T \inormh{\matdev_h{\Udh} - \mval{\matdev_h \Udh}{\Gamma_h(t)}}^2, \]
	and the result follows from \eqref{regderivative1}, \eqref{regderivative2}.
\end{proof}

Next we state a result showing that one can control the measure of the set of values such that $|\Udh| > 1$.
For brevity's sake we do not prove this, as it follows from minor adaptions to \cite[Lemma 5.8]{caetano2021cahn}.

\begin{lemma}
	\label{udhcontrol}
	There exists a constant, $C$, independent of $h,\delta$ such that
	\begin{equation}\label{udhextrem}
		\int_{\Gamma_{h}(t)} [-1 - \Udh(t)]_+ + \int_{\Gamma_{h}(t)} [\Udh(t) - 1]_+ \leq C \left(\frac{1}{|\log(\delta)|} + \delta\right),
	\end{equation}
	for all $t \in [0,T]$, where $[f]_+ := \max(0,f)$.
\end{lemma}

We now use this control to obtain uniform bounds for the potential term by adapting a proof in \cite{copetti1992numerical}.
\begin{lemma}
	\label{dpotbound}
	For sufficiently small $\delta$ we have that
	\begin{gather}\label{dpotbound1}
			\int_0^T \left\| I_h \fd(\Udh) - \mval{I_h \fd(\Udh)}{\Gamma_h(t)} \right\|_{L^2(\Gamma_h(t))}^2 \leq C,\\
			\label{dpotbound2}
			\int_0^T \| I_h \fd(\Udh) \|_{L^2(\Gamma_h(t))}^2 \leq C,
		\end{gather}
	where C denotes a constant independent of $\delta$ and $h$.
\end{lemma}
\begin{proof}[\underline{Proof of Lemma \ref{dpotbound}, part 1}]
	\ \newline
	To show the first claim we test \eqref{discregch2} with $\phi_h = I_h \fd (\Udh) - \mval{I_h \fd(\Udh)}{\Gamma_h(t)}$ to obtain 
	\begin{multline*} \frac{\theta}{2 \varepsilon} \intm \left(I_h\fd(\Udh), I_h \fd (\Udh) - \mval{I_h \fd(\Udh)}{\Gamma_h(t)}\right) =  \intm\left(\Wdh ,I_h \fd (\Udh) - \mval{I_h \fd(\Udh)}{\Gamma_h(t)}\right)\\
		- \varepsilon a_{h}\left(\Udh,I_h \fd (\Udh)\right) + \frac{1}{\varepsilon} \intm\left(\Udh,I_h \fd (\Udh) - \mval{I_h \fd(\Udh)}{\Gamma_h(t)}\right).
	\end{multline*}
	It is a straightforward computation to see that
	\begin{align*}
		\intm \left(I_h\fd(\Udh), I_h \fd (\Udh) - \mval{I_h \fd(\Udh)}{\Gamma_h(t)}\right) = \normh{t}{I_h\fd(\Udh)- \mval{I_h \fd(\Udh)}{\Gamma_h(t)}}^2.
	\end{align*}
	Now recalling the assumption that we have an evolving acute triangulation of $\Gamma_h(t)$ we use \eqref{acuteineq} to see that $ a_{h}(\Udh,I_h \fd (\Udh)) \geq 0$, and hence
	\begin{multline*}
		\normh{t}{I_h\fd(\Udh)- \mval{I_h \fd(\Udh)}{\Gamma_h(t)}}^2 \leq \frac{2}{\theta} \intm\left(\Udh,I_h \fd (\Udh) - \mval{I_h \fd(\Udh)}{\Gamma_h(t)}\right)\\
		+ \frac{2\varepsilon}{\theta} \intm\left(\Wdh - \mval{\Wdh}{\Gamma_h(t)} ,I_h \fd (\Udh) - \mval{I_h \fd(\Udh)}{\Gamma_h(t)}\right).	
	\end{multline*}
	Now by using a Young's inequality, and \eqref{intmbound0} one finds that
	 \[\normh{t}{I_h\fd(\Udh)- \mval{I_h \fd(\Udh)}{\Gamma_h(t)}}^2 \leq \frac{4}{\theta^2} \| \Udh \|_{L^2(\Gamma_h(t))}^2 + \frac{4\varepsilon^2}{\theta^2} \left\| \Wdh - \mval{\Wdh}{\Gamma_h(t)} \right\|_{L^2(\Gamma_h(t))}^2.\]
	Now by using the Poincar\'e inequality, integrating over $[0,T]$ and using the bounds from \eqref{discregbound} proves the first result.
\end{proof}
To prove the second result one first shows a uniform bound on $\mval{I_h \fd(\Udh)}{\Gamma_h(t)}$, and this requires the following preliminary result.

\begin{lemma}[{\cite[Lemma 2.4]{copetti1992numerical}}]
	\label{measure}
	Let	$\eta_h \in L^1(\Gamma_h(t))$ be such that there exist $0 < \xi, \xi' < 1$ such that
	\begin{gather*}
		\left| \mval{\eta_h}{\Gamma_h(t)} \right| < 1 - \xi,\\
		\frac{1}{|\Gamma_h(t)|} \left(\int_{\Gamma_h(t)} [-1-\eta_h]_+ + \int_{\Gamma_{h}(t)} [\eta_h - 1]_+ \right) < \xi'.
	\end{gather*}
	We define the sets
	$$ \Gamma_{h,\xi}^+(t) := \left\{ x \in \Gamma_h(t) \mid \eta_h > 1 - \frac{\xi}{2} \right\}, \quad \Gamma_{h,\xi}^-(t) := \left\{ x \in \Gamma_h(t) \mid \eta_h < - 1 + \frac{\xi}{2} \right\}.$$
	Then if $\xi' < \frac{\xi^2}{4}$ we have
	\begin{align*}
		|\Gamma_{h,\xi}^+(t)| < \left( 1 - \frac{\xi}{2} \right) |\Gamma_h(t)|, \quad |\Gamma_{h,\xi}^-(t)| < \left( 1 - \frac{\xi}{2} \right) |\Gamma_h(t)|.
	\end{align*}
\end{lemma}

For $\eta_h = \Udh$, we find that the first condition holds as $U_{h,0} \in \mathcal{I}_{h,0}$, and hence by using the mass conservation property 
\[\left|\mval{\Udh(t)}{\Gamma_h(t)}\right| = \frac{|\Gamma_h(0)|}{|\Gamma_h(t)|}\left|\mval{U_{h,0}}{\Gamma_h(0)}\right| < (1 - \xi)\frac{|\Gamma_h(0)|}{|\Gamma_h(t)|} \leq 1-\xi,\]
for some $\xi > 0$.
Here we have used the fact that $|\Gamma_h(0)| \leq |\Gamma_h(t)|$ since $\gradgh \cdot V_h \geq 0$.
The second condition holds for $\delta$ sufficiently small by using \eqref{udhextrem}.

\begin{proof}[\underline{Proof of Lemma \ref{dpotbound}, part 2}]
	\ \newline
	We note that as our initial data is admissible there exists $\xi > 0$ such that
	\[\left|\mval{\Udh}{\Gamma_h(t)}\right| =  \frac{1}{|\Gamma_h(t)|}\left|\int_{\Gamma_h(0)}U_{h,0}\right| \leq \frac{1}{|\Gamma_h(0)|}\left|\int_{\Gamma_h(0)}U_{h,0}\right| < 1- \xi.\]
	Now by taking $\delta$ sufficiently small we may use Lemma \ref{measure} to find that
	\begin{align*}
		\mval{I_h \fd(\Udh)}{\Gamma_h(t)} &=  \frac{1}{|\Gamma_h(t)|} \int_{\left\{\Udh(t)\leq 1 - \frac{\xi}{2}\right\}} I_h \fd(\Udh) +  \frac{1}{|\Gamma_h(t)|} \int_{\left\{\Udh(t) > 1 - \frac{\xi}{2}\right\}} I_h \fd(\Udh)\\
		&\leq \fd \left(1 - \frac{\xi}{2}\right) + \frac{\left( 1 - \frac{\xi}{2}\right)^{\frac{1}{2}} }{|\Gamma_h(t)|^{\frac{1}{2}}} \|I_h \fd(\Udh)\|_{L^2(\Gamma_h(t))}.
	\end{align*}
	One can similarly show that
	$$ \mval{I_h \fd(\Udh)}{\Gamma_h(t)} \geq -\fd \left(1 - \frac{\xi}{2}\right) - \frac{\left( 1 - \frac{\xi}{2}\right)^{\frac{1}{2}} }{|\Gamma_h(t)|^{\frac{1}{2}}} \|I_h \fd(\Udh)\|_{L^2(\Gamma_h(t))} .$$
	Now by combining these bounds we observe
	$$ \left( \mval{I_h \fd(\Udh)}{\Gamma_h(t)} \right)^2 \leq \left( \fd \left(1 - \frac{\xi}{2}\right) + \frac{\left( 1 - \frac{\xi}{2}\right)^{\frac{1}{2}} }{|\Gamma_h|^{\frac{1}{2}}} \|I_h \fd(\Udh)\|_{L^2(\Gamma_h(t))} \right)^2 . $$
	Noting that, for $a,b \in \mathbb{R}$, $(a+b)^2 \leq a^2 \left( 1 + \frac{2}{\xi}
	\right) + b^2 \left( 1 + \frac{\xi}{2}
	\right)$ one obtains
	\begin{equation}\label{dpotbound3}
		\left( \mval{I_h \fd(\Udh)}{\Gamma_h(t)} \right)^2 \leq \left( 1 + \frac{2}{\xi} \right) \fd\left( 1 - \frac{\xi}{2} \right)^2 + \left( \frac{4 - \xi^2}{4 |\Gamma_h(t)|} \right) \|I_h \fd(\Udh)\|_{L^2(\Gamma_h(t))}^2 .
	\end{equation}
	Clearly we have that
	$$ \frac{1}{|\Gamma_h(t)|} \left\| I_h \fd(\Udh) - \mval{I_h \fd(\Udh)}{\Gamma_h(t)} \right\|_{L^2(\Gamma_h(t))}^2 = \frac{1}{|\Gamma_h(t)|} \|I_h \fd(\Udh)\|_{L^2(\Gamma_h(t))}^2 - \left(\mval{I_h\fd(\Udh)}{\Gamma_h(t)}\right)^2,$$
	and hence using \eqref{dpotbound3},
	\begin{multline}\label{dpotbound4}
		\frac{\xi^2}{ 4|\Gamma_h(t)|} \| I_h \fd(\Udh) \|_{L^2(\Gamma_h(t))}^2 \leq \frac{1}{|\Gamma_h(t)|} \left\| I_h \fd(\Udh) - \mval{I_h \fd(\Udh)}{\Gamma_h(t)}\right\|_{L^2(\Gamma_h(t))}^2 + \left( 1 + \frac{2}{\xi} \right) \fd\left(1 - \frac{\xi}{2}\right)^2 .
	\end{multline} 
	Now by construction, we can take $\delta$ to be sufficiently small so that $\fd\left( 1 - \frac{\xi}{2} \right) = f\left( 1 - \frac{\xi}{2} \right)$, which removes the dependence on $\delta$.
	The result readily follows by integrating \eqref{dpotbound4} over $[0,T]$ and using \eqref{dpotbound1} gives \eqref{dpotbound2}.
\end{proof}
It is a straightforward application of the Poincar\'e inequality to see that our established bounds imply a bound
$$ \int_0^T \| \Wdh \|_{H^1(\Gamma_h(t))}^2 \leq C,$$
where $C$ is independent of $\delta, h$.
This follows by obtaining a uniform bound on $\int_0^T\left(\mval{\Wdh}{\Gamma_h(t)} \right)^2$ by testing \eqref{discregch2} with $\phi_h = 1$, and using \eqref{discregbound}.
\subsubsection{Passage to the limit}
As a consequence of the established uniform bounds there exist functions $U_h \in L^{\infty}_{H^1},\matdev_h U_h \in L^2_{H^{-1}},W_h \in L^2_{H^1},\Phi_h \in L^2_{L^2}$, such that as (up to a subsequence of) $\delta \searrow 0$ we have
\begin{gather*}
	\Udh \overset{*}{\rightharpoonup} U_h \text{ weak-}* \text{ in } L^\infty_{H^1} \qquad \matdev_h{\Udh} \rightharpoonup \matdev_h{U_h} \text{ weakly in } L^2_{H^{-1}},\\
	\Wdh \rightharpoonup W_h \text{ weakly in } L^2_{H^1}, \qquad I_h \fd\left( \Udh \right) \rightharpoonup \Phi_h \text{ weakly in } L^2_{L^2},
\end{gather*}
with $U_h \in S_h^T$ and $W_h, \Phi_h \in \tilde{S}_h^T$.
Moreover the Aubin-Lions type result of \cite{alphonse2023function} allows one to obtain a strongly convergent subsequence $\Udh \rightarrow U_h$ in $L^2_{L^2}$.

We now show that $|U_h|<1$ a.e. on $\Gamma_h(t)$ (for almost all $t \in [0,T]$), and $\Phi_h = I_h f(U_h)$.
As a result of Lemma \ref{udhcontrol}, taking $\delta \rightarrow 0$ we obtain the following.
\begin{corollary}
	For almost all $t \in [0,T]$ we have that $|U_h(t)| \leq 1$ almost everywhere on $\Gamma_h(t)$.
\end{corollary}

Due to the poles of $f$ at $\pm 1$, we need to strengthen this result to make sure that the set of values where $|U_h(t)| = 1$ has measure zero.
We note that since $U_h$ is piecewise linear it is sufficient to show that we have $|U_h(x_i(t),t)| < 1$ for the nodes $(x_i(t))_{i=1,...,N_h} \subset \Gamma_h(t)$ and almost all $ t \in [0,T]$.
By the strong convergence $\Udh \rightarrow U_h$ in $L^2_{L^2}$ we find that for almost all $x \in \Gamma_h(t)$ and almost all $t \in [0,T]$ that
$$\lim_{\delta \searrow 0} \Udh(x,t) = U_h(x,t).$$
We show that this in fact holds for all $x \in \Gamma_h(t)$.

\begin{lemma}
	\label{conveverywhere}
	Let $t \in [0,T]$ be such that 
	$$ \lim_{\delta \searrow 0} \Udh(x,t) = U_h(x,t) $$
	for almost all $x \in \Gamma_h(t)$.
	Then this holds for all $x \in \Gamma_h(t)$.
\end{lemma}
\begin{proof}
	We fix $t$ such that the assumption holds.
	A straightforward application of the triangle inequality shows that pointwise convergence at the nodes of $\Gamma_h(t)$ implies pointwise convergence everywhere.
    Hence we verify that pointwise convergence holds at all of the nodes.
	For a given node $x_i \in \Gamma_h(t)$ of the triangulation $\mathcal{T}_h(t)$ one has that
    \[ \left| \Udh(x_i) - U_h(x_i) \right| \leq \left\| \Udh - U_h \right\|_{L^\infty(\Gamma_h(t))} \leq Ch^{-2}\left\| \Udh - U_h \right\|_{L^1(\Gamma_h(t))}, \]
    where we have used an inverse inequality.
    Since $\Udh$ is uniformly bounded (independent of $\delta$) due to \eqref{discregbound} one may apply the dominated convergence theorem so that pointwise convergence almost everywhere implies convergence in $L^1(\Gamma_h(t))$ from which the result follows.
\end{proof}

This is a necessary prerequisite for us to adapt an argument from \cite{caetano2021cahn}.

\begin{lemma}
	\label{uh<1}
	For almost all $t \in [0,T]$ one has that $|U_h(x,t)| < 1$ everywhere on $\Gamma_h(t)$.
\end{lemma}
\begin{proof}
	Firstly, we note that one can argue along the same lines as in \cite[Lemma 5.10]{caetano2021cahn} to show the following.
	Up to a subsequence of $\delta \searrow 0$ we have for all $x \in \Gamma_h(t)$, and almost all $t \in [0,T]$
	\begin{align}
		\label{limphid}
		\lim_{\delta \searrow 0} I_h \fd \left(\Udh(x,t) \right) = \begin{cases}
			I_h f(U_h(x,t)), & \text{ if } |U_h(x,t)| < 1\\
			\infty, & \text{ otherwise}
		\end{cases}.
	\end{align}
	We note that for our purposes this requires the convergence on all of $\Gamma_h(t)$ as discussed in the previous lemma.
	Then we recall the bound \eqref{dpotbound2}, and use an inverse inequality to see that
	$$ \int_0^T \|I_h\fd\left( \Udh(t) \right)\|_{L^{\infty}(\Gamma_h(t))}^2 \leq Ch^{-2} \int_0^T \|I_h \fd\left( \Udh(t) \right)\|_{L^{2}(\Gamma_h(t))}^2 \leq Ch^{-2}, $$
	where the constant $C$ is independent of $\delta$ and $h$.
	Taking the $\liminf$ in $\delta$ of the above and using Fatou's lemma we find that
	\begin{align*}
		\int_0^T \liminf_{\delta \searrow 0}| \fd\left( \Udh(x_i,t) \right)|^2 &\leq \int_0^T \liminf_{\delta \searrow 0} \|I_h \fd\left( \Udh(t) \right)\|_{L^{\infty}(\Gamma_h(t))}^2\\
		&\leq \liminf_{\delta \searrow 0} \int_0^T \|I_h \fd\left( \Udh(t) \right)\|_{L^{\infty}(\Gamma_h(t))}^2 \leq Ch^{-2},
	\end{align*}
	for all the nodes $x_i \in \Gamma_h(t)$.
	Hence we have
	$$  \liminf_{\delta \searrow 0}|\fd\left( \Udh(x_i,t) \right)| < \infty,$$
	for almost all $t \in [0,T]$, and hence by \eqref{limphid} that $|U_h(x_i,t)| < 1$ for all nodes $x_i$.
	Since this holds for all of the nodes, it holds for all $x \in \Gamma_h(t)$ as $U_h$ is piecewise linear.
\end{proof}
As a consequence of this result, we also find that as $\Phi_h = I_h f(U_h)$.
Thus a solution pair $(\Udh, \Wdh)$ solving \eqref{discregch}, \eqref{discregch2} yields a solution pair $(U_h, W_h)$ solving \eqref{fecheqn1}, \eqref{fecheqn2} as $\delta \searrow 0$.
We note that lower semi-continuity of the norm implies a bound of the form
	\begin{align}
		\label{potbound}
		\int_0^T \| I_h f(U_h)\|_{L^2(\Gamma_h(t))}^2 \leq C,
	\end{align}
	for a constant $C$ independent of $h$.

Lastly state a result showing continuous dependence on the initial data, from which we obtain uniqueness.
\begin{proposition}
	\label{festability}
	Let $U^1_{h,0}, U^2_{h,0} \in \mathcal{I}_{h,0}$ be admissible initial data such that $\mval{U^1_{h,0}}{\Gamma_h(0)} = \mval{U^2_{h,0}}{\Gamma_h(0)}$.
	Denoting the corresponding solution pairs as $(U^i_h,W_h^i)$ with initial data $U_h^i(0) = U^i_{h,0}$ we have
	\[\inormh{U^1_h(t) - U^2_h(t)}^2 + \varepsilon \int_0^t \|\gradgh(U_h^1 - U_h^2)\|_{L^2(\Gamma_h(t))}^2  \leq C\inormh{ U^1_{h,0} - U^2_{h,0}}^2,\]
	where $C$ is independent of $h$ and depends exponentially on $t$.
\end{proposition}

\begin{proof}
    We omit this proof as it is essentially identical to \cite[Proposition 4.12]{caetano2021cahn}, where the only meaningful difference is we must use $\intinvsh$ rather than $\mathcal{G}$ as in \cite{caetano2021cahn} (see Appendix \ref{invlaps} for the definition of these operators).
\end{proof}

\subsection{Error analysis}

In this section we consider the error analysis of the pair $(U_h, W_h)$, with initial data $U_{h,0} = \Pi_h u_0$, where $u_0 \in \mathcal{I}_{0}\cap H^2(\Gamma_0)$.
We now show that $U_{h,0} \in \mathcal{I}_{h,0}$ and so the preceding theory holds true.
\begin{lemma}
	\label{initialdatalemma}
	Let $u_0 \in \mathcal{I}_0 \cap H^2(\Gamma_0)$ be such that $|u_0| < 1$.
	Then, for sufficiently small $h$, $\Pi_h u_0\in \mathcal{I}_{h,0}$.
\end{lemma}
\begin{proof}
	By definition we have that
	\[\left|\mval{\Pi_h u_0}{\Gamma_h(0)}\right| = \frac{|\Gamma_0|}{|\Gamma_h(0)|} \left|\mval{u_0}{\Gamma_0} \right| \leq (1-\xi)\frac{|\Gamma_0|}{|\Gamma_h(0)|} < 1,\]
	for sufficiently small $h$.
	Similarly by using an inverse inequality and \eqref{ritz3}, \eqref{interpolation} one finds
	\begin{align*}
		\| \Pi_h u_0 \|_{L^{\infty}(\Gamma_h(0))} &\leq \|  I_h u_0^{-\ell} \|_{L^{\infty}(\Gamma_h(0))} + \| \Pi_h u_0 - I_h u_0^{-\ell} \|_{L^{\infty}(\Gamma_h(0))}\\
		&< 1 - \xi + Ch^{-1} \| \Pi_h u_0 - I_h u_0^{-\ell} \|_{L^{2}(\Gamma_h(0))}\\
		&< 1 - \xi + Ch\|u_0\|_{H^2(\Gamma_0)} < 1.
	\end{align*}
	With this we now find that
	\begin{align*}
		\Ech_h[\Pi_h u_0] = \int_{\Gamma_h(t)} \frac{\varepsilon|\gradgh \Pi_h u_0|^2}{2} + \frac{1}{\varepsilon} I_h F(\Pi_h u_0) \leq C + C \|u_0\|_{H^1(\Gamma(t))}^2,
	\end{align*}
	where we have used \eqref{ritz1}, and the boundedness of $F$ over $[-1,1]$.
\end{proof}

We adapt ideas from \cite{barrett1995error,barrett1996error} to show an error bound for the semi-discrete scheme \eqref{fecheqn1}, \eqref{fecheqn2}.
Here $u^\delta$ is the solution of the regularised continuous problem, as given as in the proof of \eqref{deltaerror1}.
For the error analysis, we require that $u^\delta$ is such that
\begin{align}
	\int_0^T \left[ \|u^\delta\|_{H^2(\Gamma(t))}^2 + \|w^\delta\|_{H^1(\Gamma(t))}^2 +  \|\fd(u^\delta)\|_{L^2(\Gamma(t))}^2 + \|\matdev u^\delta\|_{H^2(\Gamma(t))}^2 \right] < C, \label{errorrequirement}
\end{align}
for some constant $C$ independent of $\delta$.
All of these bounds are known to hold, except for the derivative bound, which is more problematic.
A derivative bound of this form is assumed in \cite{beschle2022stability,elliott2015evolving,elliott2024fully} for the error analysis of ESFEM schemes for the Cahn-Hilliard equation with a regular potential.

To show an error bound, we decompose the error as
$$ u^{-\ell} - U_h = [u^{-\ell} - (u^{\delta})^{-\ell}] + [(u^{\delta})^{-\ell} - \Udh] + [\Udh - U_h],$$
where the first term is bounded by using lifts and \eqref{deltaerror1}, and the third term is bounded similarly as in the following theorem.

\begin{theorem}
	Let $(\Udh,\Wdh)$ be the unique solution of \eqref{discregch}, \eqref{discregch2}, and $(U_h, W_h)$ be the unique solution of \eqref{fecheqn1}, \eqref{fecheqn2}.
	Then for sufficiently small $\delta >0$ we have that
	\begin{align}
		\varepsilon \int_0^T \| \gradgh (U_h - \Udh) \|_{L^2(\Gamma_h(t))}^2 + \sup_{t \in [0,T]} \inormh{U_h - \Udh}^2 \leq C \delta, \label{deltaerror2}
	\end{align}
	for $C$ a constant independent of $\delta, h$.
\end{theorem}
\begin{proof}
	This bound follows almost identically to \eqref{deltaerror1}, where we use \eqref{potbound} to bound the potential term.
	We omit the details.
\end{proof}
Thus the main goal of this subsection is bounding the error for the regularised terms, $(u^{\delta})^{-\ell} - \Udh$.
We consider the same idea as in \cite{elliott2015evolving}, and decompose the error as
\[ (u^{\delta})^{-\ell} - \Udh = \underbrace{(u^{\delta})^{-\ell} - \Pi_h u^{\delta}}_{=: \rho_u^{\delta}} + \underbrace{\Pi_h u^{\delta} - \Udh}_{=: \sigma_u^{\delta}}, \quad (w^{\delta})^{-\ell} - \Wdh = \underbrace{(w^{\delta})^{-\ell} - \Pi_h w^{\delta}}_{=: \rho_w^{\delta}} + \underbrace{\Pi_h w^{\delta} - \Wdh}_{=: \sigma_u^{\delta}}.\]
We observe, from our assumption \eqref{errorrequirement}, that by applying \eqref{ritz3} we find
\begin{align}
	\label{deltaerror3}
	\int_0^T \|\rho_u^\delta\|_{L^2(\Gamma_h(t))}^2 + h^2 \|\gradgh \rho_u^\delta\|_{L^2(\Gamma_h(t))}^2 \leq C h^4,
\end{align}
for a constant $C$ independent of $\delta,h$, and hence it only remains to bound $\sigma_u^\delta$.

We observe that by using both versions of transport theorem
\begin{align*}
	m\left(\matdev{u^\delta} , \phi_h^\ell \right) + g(u^\delta, \phi_h^\ell)&= \frac{d}{dt} m(u^\delta, \phi_h^\ell) - m(u^\delta, \matdev \phi_h^\ell)\\
	&=m(\matdev_\ell u^\delta, \phi_h^\ell) +g_\ell(u^\delta, \phi_h^\ell) + m(u^\delta, \matdev_\ell\phi_h^\ell - \matdev \phi_h^\ell)
\end{align*}
and hence $u^{\delta}, w^\delta$ solve
\begin{gather*}
	m\left(\matdev_\ell{u^\delta} , \phi_h^\ell \right) + g_\ell(u^\delta, \phi_h^\ell)+ a (w^\delta, \phi_h^\ell) = m(u^\delta, \matdev\phi_h^\ell - \matdev_\ell \phi_h^\ell),\\ 
	m(w^\delta, \phi_h^\ell) = \varepsilon a(u^\delta,\phi_h^\ell) + \frac{\theta}{2 \varepsilon} m (\fd(u^\delta), \phi_h^\ell) - \frac{1}{\varepsilon} m(u^\delta,\phi_h^\ell).
\end{gather*}
Now as $\Udh, \Wdh$ solve \eqref{discregch}, \eqref{discregch2} one obtains, using the definition of the Ritz projection that $\sigma_u^\delta, \sigma_w^\delta$ solve
\begin{gather}
	\intm(\matdev_h \sigma_u^\delta, \phi_h) + g_h(I_h(\sigma_u^\delta\phi_h),1) + a_h(\sigma_w^\delta, \phi_h) = \sum_{i=1}^3 E_i(\phi_h),\label{errorpf1}\\
	\intm(\sigma_w^\delta, \phi_h) = \varepsilon a_h(\sigma_u^\delta, \phi_h) + \intm(I_h\fd(\Pi_h u^\delta) - I_h\fd(\Udh),\phi_h) - \frac{1}{\varepsilon} \intm(\sigma_u^\delta, \phi_h) + \sum_{i=4}^6 E_i(\phi_h),\label{errorpf2}
\end{gather}
where $E_i(\phi_h)$ denote some consistency errors, given by
\begin{gather*}
	E_1(\phi_h) := \intm(\matdev_h \Pi_h u^\delta, \phi_h) - m(\matdev u^\delta, \phi_h^\ell),\\
	E_2(\phi_h) := g_h(I_h(\Pi_h u^\delta\phi_h),1) - g_\ell(u^\delta, \phi_h^\ell) ,\\
	E_3(\phi_h) := m(u^\delta, \matdev\phi_h^\ell - \matdev_\ell \phi_h^\ell),\\
	E_4(\phi_h) := \intm(\Pi_h w^\delta, \phi_h) - m(w^\delta, \phi_h^\ell),\\
	E_5(\phi_h) := \frac{\theta}{2 \varepsilon}\left[m(\fd(u^\delta), \phi_h^\ell) - \intm(I_h\fd(\Pi_h u^\delta), \phi_h)\right],\\
	E_6(\phi_h) := \frac{1}{\varepsilon}\left[\intm(\Pi_h u^\delta, \phi_h) - m(u^\delta, \phi_h^\ell) \right].
\end{gather*}
We will be interested in taking $\phi_h$ to be $\sigma_u^\delta$ or $\intinvsh \sigma_u^\delta$, where we note that the latter is well-defined\footnote{This follows as $\int_{\Gamma_h(t)} U_h^\delta = \int_{\Gamma_h(0)} \Pi_h u_0 = \int_{\Gamma(0)} u_0 =  \int_{\Gamma(t)} u = \int_{\Gamma_h(t)} \Pi_h u$, by the mass-conserving properties of both the Ritz projection and solutions of the Cahn-Hilliard equation.}, and bounding the corresponding $E_i(\phi_h)$.
This is the content of the following lemma.

\begin{lemma}
	\label{errorlemma}
	For $E_1,...,E_6$ defined as above we have
	\begin{gather*}
		|E_1(\intinvsh \sigma_u^\delta)| \leq Ch^2 \left( \| u^\delta\|_{H^2(\Gamma(t))} + \|\matdev u^\delta\|_{H^2(\Gamma(t))} \right)\|\intinvsh \sigma_u^\delta\|_{L^2(\Gamma_h(t))},\\
		|E_2(\intinvsh \sigma_u^\delta)| \leq Ch^2\|u^\delta\|_{H^2(\Gamma(t))}\|\intinvsh \sigma_u^\delta\|_{L^2(\Gamma(t))},\\
		|E_3(\intinvsh \sigma_u^\delta)| \leq Ch^2\|u^\delta\|_{L^2(\Gamma(t))} \inormh{\sigma_u^\delta},\\
		|E_4(\sigma_u^\delta)| \leq Ch \|w^\delta\|_{H^1(\Gamma(t))} \|\sigma_u^\delta\|_{L^2(\Gamma_h(t))},\\
		|E_5(\sigma_u^\delta)|\leq \left(\frac{Ch^2 \log\left( \frac{1}{h} \right)^{\frac{1}{2}}}{\delta} \|u^\delta\|_{H^2(\Gamma(t))} + Ch^2 \|\fd(u^\delta)\|_{L^2(\Gamma(t))}\right)\|\gradgh\sigma_u^\delta\|_{L^2(\Gamma_h(t))},\\
		|E_6(\sigma_u^\delta)| \leq Ch^2\|u^\delta\|_{H^2(\Gamma(t))} \|\gradg \sigma_u^\delta\|_{L^2(\Gamma_h(t))}.
	\end{gather*}
\end{lemma}
\begin{proof}
	We begin by writing
	\begin{align*}
		E_1(\intinvsh \sigma_u^\delta) &= \left[ \intm(\matdev_h \Pi_h u^\delta, \intinvsh \sigma_u^\delta) - m_h(\matdev_h \Pi_h u^\delta, \intinvsh \sigma_u^\delta) \right]\\
		&+ \left[m_h(\matdev_h \Pi_h u^\delta, \intinvsh \sigma_u^\delta) - m(\matdev_\ell \pi_h u^\delta, (\intinvsh \sigma_u^\delta)^\ell)\right]\\
		&+\left[m(\matdev_\ell (\pi_h u^\delta - u^\delta), (\intinvsh \sigma_u^\delta)^\ell)\right] +\left[m(\matdev_\ell u^\delta - \matdev u^\delta, (\intinvsh \sigma_u^\delta)^\ell)\right],
	\end{align*}
	and hence using \eqref{derivativedifference1}, \eqref{perturb1}, \eqref{ritzddtnorm}, \eqref{ritzddt}, \eqref{intmbound2}, and \eqref{lift1} where necessary, one concludes the bounds for $|E_1(\intinvsh \sigma_u^\delta)|$.\\
	
	Next we write $E_2(\intinvsh \sigma_u^\delta)$ as
	\begin{align*}
		E_2(\intinvsh \sigma_u^\delta) &= \left[ g_h(I_h(\Pi_h u^\delta\intinvsh \sigma_u^\delta),1) -  g_h(\Pi_h u^\delta, \intinvsh \sigma_u^\delta) \right]\\
		&+\left[ g_h(\Pi_h u^\delta, \intinvsh \sigma_u^\delta) - g_\ell(\pi_h u^\delta, (\intinvsh \sigma_u^\delta)^\ell) \right] + g_\ell(\pi_h u^\delta - u^\delta, (\intinvsh \sigma_u^\delta)^\ell),
	\end{align*}
	from which the stated bound follows from \eqref{perturb3}, \eqref{ritz1}, \eqref{ritz3} and (an appropriate analogue of) \eqref{intmbound2}.
	The bound for $E_3(\intinvsh \sigma_u^\delta)$ follows immediately from \eqref{derivativedifference1} and the Poincar\'e inequality.
	Likewise the bound for $E_4(\sigma_u^\delta)$ follows by similar arguments to the above, using \eqref{perturb1}, \eqref{ritz2}, \eqref{intmbound1}.
	The last of the simple bounds is for $E_6(\sigma_u^\delta)$, where we write
	\begin{multline*}
		E_6(\sigma_u^\delta) = \frac{1}{\varepsilon}\left[\intm(\Pi_h u^\delta, \sigma_u^\delta) - m_h(\Pi_h u^\delta, \sigma_u^\delta) \right] + \frac{1}{\varepsilon}\left[m_h(\Pi_h u^\delta, \sigma_u^\delta) - m(\pi_h u^\delta, (\sigma_u^\delta)^\ell)\right]\\
        +\frac{1}{\varepsilon} m(\pi_h u^\delta - u^\delta, (\sigma_u^\delta)^\ell),
	\end{multline*}
	and use \eqref{perturb1}, \eqref{ritz3}, \eqref{intmbound2}.
    We note here that since $\int_{\Gamma_h(t)} \sigma_u^\delta = 0$, we can use the Poincar\'e inequality here so that $\|\sigma_u^\delta\|_{L^2(\Gamma_h(t))} \leq C\|\gradgh \sigma_u^\delta\|_{L^2(\Gamma_h(t))} $ --- we shall do this throughout.
	
	The bound for $E_5(\sigma_u^\delta)$ naturally is the most complicated due to the singular potential.
	We begin by bounding $E_5(\sigma_u^\delta)$ by
	\begin{multline*}
		|E_5(\sigma_u^\delta)| \leq \frac{\theta}{2 \varepsilon}\left|\intm(I_h \fd(\Pi_h u^\delta), \sigma_u^\delta) - m_h(I_h\fd(\Pi_h u^\delta), \sigma_u^\delta) \right| + \frac{\theta}{2 \varepsilon} |m_h(I_h\fd(\Pi_h u^\delta) - I_h\fd((u^\delta)^{-\ell}), \sigma_u^\delta)|\\
		 +\frac{\theta}{2 \varepsilon} \left|m_h(I_h\fd((u^\delta)^{-\ell}), \sigma_u^\delta) - m(I_h^\ell\fd(u^\delta), (\sigma_u^\delta)^\ell)\right| + \frac{\theta}{2 \varepsilon}|m(I_h^\ell\fd(u^\delta) - \fd(u^\delta), (\sigma_u^\delta)^\ell)|.
	\end{multline*}
	Now by using \eqref{intmbound2} one finds
	\[\left| \intm(I_h\fd(\Pi_h u^\delta), \sigma_u^\delta) - m_h(I_h\fd(\Pi_h u^\delta), \sigma_u^\delta) \right| \leq Ch^2 \| \gradgh I_h \fd(\Pi_h u^\delta)\|_{L^2(\Gamma_h(t))} \|\gradgh\sigma_u^\delta\|_{L^2(\Gamma_h(t))},\]
	and so it is apparent we need a bound for $\| \gradgh I_h \fd(\Pi_h u^\delta)\|_{L^2(\Gamma_h(t))}$.
	To do this we use \eqref{acuteineq} to see that
	\[ \| \gradgh I_h \fd(\Pi_h u^\delta)\|_{L^2(\Gamma_h(t))}^2 \leq \frac{C}{\delta} a_h(I_h \fd(\Pi_h u^\delta), \Pi_h u^\delta), \]
	whereby Young's inequality and \eqref{ritz1} yields
	\begin{align}
		\| \gradgh I_h \fd(\Pi_h u^\delta)\|_{L^2(\Gamma_h(t))} \leq \frac{C}{\delta} \|\gradgh \Pi_h u^\delta \|_{L^2(\Gamma_h(t))} \leq \frac{C}{\delta} \|u^\delta\|_{H^1(\Gamma(t))}. \label{phid gradient bound}
	\end{align}
	
	Next we observe that
	\[m_h(I_h\fd(\Pi_h u^\delta) - I_h\fd((u^\delta)^{-\ell}), \sigma_u^\delta) = m_h(I_h\fd(\Pi_h u^\delta) - I_h\fd(I_h(u^\delta)^{-\ell}), \sigma_u^\delta)\]
	and it is straightforward to see that, by using $|(\fd)'(r)| \leq \frac{C}{\delta}$
	\begin{align*}
		\| I_h\fd(\Pi_h u^\delta) - I_h\fd(I_h(u^\delta)^{-\ell}) \|_{L^2(\Gamma(t))} &\leq C \normh{t}{I_h\fd(\Pi_h u^\delta) - I_h\fd(I_h(u^\delta)^{-\ell})}\\
		&\leq \frac{C}{\delta} \|\Pi_h u^\delta - I_h(u^\delta)^{-\ell} \|_{L^2(\Gamma(t))} \leq \frac{Ch^2}{\delta} \|u^\delta \|_{H^2(\Gamma(t))},
	\end{align*}
	where we have used \eqref{ritz3}, \eqref{interpolation} in the final inequality.
	From \eqref{lift1} and \eqref{perturb1} one finds
	\[|m_h(I_h\fd((u^\delta)^{-\ell}), \sigma_u^\delta) - m(I_h^\ell\fd(u^\delta), (\sigma_u^\delta)^\ell)| \leq Ch^2 \|I_h^\ell\fd(u^\delta)\|_{L^2(\Gamma(t))}\|\sigma_u^\delta\|_{L^2(\Gamma_h(t))},\]
	where we will bound $\|I_h^\ell\fd(u^\delta)\|_{L^2(\Gamma(t))}$ later.
	For the final term we use an argument from \cite{barrett1996error,barrett1997finite}.
    From H\"older's inequality we find
	\[|m(I_h^\ell\fd(u^\delta) - \fd(u^\delta), (\sigma_u^\delta)^\ell)|\leq C\|I_h^\ell\fd(u^\delta) - \fd(u^\delta)\|_{L^1(\Gamma(t))} \|\sigma_u^\delta\|_{L^\infty(\Gamma_h(t))}.\]
	Now by following an argument in the proof of \cite[Theorem 3.2]{barrett1996error} one can show that $\fd(u^\delta) \in H^{2,1}(\Gamma(t))$ such that
	\[\|\fd(u^\delta) \|_{H^{2,1}(\Gamma(t))} \leq \frac{C}{\delta}\|u^\delta\|_{H^{2,1}(\Gamma(t))} \leq \frac{C}{\delta}\|u^\delta\|_{H^{2}(\Gamma(t))}. \]
	Thus using \eqref{discrete sobolev} and \eqref{interpolation2} one finds that 
	\[\|I_h^\ell\fd(u^\delta) - \fd(u^\delta)\|_{L^1(\Gamma(t))} \|\sigma_u^\delta\|_{L^\infty(\Gamma_h(t))} \leq \frac{C h^2}{\delta} \log\left(\frac{1}{h}\right)^{\frac{1}{2}} \|u^\delta\|_{H^2(\Gamma(t))} \|\gradgh \sigma_u^\delta\|_{L^2(\Gamma_h(t))}.\]
	The final thing to note now is that by using \eqref{interpolation2}, and the above bound for $\|\fd(u^\delta) \|_{H^{2,1}(\Gamma(t))}$, it is clear that
	\[\|I_h^\ell\fd(u^\delta)\|_{L^2(\Gamma(t))} \leq C\|\fd(u^\delta)\|_{L^2(\Gamma(t))} + \frac{Ch}{\delta} \|u^\delta\|_{H^2(\Gamma(t))}. \]
	Piecing this together yields the bound for $E_5(\sigma_u^\delta)$.
\end{proof}

\begin{theorem}
	For $\sigma_u^{\delta}$ defined as above, and $\delta$ sufficiently small, then assuming \eqref{errorrequirement} holds we have
	\begin{align}
		\label{deltaerror4}
		\sup_{t \in [0,T]} \inormh{\sigma_u^{\delta}}^2 
 + \frac{\varepsilon}{2}\int_0^T \| \gradgh \sigma_u^\delta\|_{L^2(\Gamma_h(t))}^2 \leq C \left( h^2 + \frac{h^4 \log\left(\frac{1}{h}\right)}{\delta^2}  \right),
	\end{align}
	where $C$ depends on $\varepsilon, T, \theta, u_0$.
\end{theorem}
\begin{proof}
	We test \eqref{errorpf1} with $\intinvsh \sigma_u^\delta$ and \eqref{errorpf2} with $\sigma_u^\delta$ to see that
	\begin{multline*}
		\intm(\matdev_h \sigma_u^\delta, \intinvsh \sigma_u^\delta) + g_h(I_h(\sigma_u^\delta\intinvsh \sigma_u^\delta),1) + \varepsilon a_h(\sigma_u^\delta, \sigma_u^\delta) + \intm(I_h \fd(\Pi_h u^\delta) - I_h \fd(\Udh),\sigma_u^\delta)\\
     = \sum_{i=1}^3 E_i(\intinvsh \sigma_u^\delta)- \sum_{i=4}^6 E_i(\sigma_u^\delta)+ \frac{1}{\varepsilon} \intm(\sigma_u^\delta, \sigma_u^\delta).
	\end{multline*}
	One then writes (using Proposition \ref{transport3} and Lemma \ref{transport5})
	\begin{align*}
		\intm(\matdev_h \sigma_u^\delta, \intinvsh \sigma_u^\delta) + g_h(I_h(\sigma_u^\delta \intinvsh \sigma_u^\delta),1) &= \frac{d}{dt} \intm(\sigma_u^\delta, \intinvsh \sigma_u^\delta) - \intm(\sigma_u^\delta, \matdev_h \intinvsh \sigma_u^\delta)\\
		&= \frac{d}{dt} \inormh{\sigma_u^\delta}^2 - a_h(\intinvsh\sigma_u^\delta, \matdev_h \intinvsh \sigma_u^\delta)\\
		&= \frac{1}{2}\frac{d}{dt} \inormh{\sigma_u^\delta}^2 - \frac{1}{2}b_h(\intinvsh\sigma_u^\delta, \intinvsh \sigma_u^\delta).
	\end{align*}
	Now using this, and the monotonicity of $\fd(\cdot)$, in the above we find
	\begin{multline}
		\frac{1}{2}\frac{d}{dt} \inormh{\sigma_u^\delta}^2 + \varepsilon \|\gradgh \sigma_u^\delta\|_{L^2(\Gamma_h(t))}^2\\
		\leq \frac{1}{2}b_h(\intinvsh\sigma_u^\delta, \intinvsh \sigma_u^\delta) + \frac{1}{\varepsilon}\normh{t}{\sigma_u^\delta}^2 + \sum_{i=1}^3|E_i(\intinvsh \sigma_u^\delta)| + \sum_{i=4}^6|E_i(\sigma_u^\delta)|. \label{errorpf3}
	\end{multline}
	We note that from the smoothness of $V$ and \eqref{intmbound0} one can show
	\[\frac{1}{2}b_h(\intinvsh\sigma_u^\delta, \intinvsh \sigma_u^\delta) + \frac{1}{\varepsilon}\normh{t}{\sigma_u^\delta}^2 \leq C \inormh{\sigma_u^\delta}^2 + \frac{\varepsilon}{8} \|\gradgh \sigma_u^\delta\|_{L^2(\Gamma_h(t))}^2.\]
	All that is left is to use the bounds from Lemma \ref{errorlemma} and Young/Poincar\'e inequalities where necessary to see that
	\begin{multline*}
		\frac{d}{dt} \inormh{\sigma_u^\delta}^2 + \varepsilon \|\gradgh \sigma_u^\delta\|_{L^2(\Gamma_h(t))}^2 \leq C \inormh{\sigma_u^\delta}^2 + \frac{Ch^4 \log\left(\frac{1}{h}\right)}{\delta^2} \|u^\delta\|_{H^2(\Gamma(t))}^2\\
		+Ch^2 \|w^\delta\|_{H^1(\Gamma(t))}^2  + Ch^4\left( \|u^\delta\|_{H^2(\Gamma(t))}^2 + \|\matdev u^\delta\|_{H^2(\Gamma(t))}^2+ \|\fd(u^\delta)\|_{L^2(\Gamma(t))}^2 \right),
	\end{multline*}
	whence the result follows by applying the Gr\"onwall inequality --- noting that $\sigma_u^\delta(0) = 0$, and that we have used our assumption \eqref{errorrequirement}.
\end{proof}

\begin{remark}
 Inspecting this proof one finds that this still holds under the weaker assumption that
\[ \int_{0}^T \|\matdev u^\delta\|_{H^2(\Gamma(t))}^2 \leq \frac{C}{\delta^2}, \]
which is more likely to hold, as the non-regularised problem is known to have limited regularity properties and hence a bound on $\matdev u^\delta$ is likely to degenerate as $\delta \rightarrow 0$.
\end{remark}

Now that we have bounded all of the relevant terms, we obtain an error bound for the semi-discrete scheme.
\begin{theorem}
	\label{semidiscerrortheorem}
	Let $(u,w)$ be the unique solution of \eqref{cheqn1}, \eqref{cheqn2} with initial data $u_0 \in \mathcal{I}_0 \cap H^2(\Gamma_0)$ and $\|u_0\|_{L^\infty(\Gamma_0)} < 1$.
	Let $(U_h, W_h)$ the unique solution of \eqref{fecheqn1}, \eqref{fecheqn2} with initial data $U_{h,0} = \Pi_h u_0$.
	Then for sufficiently small $h$,
	\begin{align}
		\int_0^T \varepsilon \| \gradgh (u^{-\ell} - U_h) \|_{L^2(\Gamma_h(t))}^2 \leq C h^\frac{4}{3} \log\left( \frac{1}{h} \right). \label{deltaerror5}
	\end{align}
\end{theorem}
\begin{proof}
	By combining \eqref{deltaerror1}, \eqref{deltaerror2}, \eqref{deltaerror3}, \eqref{deltaerror4} we find that
	\[\int_0^T \varepsilon \| \gradgh (u^{-\ell} - U_h) \|_{L^2(\Gamma_h(t))}^2 \leq C \left( \delta + h^2 + \frac{h^4 \log\left( \frac{1}{h} \right)}{\delta^2}  \right).\]
	Now we choose $\delta = C(p) h^p$ for some value $p$ to be determined and some constant $C(p)$ dependent on $p$ so that $\delta$ is sufficiently small.
	The optimal choice of $p$ is then $p = \frac{4}{3}$ and gives the result.
\end{proof}

\begin{remark}

\begin{enumerate}
    \item We do not immediately obtain an $L^\infty_{H^{-1}}$ error bound, as we have for the individual results \eqref{deltaerror1}, \eqref{deltaerror2}, \eqref{deltaerror4}.
    As this requires some notion of how one can compare the norms $\|\cdot\|_{-1}, \inormh{\cdot}$.
    \item From Proposition \ref{festability} is clear that this result holds for initial data $\tilde{U}_{h,0} \in \mathcal{I}_{h,0}$ such that
    \[\int_{\Gamma_h(0)} \tilde{U}_{h,0} = \int_{\Gamma_0} u_0, \text{ and } \inormh{\Pi_h u_0 - \tilde{U}_{h,0}} \leq C h^\frac{2}{3} \log\left( \frac{1}{h} \right)^\frac{1}{2},\]
	for some constant $C$ independent of $h$.
	A notable example of this is choosing $\tilde{U}_{h,0} = I_h u_0^{-\ell} - \tilde{c},$
	where
	$$ \tilde{c} = \frac{1}{|\Gamma_h(0)|}\left( \int_{\Gamma_h(0)} I_h u_0^{-\ell} - \int_{\Gamma_0} u_0 \right). $$
	It is straightforward to show that $\tilde{U}_{h,0} \in \mathcal{I}_{h,0}$, and that $|\tilde{c}| = \mathcal{O}(h^2)$.
	This justifies our choice of using the interpolant as initial data in our numerical experiments --- as for sufficiently small $h$ this term is negligible.
\end{enumerate}
\end{remark}

\section{Full discretisation of the problem}
\label{section: logch fully discrete}
\subsection{Time discretisation}

We consider a backward Euler time discretisation, with timestep $\tau > 0$, of the system \eqref{fecheqn1}, \eqref{fecheqn2}, introducing further numerical integration terms, to give our fully discrete scheme.

\begin{notation}
	\begin{enumerate}
		\item We use the notation $t_n := n \tau$, $\Gamma_h^n := \Gamma_h(t_n)$, $S_h^n := S_h(t_n)$.
		\item For functions $\phi_h^{n-1} \in S_h^{n-1},\phi_h^{n} \in S_h^{n}$ we define $\overline{\phi_h^{n-1}} \in S_h^{n}$ and $\underline{\phi_h^n} \in S_h^{n-1}$ to be the functions with the same nodal values, but on the succeeding/preceding surfaces respectively.
		\item We define a fully discrete material time derivative for a sequence of functions $(\phi_h^n)_n$ for $\phi_h^n \in S_h^n$ by
		\[ \matdevtau \phi_h^n := \frac{1}{\tau} \left(\phi_h^n - \overline{\phi_h^{n-1}}\right) \in S_h^n. \]
	\end{enumerate}
	
\end{notation}

For simplicity we assume that $N_T := \frac{T}{\tau} \in \mathbb{N}$ --- the analysis in this section can be modified to include non-uniform timestep sizes and timestep sizes such that $N_T := \lfloor \frac{T}{\tau}\rfloor \neq \frac{T}{\tau}$, but we shall not consider this here.
With this notation we now pose our fully discrete scheme.
We let $U_{h,0} \in \mathcal{I}_{h,0}$ be an approximation for some $u_0 \in \mathcal{I}_0$, as in the semi-discrete case.
Then for $n \geq 1$ and data $(U_h^{n-1}, W_h^{n-1})$ we want to find $(U_h^n, W_h^n) \in S_h^n \times S_h^n$ such that
\begin{gather}\label{fullydisceqn1}
	\frac{1}{\tau} \left(\intm(t_n;U^n_h, \phi_h^n) - \intm(t_{n-1};U_h^{n-1}, \phi_h^{n-1})\right)+ a_h(t_n;W_h^n , \phi_h^n) = \intm(t_{n-1};U_h^{n-1}, \underline{\matdevtau \phi_h^n}),\\
	\intm(t_n;W_h^n, \phi_h^n) = \varepsilon a_h(t_n;U^n_h , \phi_h^n) + \frac{\theta}{2\varepsilon} \intm(t_n;I_h f(U^n_h),\phi_h^n) - \frac{1}{\varepsilon}\intm(t_n;U^n_h, \phi_h^n),	\label{fullydisceqn2}
\end{gather}
for all $\phi_h^{n-1} \in S_h^{n-1},\phi_h^n \in S_h^n$, and such that $U_h^0 = U_{h,0}$.
As in \cite{elliott2024fully}, we note that in our notation we can write \eqref{fullydisceqn1} as
\[ \frac{1}{\tau} \left(\intm(t_n;U^n_h, \phi_h^n) - \intm(t_{n-1};U_{h}^{n-1}, \underline{\phi_h^{n}})\right)+ a_h(W_h^n , \phi_h^n) =0, \]
which is a more natural formulation of this equality.
This form will be used throughout, notably in proving existence by a similar argument to that in \cite{elliott2024fully}.
From here on we shall omit the time argument from the bilinear forms, as the timestep will be clear from context.

\subsection{Well-posedness}
\subsubsection{Regularisation}
As before, we show well-posedness by considering the regularised potential \eqref{dpot} and the corresponding equations.
As such we are interested in the following regularised problem for $\delta \in (0,1)$.

Given initial data $U_{h,0} \in \mathcal{I}_{h,0}$, approximating some $u_0 \in \mathcal{I}_{0}$, we want to find $(\Udhn{n}, \Wdhn{n}) \in S_h^n \times S_h^n$ such that
\begin{gather}\label{regfullydisceqn1}
	\frac{1}{\tau} \left(\intm(\Udhn{n}, \phi_h^n) - \intm(\Udhn{n-1}, \underline{\phi_h^{n}})\right)+ a_h(\Wdhn{n} , \phi_h^n) = 0,\\
	\intm(\Wdhn{n}, \phi_h^n) = \varepsilon a_h(\Udhn{n} , \phi_h^n) + \frac{\theta}{2\varepsilon} \intm(I_h\fd(\Udhn{n}),\phi_h^n) - \frac{1}{\varepsilon}\intm(\Udhn{n}, \phi_h^n),	\label{regfullydisceqn2}
\end{gather}
for all $\phi_h^n \in S_h^n$, and such that $\Udhn{0} = U_{h,0}$.

\subsubsection{Existence}
To show existence, we consider the minimisation of an appropriate functional, as in \cite{copetti1992numerical,elliott1989cahn, elliott2024fully}.
To do so we first state some results from \cite{dziuk2012fully}.

\begin{lemma}[{\cite[Lemma 3.6]{dziuk2012fully}}]
	For $\phi_h^n \in S_h^n$, $\tau$ sufficiently small and $t \in [t_{n-1}, t_n]$ there exists a constant $C$ independent of $t, \tau, h$ such that
	\begin{align}
		\| \underline{\phi_h^n}(t) \|_{L^2(\Gamma_h(t))} &\leq C \| \phi_h^n \|_{L^2(\Gamma_h^n)}, \label{timenorm1}\\
		\| \gradgh \underline{\phi_h^n}(t) \|_{L^2(\Gamma_h(t))} &\leq C \| \gradgh \phi_h^n \|_{L^2(\Gamma_h^n)}, \label{timenorm2}
	\end{align}
	where $\underline{\phi_h^n}(t)$ is the function on $\Gamma_h(t)$ with the same nodal values as $\phi_h^n$.
	In this notation $\underline{\phi_h^n} = \underline{\phi_h^n}(t_{n-1})$.
\end{lemma}

\begin{lemma}
	For $\zeta_h^n, \eta_h^n \in S_h^n$ and sufficiently small $\tau$ we have
	\begin{align}
		|m_h(\zeta_h^n, \eta_h^n) - m_h(\underline{\zeta_h^n}, \underline{\eta_h^n})| &\leq C \tau \| \zeta_h^n \|_{L^2(\Gamma_h^n)} \| \eta_h^n \|_{L^2(\Gamma_h^n)}, \label{timeperturb1}\\
		|a_h(\zeta_h^n, \eta_h^n) - a_h(\underline{\zeta_h^n}, \underline{\eta_h^n})| &\leq  C \tau \|\gradgh \zeta_h^n \|_{L^2(\Gamma_h^n)} \|\gradgh \eta_h^n \|_{L^2(\Gamma_h^n)}, \label{timeperturb2}\\
		|\intm(\zeta_h^n, \eta_h^n) - \intm(\underline{\zeta_h^n}, \underline{\eta_h^n})| &\leq C \tau \normh{t}{\zeta_h^n} \normh{t}{\eta_h^n}, \label{timeperturb3}
	\end{align}
	where $C$ denotes a constant independent of $\tau, h$.
\end{lemma}
\begin{proof}
	\eqref{timeperturb1} is shown in \cite{dziuk2012fully} by combining Lemma 3.6 and Lemma 3.7, and the previous result.
	The proofs for \eqref{timeperturb2}, \eqref{timeperturb3} follow similarly.
\end{proof}
\begin{corollary}
	Let $\tau$ be sufficiently small, then for $\phi_h^n \in S_h^n$ we have that
	\begin{align}
		\|\phi_h^n\|_{L^2(\Gamma_h^n)} &\leq C \|\underline{\phi_h^n}\|_{L^2(\Gamma_h^{n-1})}, \label{timenorm3}\\
		\|\gradgh \phi_h^n\|_{L^2(\Gamma_h^n)} &\leq C \|\gradgh \underline{\phi_h^n}\|_{L^2(\Gamma_h^{n-1})}, \label{timenorm4},
	\end{align}
	where $C$ denotes a constant independent of $\tau, h$.
\end{corollary}
\begin{proof}
	For \eqref{timenorm3} this follows by writing
	$$\|\phi_h^n\|_{L^2(\Gamma_h^n)}^2 = \|\underline{\phi_h^n}\|_{L^2(\Gamma_h^{n-1})}^2 + \left[ \|\phi_h^n\|_{L^2(\Gamma_h^n)}^2 - \|\underline{\phi_h^n}\|_{L^2(\Gamma_h^{n-1})}^2 \right],$$
	and bounding the term in square brackets by using \eqref{timeperturb1} with a sufficiently small $\tau$.
	\eqref{timenorm4} follows similarly.
\end{proof}
In fact these results can be generalised, where one replaces $\underline{\phi_h^n} \in S_h^n$ with $\underline{\phi_h^n}(t) \in S_h(t)$ for $t \in [t_{n-1}, t_n]$.

Following the approach of \cite{elliott2024fully} now introduce some tools to be used in showing existence of a solution.
\begin{definition}
	For $z_h^{n-1} \in S_h^{n-1}$ we define $z_{h,+}^{n-1} \in S_h^n$ to be the unique solution of
	\begin{align}
		\intm(t_n;z_{h,+}^{n-1}, \phi_h^n) = \intm(t_{n-1};z_{h}^{n-1}, \underline{\phi}_h^n), \label{timeproj}
	\end{align}
	for all $\phi_h^n \in S_h^n$.
\end{definition}

This is clearly well defined by the Lax-Milgram theorem.
This time projection has the following properties.
\begin{lemma}
	For $z_h^{n-1} \in S_h^{n-1}$ and $z_{h,+}^{n-1}$ as defined above we have:
	\begin{gather}
			\|z_{h,+}^{n-1}\|_{h,t_n} \leq C \|z_h^{n-1}\|_{L^2(\Gamma_h^{n-1})}, \label{timeproj1}\\
			\|\overline{z_h^{n-1}} -  z_{h,+}^{n-1}\|_{h,t_n} \leq C{\tau} \|z_h^{n-1}\|_{L^2(\Gamma_h^{n-1})}, \label{timeproj2}
		\end{gather}

	where $C$ denotes a constant independent of $\tau,h$.
\end{lemma}
\begin{proof}
	To show \eqref{timeproj1} we test \eqref{timeproj} with $z_{h,+}^{n-1}$ and use \eqref{intmbound0}, \eqref{timenorm1} to see that
	$$ \|z_{h,+}^{n-1}\|_{L^2(\Gamma_h^n)}^2 \leq C \|\underline{z_{h,+}^{n-1}}\|_{L^2(\Gamma_h^{n-1})} \|z_h^{n-1}\|_{L^2(\Gamma_h^{n-1})} \leq C\|z_{h,+}^{n-1}\|_{L^2(\Gamma_h^n)} \|z_h^{n-1}\|_{L^2(\Gamma_h^{n-1})},$$
	from which \eqref{timeproj1} holds.
	To show \eqref{timeproj2} we note that
	$$ \intm(\overline{z_h^{n-1}} -  z_{h,+}^{n-1}, \phi_h^n) = \intm(\overline{z_h^{n-1}},\phi_h^n) - \intm(z_h^{n-1}, \underline{\phi_h^n}),$$ and hence taking $\phi_h^n = \overline{z_h^{n-1}} -  z_{h,+}^{n-1}$ and using \eqref{timeperturb1} one obtains the result.
\end{proof}
Using this time projection we notice that one may write \eqref{regfullydisceqn1} as
\begin{align}
    \frac{1}{\tau} \intm(\Udhn{n} - U_{h,+}^{n-1,\delta}, \phi_h^n)+ a_h(\Wdhn{n} , \phi_h^n) = 0. \label{regfullydisceqn1 alternate}
\end{align}

Lastly, we require an appropriate notion of a discrete inverse Laplacian, $\intinvsh$, as defined in Appendix \ref{invlaps}.
With these considerations we can decouple the equations \eqref{regfullydisceqn1}, \eqref{regfullydisceqn2}.
To do so we notice that 
\[ \frac{1}{\tau} a_h(\intinvsh(\Udhn{n} - U_{h,+}^{n-1,\delta}), \phi_h^n)+ a_h(\Wdhn{n} , \phi_h^n) = 0,  \]
and hence one finds
\begin{align}
	\label{fullydiscdecoupled}
	\Wdhn{n} =\lambda_h^{n,\delta} - \frac{1}{\tau}\intinvsh(\Udhn{n} - U_{h,+}^{n-1,\delta}) ,
\end{align}
where \eqref{regfullydisceqn2} implies
\[\lambda_h^{n,\delta} = \mval{\left( \frac{\theta}{2 \varepsilon} I_h\fd(\Udhn{n}) - \frac{1}{\varepsilon} \Udhn{n} \right)}{\Gamma_h^n}.\]
Hence the system \eqref{regfullydisceqn1}, \eqref{regfullydisceqn2} can be written as a single equation,
\begin{multline}
	\label{fullydiscdecoupledeqn}
	\varepsilon a_h \left( \Udhn{n}, \phi_h^n \right) + \frac{\theta}{2 \varepsilon}\intm\left(\fd\left(\Udhn{n}\right), \phi_h^n - \mval{\phi_h^n}{\Gamma_h^n} \right) - \frac{1}{\varepsilon}\intm\left(\Udhn{n}, \phi_h^n - \mval{\phi_h^n}{\Gamma_h^n}\right)\\
	+\frac{1}{\tau}\intm \left( \intinvsh(\Udhn{n} - U^{n-1,\delta}_{h,+}), \phi_h^n\right) = 0.
\end{multline}
This observation motivates one to define a functional, $J_h^{n,\delta} : D^n \rightarrow \mathbb{R}$ given by
\[J_h^{n,\delta} (z_h^n) := \frac{1}{\varepsilon} \intm \left( F^{\delta} (z_h^n), 1 \right) + \frac{\varepsilon}{2} \|\gradgh z_h^n \|_{L^2(\Gamma_h^n)}^2 + \frac{1}{2 \tau} \|z_h^n - U_{h,+}^{n-1,\delta}\|_{-h}^2 ,\]
where
\[D^n := \left\{ z_h^n \in S_h^n \mid \intm(z_h^n,1) = \intm(\Udhn{n-1},1) \right\}.\]
$D^n$ is clearly an affine subspace of $S_h^n$, and thus finite dimensional.

\begin{lemma}
	Let $U_{h,0} \in \mathcal{I}_{h,0}$.
	Then there exists a solution pair $(\Udhn{n},\Wdhn{n})$ solving \eqref{regfullydisceqn1}, \eqref{regfullydisceqn2}.
\end{lemma}
\begin{proof}
    With the observation that $\Udhn{n}$ solves \eqref{fullydiscdecoupledeqn}, one may argue by the same logic as in \cite[Lemma 3.5]{elliott2024fully} --- as such we omit further details.
    Moreover, by arguing as in \cite[Lemma 3.1]{elliott2024fully} one finds that this solution is unique under the assumption that $\varepsilon < 4 \varepsilon^3$.
    
\end{proof}

Now as in Section \ref{section: logch semi discrete} we need to establish bounds independent of $\delta$ to consider the limit $\delta \rightarrow 0$.
As before, this will make use of our assumption that $\gradgh \cdot V_h \geq 0$.

\begin{lemma}
	Let $\tau > 0$ be sufficiently small and such that $\tau \leq \frac{\varepsilon^3}{2}$.
	Then the unique solution $(\Udhn{n}, \Wdhn{n})$ of \eqref{regfullydisceqn1}, \eqref{regfullydisceqn2} satisfies for $ 1 \leq N \leq \lfloor \frac{T}{\tau} \rfloor$,
	\begin{align}
		\Echdh[\Udhn{N}] + \frac{\varepsilon\tau^2}{4} \sum_{n=1}^N \|\gradgh \matdevtau \Udhn{n}\|_{L^2(\Gamma_h^n)}^2 + \tau \sum_{n=1}^N \|\gradgh \Wdhn{n}\|_{L^2(\Gamma_h^n)}^2 \leq C \exp\left(Ct_N\right),\label{fdiscregstability}
	\end{align}
	where $C$ denotes a constant independent of $\delta,\tau, h$.
\end{lemma}
\begin{proof}
	We begin by writing \eqref{regfullydisceqn1}, \eqref{regfullydisceqn2} in matrix form as
	\begin{gather}
		\bar{M}^n \frac{(U^{n,\delta} - U^{n-1,\delta})}{\tau} + \frac{(\bar{M}^n - \bar{M}^{n-1})}{\tau} U^{n-1,\delta} + A^n W^{n,\delta} = 0,\label{matrixeqn1}\\
		\bar{M}^n W^{n,\delta} = \varepsilon A^n U^{n,\delta} + \frac{\theta}{2\varepsilon}\bar{M}^n \fd(U^{n,\delta}) - \frac{1}{\varepsilon}\bar{M}^n U^{n,\delta}, \label{matrixeqn2}
	\end{gather}
	where $U^{n,\delta}, W^{n,\delta}$ denote the vector of nodal values of $\Udhn{n}, \Wdhn{n}$ respectively, and
	\begin{align*}
		(\bar{M}^n)_{ij} &= \intm(t_n; \phi_i^n, \phi_j^n),\\
		(A^n)_{ij} &= a_h(t_n; \phi_i^n, \phi_j^n),
	\end{align*}
	for $\phi_i^n$ the `$i$'th nodal basis function in $S_h^n$.
	Here we understand $\fd(U^{n,\delta})$ as the vector with entries $\fd(U^{n,\delta}_i)$, as we justified in Section \ref{section: logch semi discrete}.
	We dot \eqref{matrixeqn1} with $W^{n,\delta}$ and \eqref{matrixeqn2} with $\frac{1}{\tau}(U^{n,\delta} - U^{n-1,\delta})$.
	This yields
	\begin{align}
		\begin{gathered}
		\frac{\varepsilon}{\tau}U^{n,\delta} \cdot A^n (U^{n,\delta} - U^{n-1,\delta}) + \frac{\theta}{2\tau \varepsilon} (U^{n,\delta} - U^{n-1,\delta}) \cdot \bar{M}^n \fd(U^{n,\delta}) + W^{n,\delta} \cdot A^n W^{n,\delta}\\
		+  W^{n,\delta} \cdot \frac{(\bar{M}^n - \bar{M}^{n-1})}{\tau} U^{n-1,\delta} = \frac{1}{\tau\varepsilon} U^{n,\delta} \cdot \bar{M}^n(U^{n,\delta}-U^{n-1,\delta}).
		\end{gathered}\label{regfullydiscenergy1}
	\end{align}
	It is then straightforward to verify that
	\begin{align}
		\begin{aligned}
			\frac{\varepsilon}{\tau}U^{n,\delta} \cdot A^n (U^{n,\delta} - U^{n-1,\delta}) &= \frac{\varepsilon}{2\tau} \left[ U^{n,\delta} \cdot A^n U^{n,\delta} - U^{n-1,\delta} \cdot A^{n-1} U^{n-1,\delta} \right]\\
			&+ \frac{\varepsilon}{2\tau} \left(U^{n,\delta} - U^{n-1,\delta}\right) \cdot A^n \left(U^{n,\delta} - U^{n-1,\delta}\right)\\
			&+ \frac{\varepsilon}{2\tau} \left[ U^{n-1,\delta} \cdot A^{n-1} U^{n-1,\delta} - U^{n-1,\delta} \cdot A^n U^{n-1,\delta}  \right], 
		\end{aligned}\label{regfullydiscenergy2}
	\end{align}
	which we can further write as
	\begin{align*}
		\begin{aligned}
			\frac{\varepsilon}{\tau}U^{n,\delta} \cdot A^n (U^{n,\delta} - U^{n-1,\delta}) &= \frac{\varepsilon}{2\tau} \left[ \| \gradgh \Udhn{n} \|_{L^2(\Gamma_h^n)}^2 - \| \gradgh \Udhn{n-1} \|_{L^2(\Gamma_h^{n-1})}^2 \right]\\
			&+ \frac{\varepsilon\tau}{2} \| \gradgh \matdevtau \Udhn{n} \|_{L^2(\Gamma_h^n)}^2\\
			&+ \frac{\varepsilon}{2\tau} \left[ a_h(\Udhn{n-1},\Udhn{n-1}) - a_h(\overline{\Udhn{n-1}},\overline{\Udhn{n-1}})  \right].
		\end{aligned}
	\end{align*}
	This can be understood as a fully discrete form of the transport theorem, see \cite[Lemma 3.5]{dziuk2012fully}.\\
	
	Next we turn to the term $W^{n,\delta} \cdot \frac{(\bar{M}^n - \bar{M}^{n-1})}{\tau} U^{n-1,\delta}$, which we deal with in an analogous way to Lemma \ref{regfeexist}.
	Firstly we recall that, due to mass-lumping, $\bar{M}^n, \bar{M}^{n-1}$ are diagonal matrices, and hence if we define\footnote{Note that $\bar{G}^n \neq \bar{G}(t_n)$ for $\bar{G}(t)$ as defined in Section \ref{section: logch semi discrete}, but $\bar{G}^n$ is a finite difference approximation of $\bar{G}(t_n)$.}
	\[\bar{G}^n := \frac{(\bar{M}^n - \bar{M}^{n-1})}{\tau},\]
	then $\bar{G}^n$ is a diagonal matrix, and we show that since $\gradgh \cdot V_h \geq 0$ it is also positive semi-definite.
	This follows since
	\begin{align*}
		\bar{G}^n_{ii} = \frac{(\bar{M}_{ii}^n - \bar{M}_{ii}^{n-1})}{\tau} = \frac{1}{\tau} \left( \int_{\Gamma_h^n} \phi_i^n - \int_{\Gamma_h^{n-1}} \phi_i^{n-1} \right) = \frac{1}{\tau}  \int_{t_{n-1}}^{t_n}\frac{d}{dt} m_h(\phi_i(t), 1) =\underbrace{\frac{1}{\tau}  \int_{t_{n-1}}^{t_n} g_h(\phi_i(t), 1)}_{\geq 0},
	\end{align*}
	since $\phi_i(t) \geq 0$ and we have assumed $\gradgh \cdot V_h \geq 0$.
	Now as $\bar{M}^n$ is invertible one finds that
	\begin{align}
		W^{n,\delta} \cdot \bar{G}^n U^{n-1,\delta} = \left( \varepsilon(\bar{M}^n)^{-1}A^n U^{n,\delta} + \frac{\theta}{2\varepsilon}f(U^{n,\delta})- \frac{1}{\varepsilon} U^{n,\delta} \right)\cdot \bar{G}^n U^{n-1,\delta}. \label{regfullydiscenergy W term}
	\end{align}
	
	Since $\bar{G}^n$ is symmetric this first term may be rewritten as
	\[ (\bar{M}^n)^{-1}A^n U^{n,\delta} \cdot \bar{G}^n U^{n-1,\delta} = U^{n-1,\delta}\cdot \bar{G}^n (\bar{M}^n)^{-1}A^n U^{n,\delta}, \]
	and we shall defer treatment of this term until the end of the proof.
	For now we shall only deal with the the potential term, for which we observe that the left-hand side of \eqref{regfullydiscenergy1} contains
	\begin{multline*}
		\frac{\theta}{2 \tau\varepsilon} \fd(U^{n,\delta})\cdot \bar{M}^n (U^{n,\delta} - U^{n-1,\delta})+ \frac{\theta}{2 \varepsilon} \fd(U^{n,\delta}) \cdot \bar{G}^n U^{n-1,\delta}\\
		 = \frac{\theta}{2 \varepsilon}\fd(U^{n,\delta}) \cdot \bar{G}^n U^{n,\delta} + \frac{\theta}{2 \tau\varepsilon}  \fd(U^{n,\delta})\cdot \bar{M}^n (U^{n,\delta} - U^{n-1,\delta}) - \frac{\theta}{2\varepsilon} \fd(U^{n,\delta}) \cdot \bar{G}^n (U^{n,\delta} - U^{n-1,\delta}).
	\end{multline*}
	For this first term we notice that as $\bar{G}^n$ is diagonal and positive semi-definite, and $r\fd(r) \geq 0$ one has
	\[\fd(U^{n,\delta}) \cdot \bar{G}^n U^{n,\delta} = \sum_{i=1}^{N_h} \fd(U_i^{n,\delta})U_i^{n,\delta} \bar{G}^n_{ii} \geq 0\]
	Next use the definition of $\bar{G}^n$ to see that
	\begin{multline*}
		\frac{\theta}{2 \tau\varepsilon}  \fd(U^{n,\delta})\cdot \bar{M}^n (U^{n,\delta} - U^{n-1,\delta}) - \frac{\theta}{2 \varepsilon} \fd(U^{n,\delta}) \cdot \bar{G}^n (U^{n,\delta} - U^{n-1,\delta})\\
		= \frac{\theta}{2 \tau\varepsilon} \fd(U^{n,\delta}) \cdot \bar{M}^{n-1} (U^{n,\delta} - U^{n-1,\delta}).
	\end{multline*}
	In terms of our bilinear forms this is
	\[ \frac{\theta}{2 \tau\varepsilon} \fd(U^{n,\delta}) \cdot \bar{M}^{n-1} (U^{n} - U^{n-1}) = \frac{\theta}{2 \tau\varepsilon} \intm(\underline{I_h\fd(\Udhn{n})}, \underline{\Udhn{n}} - \Udhn{n-1}), \]
	and now using the convexity of $F_{\log}^{\delta}(\cdot)$ one finds that
	\[\frac{\theta}{2\tau\varepsilon} \intm(\underline{I_h\fd(\Udhn{n})}, \underline{\Udhn{n}} - \Udhn{n-1}) \geq \frac{\theta}{2\tau\varepsilon} \intm(\underline{I_hF_{\log}^{\delta}(\Udhn{n})} - I_hF_{\log}^{\delta}(\Udhn{n-1}),1). \]
	This is essentially the correct term for the \emph{convex part} of the potential, but it remains to add the \emph{quadratic part}, i.e. terms to do with the $\frac{1-r^2}{2}$ term in $F^\delta(r)$.
	To retrieve the quadratic part of the potential we observe that the right-hand side of \eqref{regfullydiscenergy1} can be expressed as
	\begin{multline*}
		\frac{1}{\tau\varepsilon}  U^{n,\delta} \cdot \bar{M}^n(U^{n,\delta}-U^{n-1,\delta}) =  \frac{1}{2 \tau \varepsilon} \left( U^{n,\delta} \cdot \bar{M}^n U^{n,\delta} - U^{n-1,\delta} \cdot \bar{M}^{n-1} U^{n-1,\delta} \right)\\
		+ \frac{1}{2 \tau \varepsilon}(U^{n,\delta} - U^{n-1,\delta}) \cdot \bar{M}^n (U^{n,\delta} - U^{n-1,\delta}) + \frac{1}{2 \tau \varepsilon} \left[U^{n-1,\delta} \cdot \bar{M}^{n-1} U^{n-1,\delta} - U^{n-1,\delta} \cdot \bar{M}^{n} U^{n-1,\delta} \right],
	\end{multline*}
	where this first term is essentially the correct term for the quadratic part of the potential. 
	From this we have that
	\begin{align}
		\begin{aligned}
		\frac{1}{\tau\varepsilon}  U^{n,\delta} \cdot \bar{M}^n(U^{n,\delta}-U^{n-1,\delta}) &= -\frac{1}{\tau\varepsilon}\left[\intm\left(\frac{1-(\Udhn{n})^2}{2},1\right) - \intm\left(\frac{1-(\Udhn{n-1})^2}{2},1\right)\right]\\
  &+  \frac{1}{2 \tau \varepsilon}\left[ \intm(t_n;1,1) - \intm(t_{n-1};1,1) \right]\\
  & + \frac{\tau}{2\varepsilon} \|\matdevtau \Udhn{n} \|_{h,t_n}^2\\
		&+ \frac{1}{2 \tau \varepsilon} \left[ \intm(\Udhn{n-1},\Udhn{n-1}) - \intm(\overline{\Udhn{n-1}}, \overline{\Udhn{n-1}}) \right],
		\end{aligned}\label{regfullydiscenergy4}
	\end{align}
    where we have introduced extra terms to match the form of the quadratic part of the potential.
	It is straightforward to see that
	\[ \frac{1}{2 \tau \varepsilon}\left[ \intm(t_n;1,1) - \intm(t_{n-1};1,1) \right] = \frac{(|\Gamma_h^n| - |\Gamma_h^{n-1}|)}{2\tau \varepsilon} .\]
	
	Combining \eqref{regfullydiscenergy2}, \eqref{regfullydiscenergy4} in \eqref{regfullydiscenergy1} one finds
	\begin{multline}
		\frac{\varepsilon}{2 \tau}\left( \|\gradgh \Udhn{n}\|_{L^2(\Gamma_h^n)}^2 - \|\gradgh \Udhn{n-1}\|_{L^2(\Gamma_h^{n-1})}^2  \right) + \frac{1}{\tau \varepsilon}\left(\int_{\Gamma_h^n} I_h F^{\delta}(\Udhn{n}) - \int_{\Gamma_h^{n-1}} I_h F^{\delta}(\Udhn{n-1}) \right)\\
		+\frac{\varepsilon \tau}{2}\|\gradgh \matdevtau \Udhn{n}\|_{L^2(\Gamma_h^n)}^2+ \|\gradgh \Wdhn{n}\|_{L^2(\Gamma_h^n)}^2 \leq C + \sum_{k=1}^6 I_k,
		\label{regfullydiscenergy5}
	\end{multline}
	where 
	\begin{gather*}
		I_1 := \frac{\varepsilon}{2\tau} \left[ a_h(\overline{\Udhn{n-1}},\overline{\Udhn{n-1}}) - a_h(\Udhn{n-1},\Udhn{n-1})\right],\\
		I_2 := \frac{\theta}{2 \tau \varepsilon}\left[\intm({I_hF_{\log}^{\delta}(\Udhn{n})},1) -  \intm(\underline{I_hF_{\log}^{\delta}(\Udhn{n})},1) \right],\\
		I_3 := \frac{1}{\varepsilon} U^{n,\delta} \cdot \bar{G}^n U^{n-1,\delta} = \frac{1}{\tau \varepsilon}\left[ \intm(\Udhn{n}, \overline{\Udhn{n-1}}) - \intm(\underline{\Udhn{n}}, \Udhn{n-1}) \right],\\
		I_4 := \frac{1}{2 \tau \varepsilon} \left[\intm(\overline{\Udhn{n-1}}, \overline{\Udhn{n-1}}) - \intm(\Udhn{n-1},\Udhn{n-1}) \right],\\
		I_5 := \frac{\tau}{2\varepsilon} \|\matdevtau \Udhn{n} \|_{h,t_n}^2,\\
		I_6 := -\varepsilon U^{n-1,\delta}\cdot \bar{G}^n (\bar{M}^n)^{-1}A^n U^{n,\delta},
	\end{gather*}
	and $C$ is a constant, independent of $\delta, h, \tau$, such that 
	\[\left|\frac{|\Gamma_h^n| - |\Gamma_h^{n-1}|}{2\tau \varepsilon}\right| \leq C,\]
	which one can obtain from the smoothness of $V$.
	It is important to note that now the potential term in \eqref{regfullydiscenergy5} is in terms of the full potential, $F^\delta$, not just the convex part, $F^\delta_{\log}$.
	
	We now bound each of these terms.
	Firstly, we find that 
	\begin{align}
		|I_1| \leq C \varepsilon \|\gradgh \Udhn{n-1}\|_{L^2(\Gamma_h^{n-1})}^2, \label{regfullydiscenergy6}
	\end{align}
	from \eqref{timeperturb2}.
	Similarly we observe that one obtains
	\begin{gather}
		|I_2| \leq \frac{C}{\varepsilon} \intm(I_h F_{\log}^{\delta}(\Udhn{n}),1),\label{regfullydiscenergy7}\\
		|I_3| \leq \frac{C}{\varepsilon}\|\Udhn{n}\|_{L^2(\Gamma_h^n)}\|\Udhn{n-1}\|_{L^2(\Gamma_h^{n-1})},\label{regfullydiscenergy8}\\
		|I_4| \leq \frac{C}{\varepsilon} \|\Udhn{n-1}\|_{L^2(\Gamma_h^{n-1})}^2,\label{regfullydiscenergy9}
	\end{gather}
	from using \eqref{timeperturb3} (and \eqref{intmbound0} where necessary).
	
	It remains to bound $I_5$ and $I_6$, which are the two most problematic terms.
	Firstly, to bound $I_5$ we notice that we may write \eqref{regfullydisceqn1} as
	\[ \intm(\matdevtau \Udhn{n}, \phi_h^n) + \frac{1}{\tau} \left( \intm(\overline{\Udhn{n-1}}, \phi_h^n) - \intm({\Udhn{n-1}}, \underline{\phi_h^n})\right) + a_h(\Wdhn{n}, \phi_h^n) = 0.\]
	We test this with $\phi_h^n = \frac{\tau}{2 \varepsilon} \matdevtau \Udhn{n}$ to see that
	\[I_5 = -\frac{1}{2\varepsilon}\left( \intm(\overline{\Udhn{n-1}}, \matdevtau \Udhn{n}) - \intm({\Udhn{n-1}}, \underline{\matdevtau \Udhn{n}})\right) - \frac{\tau}{2\varepsilon}a_h(\Wdhn{n}, \matdevtau \Udhn{n}).\]
    Recalling that $I_5 = \frac{\tau}{2 \varepsilon} \|\matdevtau \Udhn{n}\|_{h,t_n}^2$, we may use \eqref{timeperturb3} and Young's inequality to find that
	\begin{align}
		\begin{aligned}
		|I_5| &\leq C\tau \|\Udhn{n-1}\|_{h,t_{n-1}}^2 + \frac{1}{2} \|\gradgh \Wdhn{n}\|_{L^2(\Gamma_h^n)}^2 + \frac{\tau^2}{2 \varepsilon^2} \|\gradgh \matdevtau \Udhn{n}\|_{L^2(\Gamma_h^n)}^2\\
		&\leq C\tau \|\Udhn{n-1}\|_{h,t_{n-1}}^2 + \frac{1}{2} \|\gradgh \Wdhn{n}\|_{L^2(\Gamma_h^n)}^2 + \frac{\varepsilon \tau}{4} \|\gradgh \matdevtau \Udhn{n}\|_{L^2(\Gamma_h^n)}^2,
		\end{aligned}
		\label{regfullydiscenergy10}
	\end{align}
	where we have used the assumption $\tau \leq \frac{\varepsilon^3}{2}$.
	We now turn to $I_6$, which is dealt with by similar means to the proof of Lemma \ref{regfeexist}.
	We define the function $\widetilde{\Udhn{n}} \in S_h^n$ to be the unique solution of
	\[ \intm(\widetilde{\Udhn{n}}, \phi_h^n) = a_h({\Udhn{n}}, \phi_h^n), \]
	for all $\phi_h^n \in S_h^n$.
	As in \eqref{laplacian inverse inequality} one finds that
	\[ \|\widetilde{\Udhn{n}} \|_{h, t_{n}} \leq \frac{C}{{h}} \|\gradgh \Udhn{n}\|_{L^2(\Gamma_h^n)},\]
	for a constant independent of $\tau, h$.
	Denoting the vector of nodal values of $\widetilde{\Udhn{n}}$ as $\widetilde{U^{n,\delta}}$ then one finds that $\widetilde{U^{n,\delta}} = (\bar{M}^n)^{-1} A^n {U^{n,\delta}}$.
	Hence we find that
	\[ |I_6| = \varepsilon|U^{n-1,\delta} \cdot \bar{G}^n \widetilde{U^{n,\delta}}| = \frac{\varepsilon}{\tau}|\intm(\overline{\Udhn{n-1}}, \widetilde{\Udhn{n}}) - \intm({\Udhn{n-1}}, \underline{\widetilde{\Udhn{n}}})|.\]
	The idea is now to write
	\begin{align*}\intm(\overline{\Udhn{n-1}}, \widetilde{\Udhn{n}}) - \intm({\Udhn{n-1}}, \underline{\widetilde{\Udhn{n}}}) &= \int_{t_{n-1}}^{t_n}\frac{d}{dt}\intm(\overline{\Udhn{n-1}}(s), \underline{\widetilde{\Udhn{n}}}(s)) \, ds\\
	&= \int_{t_{n-1}}^{t_n}g_h(I_h(\overline{\Udhn{n-1}}(s) \underline{\widetilde{\Udhn{n}}}(s)),1) \, ds.
	\end{align*}
	Now one proceeds almost identically to Lemma \ref{regfeexist} (as well as using \eqref{timenorm3}) to find that
	\[ |g_h(I_h(\overline{\Udhn{n-1}}(s) \underline{\widetilde{\Udhn{n}}}(s)),1)| \leq C + C \Echdh[\Udhn{n-1}] + C \Echdh[\Udhn{n}]. \]
	In the interest of brevity we do not expound upon these details.
	The end result of this calculation is that
	\begin{align}
		|I_6| \leq C + C \Echdh[\Udhn{n-1}] + C \Echdh[\Udhn{n}]. \label{regfullydiscenergy11}.
	\end{align}
	
	Combining \eqref{regfullydiscenergy6}--\eqref{regfullydiscenergy11} in \eqref{regfullydiscenergy5}, and using Young/Poincar\'e inequalities accordingly, one readily obtains
	\begin{multline*}
		\Echdh[\Udhn{N}] - \Echdh[U_{h,0}] +  \frac{\varepsilon\tau^2}{4} \sum_{n=1}^N \|\gradgh \matdevtau \Udhn{n}\|_{L^2(\Gamma_h^n)}^2 + \tau \sum_{n=1}^N \| \gradgh \Wdhn{N}\|_{L^2(\Gamma_h^n)}^2\\
		\leq C + C \tau\sum_{n=0}^{N-1} \Echdh[\Udhn{n}] + C\tau \Echdh[\Udhn{N}].
	\end{multline*}
	The result follows from a discrete Gr\"onwall inequality, provided $\tau$ is sufficiently small.
	This can then be bounded independent of $\delta, h$ by similar arguments to Lemma $\ref{regfeexist}$.
\end{proof}

We also obtain slightly stronger bounds on the derivative by considering weaker norms.
This is the content of the following lemma, which will be invaluable for our later error analysis.
\begin{lemma}
	Under the assumptions in the previous lemma, and assuming that a Courant-Friedrichs-Lewy (CFL) condition, $\tau \leq Ch^2$, holds then one has that for $1 \leq N \leq N_T$,
	\begin{gather}
		\tau^{\frac{3}{2}} \sum_{n=1}^N \|\matdevtau \Udhn{n}\|_{L^2(\Gamma_h^n)}^2 \leq C, \label{fdisc derivative bound1}\\
		\tau \sum_{n=1}^N \inormh{\matdevtau \Udhn{n} - \mval{\matdevtau \Udhn{n}}{\Gamma_h^n} }^2 \leq C, \label{fdisc derivative bound2}
	\end{gather}
	for a constant $C$ independent of $\delta, \tau, h$.
\end{lemma}
\begin{proof}
	
	We first prove \eqref{fdisc derivative bound1}, noting from \eqref{intmbound0} that
	\begin{align*}
		\tau^\frac{3}{2} \|\matdevtau \Udhn{n}\|_{L^2(\Gamma_h^n)}^2 &\leq \frac{C}{\sqrt{\tau}} \normh{t_n}{\Udhn{n} - U_{h,+}^{n-1,\delta}}^2+ \frac{C}{\sqrt{\tau}} \normh{t_n}{U_{h,+}^{n-1,\delta} - \overline{\Udhn{n-1}}}^2
	\end{align*}
	From \eqref{timeproj2} one immediately finds
	\[\frac{C}{\sqrt{\tau}} \normh{t_n}{U_{h,+}^{n-1,\delta} - \overline{\Udhn{n-1}}}^2 \leq C\tau^\frac{3}{2} \|\Udhn{n-1}\|_{L^2(\Gamma_h^{n-1})}^2,\]
    and hence we sum over $N = 1,...,N$ and use the mass-conservation property of $\Udhn{n}$ along with the Poincar\'e inequality for
    \[ \frac{C}{\sqrt{\tau}} \sum_{n=1}^N  \normh{t_n}{U_{h,+}^{n-1,\delta} - \overline{\Udhn{n-1}}}^2 \leq C\tau^\frac{3}{2} \sum_{n=1}^N\left( 1 + \|\gradgh \Udhn{n-1}\|_{L^2(\Gamma_h^{n-1})}^2 \right) \leq C. \]
    Here the final inequality follows by using \eqref{fdiscregstability}, the fact that $N \tau \leq T$ and the upper bound on $\tau$.
	
	The other term will be bounded by an interpolation argument.
	Firstly we notice that
	\begin{align*}
		\normh{t_n}{\Udhn{n} - U_{h,+}^{n-1,\delta}}^2 &= a_h(\intinvsh(\Udhn{n} - U_{h,+}^{n-1,\delta}), \Udhn{n} - U_{h,+}^{n-1,\delta})\\
		&\leq \inormh{\Udhn{n} - U_{h,+}^{n-1,\delta}} \|\gradgh(\Udhn{n} - U_{h,+}^{n-1,\delta})\|_{L^2(\Gamma_h^n)}.
	\end{align*}
	
	We now establish a bound on $\inormh{\left(\Udhn{n} - U_{h,+}^{n-1,\delta}\right)}$.
	To do this we test \eqref{regfullydisceqn1 alternate} with $\intinvsh\left( \Udhn{n} - U_{h,+}^{n-1,\delta} \right)$ to obtain
	\begin{align*}
		\inormh{\Udhn{n} - U_{h,+}^{n-1,\delta}}^2 = -\tau a_h\left( \Wdhn{n}, \intinvsh\left(\Udhn{n} - U_{h,+}^{n-1,\delta}\right)\right),
	\end{align*}
	from which Young's inequality yields
	\[ \inormh{\Udhn{n} - U_{h,+}^{n-1,\delta}} \leq \tau \| \gradgh \Wdhn{n}\|_{L^2(\Gamma_h^n)}, \]
	and by using \eqref{fdiscregstability} it follows that
	\[  \frac{1}{\tau}\sum_{n=1}^N\inormh{\Udhn{n} - U_{h,+}^{n-1,\delta}}^2 \leq \tau \sum_{n=1}^N \| \gradgh \Wdhn{n}\|_{L^2(\Gamma_h^n)}^2 \leq C. \]
	By using H\"older's inequality we find that
	\[\frac{1}{\sqrt{\tau}}\sum_{n=1}^N\normh{t_n}{\Udhn{n} - U_{h,+}^{n-1,\delta}}^2 \leq \left( \frac{1}{\tau}\sum_{n=1}^N\inormh{\Udhn{n} - U_{h,+}^{n-1,\delta}}^2\right)^{\frac{1}{2}} \left(\sum_{n=1}^N\|\gradgh(\Udhn{n} - U_{h,+}^{n-1,\delta})\|_{L^2(\Gamma_h^n)}^2\right)^{\frac{1}{2}},\]
	where we claim that the rightmost term is bounded independent of $\delta, \tau, h$.
	To see this is true we write
	\begin{align*}
		\sum_{n=1}^N\|\gradgh(\Udhn{n} - U_{h,+}^{n-1,\delta})\|_{L^2(\Gamma_h^n)}^2 &\leq 2\tau^2 \sum_{n=1}^N\|\gradgh\matdevtau\Udhn{n}\|_{L^2(\Gamma_h^n)}^2 + 2\sum_{n=1}^N\|\gradgh(\overline{\Udhn{n-1}} - U_{h,+}^{n-1,\delta})\|_{L^2(\Gamma_h^n)}^2\\
		& \leq C + \frac{C \tau^2}{h^2} \sum_{n=1}^N \|\Udhn{n-1}\|_{L^2(\Gamma_h^{n-1})}^2 \leq C,
	\end{align*}
	where we have used \eqref{timeproj2}, \eqref{fdiscregstability}, an inverse inequality, our CFL condition $\tau \leq Ch^2$, the bound $N\tau \leq T$, as well as using the Poincar\'e inequality in the usual way.
	Combining all of this together appropriately yields \eqref{fdisc derivative bound1}.
	
	To show \eqref{fdisc derivative bound2} we observe that \eqref{regfullydisceqn1} can be written as
	\[ \intm(\matdevtau \Udhn{n}, \phi_h^n) + \frac{1}{\tau} \left( \intm(\overline{\Udhn{n-1}}, \phi_h^n ) - \intm(\Udhn{n-1}, \underline{\phi_h^n})\right) + a_h(\Wdhn{n}, \phi_h^n) = 0. \]
	We test the above with $\intinvsh(\matdevtau \Udhn{n} - \mval{\matdevtau \Udhn{n}}{\Gamma_h^n})$ to see that
	\begin{multline*}
		\tau \inormh{\matdevtau \Udhn{n} - \mval{\matdevtau \Udhn{n}}{\Gamma_h^n}}^2 = -\tau a_h\left(\Wdhn{n},\intinvsh\left(\matdevtau \Udhn{n} - \mval{\matdevtau \Udhn{n}}{\Gamma_h^n}\right)\right)\\
		+\left(\intm\left(\Udhn{n-1}, \underline{\intinvsh\left(\matdevtau \Udhn{n} - \mval{\matdevtau \Udhn{n}}{\Gamma_h^n}\right)}\right)- \intm\left(\overline{\Udhn{n-1}}, \intinvsh\left(\matdevtau \Udhn{n} - \mval{\matdevtau \Udhn{n}}{\Gamma_h^n}\right)\right)\right).
	\end{multline*}
	\eqref{fdisc derivative bound2} now follows by using Young/Poincar\'e inequalities, \eqref{timeperturb3}, and \eqref{fdiscregstability}, where one bounds the $\|\Udhn{n-1}\|_{L^2(\Gamma_h^{n-1})}$ term as we have above.
\end{proof}

Next we state an analogue of Lemma \ref{udhcontrol}, which allows us to control the measure of the sets $ \{ |\Udhn{n}| > 1 \} \subset \Gamma_h^n$.
This is identical to the proof of Lemma \ref{udhcontrol}, which considers a fixed time $t \in [0,T]$.

\begin{lemma}
	There exist constants $C_1, C_2 > 0$, independent of $\delta \in (0,1)$, such that for $n \geq 1 $
	\begin{equation}\label{fdiscextrem}
		\int_{\Gamma_{h}^n} [-1 - \Udhn{n}]_+ + \int_{\Gamma_{h}^n} [\Udhn{n} - 1]_+ \leq C\left(\frac{1}{|\log(\delta)|} + \delta \right).
	\end{equation}
\end{lemma}

As in the semi-discrete case this is used in bounding the potential term, as in the following lemma which is a discrete time analogue of Lemma \ref{dpotbound}.
\begin{lemma}
	For $N \geq 1$, and sufficiently small $\delta$ we have the following bounds,
	\begin{gather}
		\label{fdiscdpot1}
		\tau \sum_{n = 1}^N \left\| I_h\fd(\Udhn{n}) - \mval{I_h\fd(\Udhn{n})}{\Gamma_h^n} \right\|_{L^2(\Gamma_h^n)}^2 \leq C,\\
		\label{fdiscdpot2}
		\tau \sum_{n = 1}^N \left\| I_h\fd(\Udhn{n}) \right\|_{L^2(\Gamma_h^n)}^2 \leq C,
	\end{gather}
	where $C$ denotes a constant independent of $\delta, \tau$.
\end{lemma}

We omit the proof of this result as it is essentially identical to Lemma \ref{dpotbound}.
As in the semi-discrete case, this allows one to establish $\delta$-independent bounds on $\tau \sum_{n=1}^{N_T}\|\Wdhn{n}\|_{H^1(\Gamma_h^n)}^2$ by using the Poincar\'e inequality and \eqref{fdiscregstability} (and obtaining an appropriate bound on $\mval{\Wdhn{n}}{\Gamma_h^n}$ analogously to the argument in Section \ref{section: logch semi discrete}).

\subsubsection{Passage to the limit}

We now use the established uniform bounds to pass to the limit as $\delta \searrow 0$.
By the compact embedding $H^1(\Gamma_h^n) \overset{c}{\hookrightarrow} L^2(\Gamma_h^n)$ we can find a subsequence of $\delta \searrow 0$ such that $\Udhn{n} \rightarrow U_h^n$ and $\Wdhn{n} \rightarrow W_h^n$ strongly in $L^2(\Gamma_h^n)$, where $U_h^n, W_h^n$ are the appropriate limits in $S_h^n$.
Thus we obtain pointwise convergence for almost all $x \in \Gamma_h^n$, and by arguing as in Lemma \ref{conveverywhere} this is in fact convergence for all $x \in \Gamma_h^n$.

Similar to the semi-discrete solution, we have that $|U_h^n| \leq 1$ almost everywhere on $\Gamma_h^n$ for $n \geq 1$.
We must again strengthen this to be a strict inequality, for which one can repeat the arguments of Lemma \ref{uh<1}.
Using this, the continuity of $f$ away from $\pm 1$, and the convergence $\Udhn{n}(x_i) \rightarrow U_h^n(x_i)$ we find
$$ \fd(\Udhn{n}(x_i))  = f(\Udhn{n}(x_i))\rightarrow f(U_h^n(x_i)),$$
for each of the nodes $x_i \in \Gamma_h^n$.
Note that this implies pointwise convergence everywhere on $\Gamma_h^n$ hence we also obtain strong convergence $I_h\fd(\Udhn{n}) \rightarrow I_h f(U_h^n)$ in $L^p(\Gamma_h^n)$ for $p \in [1,\infty]$ as the maximum of these functions occurs at the nodes of $\Gamma_h(t)$.
Hence by passing to the limit as $\delta \rightarrow 0$ in \eqref{regfullydisceqn1}, \eqref{regfullydisceqn2} we obtain solutions of \eqref{fullydisceqn1}, \eqref{fullydisceqn2}.
We end this subsection by stating a discrete time analogue of Proposition \ref{festability}.
\begin{proposition}
	\label{fullydiscstability}
	Let $U_{h,0}^1, U_{h,0}^2 \in \mathcal{I}_{h,0}$ be initial data with $\mval{U_{h,0}^1}{\Gamma_{h}(0)}=\mval{U_{h,0}^2}{\Gamma_{h}(0)}$.
	We denote the corresponding solution pair of \eqref{fullydisceqn1}, \eqref{fullydisceqn2} by $(U_h^{n,i}, W_h^{n,i})$, for $i=1,2$.
	Then if $\tau < \varepsilon^3$, we have for all $N \geq 1$,
	\[ \inormh{U_h^{N,1} - U_h^{N,2}}^2 + \frac{\varepsilon^4 \tau}{\varepsilon^3 - \tau} \sum_{n=1}^N \|\gradgh (U_h^{n,1} - U_h^{n,2})\|_{L^2(\Gamma_h^n)}^2 \leq e^{Ct_N}\left(\frac{ \varepsilon^3}{\varepsilon^3 - \tau}\right) \inormh{U^1_{h,0} - U^2_{h,0}}^2. \]
\end{proposition}

\begin{proof}
    This proof is essentially identical to \cite[Proposition 3.1]{elliott2024fully} and is hence omitted.
    The only meaningful difference is that due to the use of mass-lumping we use $\intinvsh$ rather than $\invsh$ (again see Appendix \ref{invlaps}).
\end{proof}

\begin{remark}
    The smallness condition on $\tau$ is quite severe since in practise $\varepsilon$ is very small (to approximate a sharp interface).
    One could mitigate these restrictions by using an implicit-explicit scheme as in \cite[Section 5]{elliott2024fully}, but we leave this as a topic for future work.
\end{remark}

\subsection{Error analysis}
In this subsection we provide a similar analysis to that of the semi-discrete scheme, and prove error bounds for the fully discrete scheme.
As in the semi-discrete case, we assume $u_0$ is such that Lemma \ref{initialdatalemma} holds, and consider $U_{h,0} = \Pi_h u_0$.

As in \cite{barrett1995error, copetti1992numerical} we construct piecewise functions in time from the values $U_h^n , W_h^n$, and compare these to the semi-discrete solutions, $U_h, W_h$.
The reason for this is twofold.
Firstly, we can use the error bounds we established for the semi-discrete scheme.
Secondly, if we were to show a time discrete analogue of \eqref{deltaerror1} (on the true surface $\Gamma(t)$) we would require stronger regularity on the true solution, $u$.
This is not ideal as the singular potential limits the known regularity of $u$.
For contrast, we refer the reader to the analysis of a fully discrete ESFEM scheme for the Cahn-Hilliard equation with a regular potential \cite{elliott2024fully} where this approach is not required.

Given our fully discrete solutions $U_h^n, W_h^n$, we define piecewise linear functions
\begin{gather*}
	U_h^\tau(t) := \left(\frac{t-t_{n-1}}{\tau}\right) \Phi_t^h \Phi_{-t_{n}}^h U_h^n + \left(\frac{t_n - t}{\tau}\right) \Phi_t^h \Phi_{-t_{n-1}}^h U_h^{n-1},\\
	W_h^\tau(t) := \left(\frac{t-t_{n-1}}{\tau}\right) \Phi_t^h \Phi_{-t_{n}}^h W_h^n + \left(\frac{t_n - t}{\tau}\right) \Phi_t^h \Phi_{-t_{n-1}}^h W_h^{n-1},
\end{gather*}
for $t \in [t_{n-1},t_{n}]$.
It is straightforward to see from the transport property that
\[ \matdev_h U_h^\tau = \Phi_t^h \Phi_{-t_n}^h \matdevtau U_h^n = \underline{\matdevtau U_h^n}(t). \]
Similarly we define piecewise constant functions
\begin{gather*}
	\underline{U_h}(t) := \Phi_t^h \Phi_{-t_n}^h U_h^n = \underline{U_h^n}(t),\\
	\underline{W_h}(t) := \Phi_t^h \Phi_{-t_n}^h W_h^n = \underline{W_h^n}(t),
\end{gather*}
for $t \in (t_{n-1}, t_{n})$.
One can define $U_h^{\tau,\delta}, W_h^{\tau,\delta}, \underline{\Udh}, \underline{\Wdh}$ similarly.
The fully discrete error bound will be established by considering the decomposition
\[u^{-\ell} - \underline{U_h} = \underbrace{[u^{-\ell} - (u^{\delta})^{-\ell}]}_{\mathcal{O}(\delta)} + \underbrace{[(u^{\delta})^{-\ell} - U_h^\delta]}_{\mathcal{O}\left(\frac{h^4 \log\left(\frac{1}{h}\right)}{\delta^2}\right)} + \underbrace{[\Udh - \underline{\Udh}]}_{\mathcal{O}(?)} - \underbrace{[\underline{\Udh}-\underline{U_h}]}_{\mathcal{O}(?)},\]
where we have bounded all but the last two terms.

We now observe that these piecewise functions solve a perturbed version of the semi-discrete system \eqref{fecheqn1}, \eqref{fecheqn2}.
Specifically we have that, since $U_h^n, W_h^n$ solve \eqref{fullydisceqn1}, \eqref{fullydisceqn2} for all $\phi_h^n \in S_h^n$, and almost all $t \in [0,T]$, one has
\begin{align}
	\label{defect equation}
	\begin{gathered}
	\intm(\matdev_h U_h^\tau, \phi_h) + g_h(I_h(\underline{U_h}\phi_h),1) + a_h(\underline{W_h}, \phi_h) = \sum_{i=1}^3 D_i,\\
	\intm(\underline{W_h}, \phi_h(t)) = \varepsilon a_h(\underline{U_h}, \phi_h(t)) + \frac{\theta}{2 \varepsilon} \intm(I_hf(\underline{U_h}), \phi_h(t)) - \frac{1}{\varepsilon}\intm(\underline{U_h}, \phi_h(t)) + \sum_{i=4}^7 D_i,
	\end{gathered}
\end{align}
for some defects $D_i = D_i(t; \phi_h)$.
For $t \in (t_{n-1},t_n]$, these defects are given by
\begin{gather*}
	D_1(t; \phi_h) = \intm(\matdev_h U_h^\tau, \phi_h(t)) - \intm(\matdevtau U_h^n, \phi_h^n),\\
	D_2(t; \phi_h) = g_h(I_h(\underline{U_h}\phi_h(t)),1) - \frac{1}{\tau} \left( \intm(\overline{U_h^{n-1}},\phi_h^n)- \intm(U_h^{n-1},\underline{\phi_h^n}) \right),\\
	D_3(t; \phi_h)= a_h(\underline{W_h}, \phi_h(t)) - a_h(W_h^n, \phi_h^n),\\
	D_4(t; \phi_h)=\intm(\underline{W_h}, \phi_h(t)) - \intm(W_h^n, \phi_h^n),\\
	D_5(t; \phi_h)=\varepsilon\left[ a_h(U_h^n, \phi_h^n)- a_h(\underline{U_h}, \phi_h^n(t))\right],\\
	D_6(t; \phi_h)=\frac{\theta}{2 \varepsilon} \left[ \intm(I_h f(U_h^n), \phi_h^n) - \intm(I_hf(\underline{U_h}), \phi_h(t))\right],\\
	D_7(t; \phi_h)=\frac{1}{\varepsilon}\left[\intm(U_h^n, \phi_h^n)- \intm(\underline{U_h}, \phi_h(t))  \right].
\end{gather*}
Here $\phi_h(t) \in S_h(t)$, and $\phi_h^n \in S_h^n$ is chosen as $\phi_h^n = \Phi_{t_n}^h \Phi_{-t}^h \phi_h(t)$.
We bound these defects in the following lemma.

\begin{lemma}
	\label{defect bounds}
	$D_1,...,D_7$ as defined above are bounded as
	\begin{gather*}
		|D_1(t;\phi_h)| \leq C \tau \|\matdevtau U_h^n\|_{L^2(\Gamma_h^n)} \|\phi_h(t)\|_{L^2(\Gamma_h(t))},\\
		|D_2(t;\phi_h)| \leq C \tau \left( \|U_h^n \|_{L^2(\Gamma_h^n)} + \|\matdevtau U_h^n \|_{L^2(\Gamma_h^n)} \right)\|\phi_h(t)\|_{L^2(\Gamma_h(t))},\\
		|D_3(t;\phi_h)| \leq C \tau \|\gradgh W_h^n\|_{L^2(\Gamma_h^n)}\|\gradgh\phi_h(t)\|_{L^2(\Gamma_h(t))},\\
		|D_4(t;\phi_h)| \leq C \tau \| W_h^n\|_{L^2(\Gamma_h^n)}\|\phi_h(t)\|_{L^2(\Gamma_h(t))},\\
		|D_5(t;\phi_h)| \leq C \tau \|\gradgh U_h^n\|_{L^2(\Gamma_h^n)}\|\gradgh\phi_h(t)\|_{L^2(\Gamma_h(t))},\\
		|D_6(t;\phi_h)| \leq C \tau \| I_h f(U_h^n)\|_{L^2(\Gamma_h^n)}\|\phi_h(t)\|_{L^2(\Gamma_h(t))},\\
		|D_7(t;\phi_h)| \leq C \tau \| U_h^n\|_{L^2(\Gamma_h^n)}\|\phi_h(t)\|_{L^2(\Gamma_h(t))}
	\end{gather*}
	where $C$ denotes a constant independent of $t, h, \tau$.
\end{lemma}
\begin{proof}
	Most of these bounds are readily shown from \eqref{timeperturb1}, \eqref{timeperturb3} and \eqref{intmbound0} where necessary, and we note that for $\tau$ sufficiently small one has
	\[ \|\phi_h^n\|_{L^2(\Gamma_h^n)} \leq C\|\phi_h(t)\|_{L^2(\Gamma_h(t))}, \qquad \|\gradgh\phi_h^n\|_{L^2(\Gamma_h^n)} \leq C\|\gradgh\phi_h(t)\|_{L^2(\Gamma_h(t))} \]
	by similar arguments to \eqref{timenorm3}, \eqref{timenorm4}.
	
	The only bound which isn't immediate is that of $D_2$.
	To bound $D_2$ we firstly notice that
	\begin{align*}
		\frac{1}{\tau} \left( \intm(\overline{U_h^{n-1}},\phi_h^n)- \intm(U_h^{n-1},\underline{\phi_h^n}) \right) &= \frac{1}{\tau} \int_{t_{n-1}}^{t_n} \frac{d}{dt} \intm(\overline{U_h^{n-1}}(s),\underline{\phi_h^n}(s)) \, ds\\
		&= \frac{1}{\tau} \int_{t_{n-1}}^{t_n} g_h(I_h(\overline{U_h^{n-1}}(s)\underline{\phi_h^n}(s)),1) \, ds
	\end{align*}
	from which we see
	\[ D_2(t; \phi_h) = \frac{1}{\tau} \int_{t_{n-1}}^{t_n} g_h(I_h(\underline{U_h}\underline{\phi_h^n}(t)),1) - g_h(I_h(\overline{U_h^{n-1}}(s)\underline{\phi_h^n}(s)),1) \, ds. \]
	We then write this as
	\begin{multline*}
		 D_2(t; \phi_h) = \frac{1}{\tau} \int_{t_{n-1}}^{t_n} g_h(I_h(\underline{U_h}(t)\underline{\phi_h^n}(t)),1) - g_h(I_h(\underline{U_h}(s)\underline{\phi_h^n}(s)),1) \, ds\\
		 +  \frac{1}{\tau} \int_{t_{n-1}}^{t_n} g_h(I_h(\underline{U_h}(s)\underline{\phi_h^n}(s))- I_h(\overline{U_h^{n-1}}(s)\underline{\phi_h^n}(s)),1) \, ds,
	\end{multline*}
	which one readily sees is
	\begin{align}
		D_2(t; \phi_h) = \frac{1}{\tau} \int_{t_{n-1}}^{t_n} \int_{s}^{t} \frac{d}{dr} g_h(I_h(\underline{U_h}(r)\underline{\phi_h^n}(r)),1) \, dr  \, ds +  \int_{t_{n-1}}^{t_n} g_h(I_h(\underline{\matdevtau U_h^n}(s)\underline{\phi_h^n}(s)), 1) \, ds . \label{defectpf}
	\end{align}
	Noting that $\matdev_h I_h(\underline{U_h}\underline{\phi_h^n}) = 0$ one can readily show that
	\[ \frac{d}{dr} g_h(I_h(\underline{U_h}(r)\underline{\phi_h^n}(r)),1) = \int_{\Gamma_h(r)} I_h(\underline{U_h}(r)\underline{\phi_h^n}(r)) \left( \matdev_h (\gradgh \cdot V_h) +  (\gradgh \cdot V_h)^2 \right). \]
	We do not show this but note that it follows from similar logic to \cite[Lemma 5.6]{dziuk2013finite}.
	Then as $V_h = I_h V^{-\ell}$, our assumed smoothness on $V, \matdev V$ lets us conclude that
	\[ \sup_{t \in [0,T]} \sup_{\Gamma_h(t)} \left| \matdev_h (\gradgh \cdot V_h) +  (\gradgh \cdot V_h)^2 \right| \leq C, \]
	for some constant $C$ independent of $h$.
	Thus using this uniform bound in \eqref{defectpf}, and using \eqref{timenorm1} where appropriate, one readily concludes the claimed bound for $|D_2(t;\phi_h)|$.
\end{proof}

One can show that $U_h^{\tau,\delta}, \underline{\Udh}, \underline{\Wdh}$ solve a similar system with defects $D_i^\delta$, $i=1,...7$.
Moreover, these defects obey almost identical bounds to those in the preceding lemma.

Next we show that $U_h^\tau$ and $\underline{U_h}$ remain close, in appropriate norms, as is the content of the following lemma.
\begin{lemma}
\label{defectderivative lemma}
	\begin{gather}
		\label{defectderivative1}
		\int_{0}^T \|(U_h^\tau - \underline{U_h})\|_{L^2(\Gamma_h(t))}^2 + \sqrt{\tau} \|\gradgh(U_h^\tau - \underline{U_h})\|_{L^2(\Gamma_h(t))}^2 \leq C \tau^\frac{3}{2},\\
		\int_0^T \inormh{(U_h^\tau - \underline{U_h}) - \mval{(U_h^\tau - \underline{U_h})}{\Gamma_h(t)}}^2 \leq C \tau^2, \label{defectderivative2}
	\end{gather}
	for some constant $C$ independent of $\tau, h$.
\end{lemma}
\begin{proof}
	We start by showing the $L^2$ bound of \eqref{defectderivative1}.
	By definition one has
	\begin{align*}
		\int_{0}^T \|U_h^\tau - \underline{U_h}\|_{L^2(\Gamma_h(t))}^2 &= \sum_{n=1}^{N_T} \int_{t_{n-1}}^{t_n} \left\| \left(\frac{t-t_{n-1}}{\tau} - 1\right) \Phi_t^h \Phi_{-t_{n}}^h U_h^n + \left(\frac{t_n - t}{\tau}\right) \Phi_t^h \Phi_{-t_{n-1}}^h U_h^{n-1} \right\|_{L^2(\Gamma_h(t))}^2\\
		& \leq \sum_{n=1}^{N_T} \int_{t_{n-1}}^{t_n} \left\| \Phi_t^h \Phi_{-t_{n}}^h U_h^n - \Phi_t^h \Phi_{-t_{n-1}}^h U_h^{n-1} \right\|_{L^2(\Gamma_h(t))}^2,
	\end{align*}
	and so using \eqref{timeperturb1} and (a $\delta \rightarrow 0$ analogue of) \eqref{fdisc derivative bound1} one obtains
	\begin{align*}
		\int_{0}^T\|U_h^\tau - \underline{U_h}\|_{L^2(\Gamma_h(t))}^2 \leq C \tau \sum_{n=1}^{N_T} \| U_h^n - \overline{U_h^{n-1}}\|_{L^2(\Gamma_h^n)}^2 = C \tau^3 \sum_{n=1}^{N_T} \| \matdevtau U_h^n \|_{L^2(\Gamma_h^n)}^2 \leq C \tau^\frac{3}{2}.
	\end{align*}
	The $H^1$ bound is similar, but uses the bound on $\gradgh \matdevtau U_h^n$ (which one obtains in the $\delta \rightarrow 0$ limit of \eqref{fdiscregstability}) instead.
	
	The bound for \eqref{defectderivative2} is more subtle.
    To begin we compute, by a similar calculation to the above, that
    \begin{align*}
        \int_0^T \inormh{(U_h^\tau - \underline{U_h}) - \mval{(U_h^\tau - \underline{U_h}) }{\Gamma_h(t)}}^2
        \leq \tau^2 \sum_{n=1}^{N_T} \int_{t_{n-1}}^{t_n} \left\| \gradgh \intinvsh \left(\Phi_t^h \Phi_{-t_n}^h\matdevtau U_h^n - \mval{\Phi_t^h \Phi_{-t_n}^h\matdevtau U_h^n}{\Gamma_h(t)} \right) \right\|_{L^2(\Gamma_h(t)))}^2.
    \end{align*}
    The issue arises as we cannot immediately apply \eqref{timenorm2}, analogously to the previous result, as $\intinvsh$ does not commute with pullbacks/pushforwards.
	To mitigate this, the idea is to write 
	\begin{align*}
		\intinvsh \Phi_t^h \Phi_{-t_n}^h = \Phi_t^h \Phi_{-t_n}^h \intinvsh + \left(\intinvsh \Phi_t^h \Phi_{-t_n}^h - \Phi_t^h \Phi_{-t_n}^h \intinvsh \right),
	\end{align*}
    where we have omitted the mean value terms for simplicity.
	This first operator can now be dealt with by using \eqref{timenorm2}, and (a $\delta \rightarrow 0$ analogue of) \eqref{fdisc derivative bound2}.
    For the bracketed term we obtain a similar bound by applying Lemma \ref{invlap commutator} with $z_h^n = \matdevtau U_h^n$.
	We omit the calculations as they are tedious, and provide nothing that we have not remarked already --- but a patient reader can indeed verify that
	\begin{align*}
		\int_0^T \inormh{(U_h^\tau - \underline{U_h}) - \mval{(U_h^\tau - \underline{U_h})}{\Gamma_h(t)}}^2 \leq C \tau^3 \sum_{n=1}^{N_T} \inormh{\matdevtau U_h^n - \mval{\matdevtau U_h^n}{\Gamma_h^n}}^2 + C \tau^5 \sum_{n=1}^{N_T} \left\|\matdevtau U_h^n\right\|_{L^2(\Gamma_h^n)}^2.
	\end{align*}
	\eqref{defectderivative2} now follows from ($\delta \rightarrow 0$ versions of) \eqref{fdisc derivative bound1} and \eqref{fdisc derivative bound2}.
\end{proof}

\begin{remark}
	It is a straightforward modification of this proof to also verify that
	\begin{align}
		\int_{0}^T \inormh{\matdev_h U_h^\tau - \mval{\matdev_h U_h^\tau }{\Gamma_h(t)}}^2 \leq C, \label{pw linear derivative bound}
	\end{align}
	for a constant, $C$, independent of $h,\tau$.
    We do not elaborate on further details on this proof.
    Note also that by definition of $U_h^\tau$ we can also write $\matdev_h U_h^\tau = \underline{\matdevtau U_h^n}(t)$, and we shall use these interchangeably. 
\end{remark}

We also have the following bounds on the mean values,
\begin{gather}
	\left|\int_{\Gamma_h(t)} U_h^\tau - \int_{\Gamma_h(0)} U_{h,0}\right| \leq C \tau, \qquad \left|\int_{\Gamma_h(t)} \underline{U_h} - \int_{\Gamma_h(0)} U_{h,0}\right| \leq C \tau.
	\label{defectbound2}
\end{gather}
To see this we first observe that for $t \in (t_{n-1},t_n]$ we may use \eqref{timeperturb1} to see
\[\left|\int_{\Gamma_h(t)} \underline{U_h} - \int_{\Gamma_h(0)} U_{h,0}\right| = \left|\int_{\Gamma_h(t)} \underline{U_h}(t) -\int_{\Gamma_h^n} U_h^n\right| \leq C \tau \|U_h^n\|_{L^2(\Gamma_h^n)} \leq C\tau. \]
The bound for $\left|\int_{\Gamma_h(t)} U_h^\tau - \int_{\Gamma_h(0)} U_{h,0}\right|$ can be shown similarly by writing
\[ \left|\int_{\Gamma_h(t)} U_h^\tau - \int_{\Gamma_h(0)} U_{h,0}\right|\leq \left(\frac{t-t_{n-1}}{\tau}\right)\left|\int_{\Gamma_h(t)} \underline{U_h^n}(t) - \int_{\Gamma_h(0)} U_{h,0}\right| + \left(\frac{t_n - t}{\tau}\right)\left|\int_{\Gamma_h(t)} \overline{U_h^{n-1}}(t) - \int_{\Gamma_h(0)} U_{h,0}\right| .  \]
With these bounds at hand we can now show a fully discrete analogue of \eqref{deltaerror2}.
\begin{lemma}
	The piecewise functions defined at the beginning of this section are such that, for sufficiently small $h, \tau, \delta$, with $\tau \leq \frac{\varepsilon^3}{2}$ we have
	\begin{align}
		\varepsilon \int_{0}^T \| \gradgh (\underline{U_h} - \underline{\Udh}) \|_{L^2(\Gamma_h(t))}^2 \leq C\delta, \label{deltaerror6}
	\end{align}
	for some constant $C$ independent of $h, \tau, \delta$.
\end{lemma}
\begin{proof}
	We begin by finding some estimates at the time discrete level.
	Firstly we define $\widehat{\Udhn{n}} = U_h^n - \Udhn{n}, \widehat{\Wdhn{n}} = W_h^n - \Wdhn{n}$ and we have from \eqref{fullydisceqn1}, \eqref{fullydisceqn2} and \eqref{regfullydisceqn1}, \eqref{regfullydisceqn2} that
	\begin{gather}
		\frac{1}{\tau} \left(\intm(\widehat{\Udhn{n}},\phi_h^n) - \intm(\widehat{\Udhn{n-1}},\underline{\phi_h^n})\right) + a_h(\widehat{\Wdhn{n}},\phi_h^n) = 0,\label{defectpf1}\\
		\intm(\widehat{\Wdhn{n}}, \phi_h^n) = \varepsilon a_h(\widehat{\Udhn{n}}, \phi_h^n) + \frac{\theta}{2 \varepsilon} \intm(I_h f(U_h^n) - I_h \fd(\Udhn{n}), \phi_h^n) - \frac{1}{\varepsilon} \intm(\widehat{\Udhn{n}}, \phi_h^n),\label{defectpf2}
	\end{gather}
	for all $\phi_h^n \in S_h^n$.
	We test \eqref{defectpf1} with $\intinvsh \widehat{\Udhn{n}}$ and \eqref{defectpf2} with $\widehat{\Udhn{n}}$ to obtain
    \begin{multline*}
        \frac{1}{\tau} \left(\intm(\widehat{\Udhn{n}},\intinvsh\widehat{\Udhn{n}}) - \intm(\widehat{\Udhn{n-1}},\underline{\intinvsh\widehat{\Udhn{n}}})\right) + \varepsilon a_h(\widehat{\Udhn{n}}, \widehat{\Udhn{n}}) + \frac{\theta}{2 \varepsilon} \intm(I_h f(U_h^n) - I_h \fd(\Udhn{n}), \widehat{\Udhn{n}})\\
        = \frac{1}{\varepsilon} \intm(\widehat{\Udhn{n}}, \widehat{\Udhn{n}}).
    \end{multline*}
    One can then verify (see also the proof of \cite[Proposition 3.1]{elliott2024fully}) that
    \begin{align*}
        \frac{1}{\tau} \left(\intm(\widehat{\Udhn{n}},\intinvsh\widehat{\Udhn{n}}) - \intm(\widehat{\Udhn{n-1}},\underline{\intinvsh\widehat{\Udhn{n}}})\right) &= \frac{1}{\tau}a_h(\intinvsh\widehat{\Udhn{n}} - \intinvsh\widehat{U_{h,+}^{n-1,\delta}},\intinvsh \widehat{\Udhn{n}} )\\
        &= \frac{1}{2 \tau} \left( \inormh{\widehat{\Udhn{n}}}^2 - \inormh{\widehat{U_{h,+}^{n-1,\delta}}}^2 + \inormh{\widehat{\Udhn{n}} - \widehat{U_{h,+}^{n-1,\delta}}}^2  \right).
    \end{align*}
    
    Using this, we find that
	\begin{align}
    \begin{aligned}
		&\inormh{\widehat{\Udhn{n}}}^2 - \inormh{\widehat{\Udhn{n-1}}}^2 + \inormh{\widehat{\Udhn{n}} - \widehat{U_{h,+}^{n-1,\delta}}}^2 + 2\varepsilon \tau \|\gradgh \widehat{\Udhn{n}}\|_{L^2(\Gamma_h^n)}^2\\
		&+ \frac{\tau \theta}{ \varepsilon} \intm(I_h \fd(U_h^n) - I_h \fd(\Udhn{n}), \widehat{\Udhn{n}}) = \left[ \inormh{U_{h,+}^{n-1,\delta}}^2 - \inormh{\widehat{\Udhn{n-1}}}^2 \right] + \frac{2\tau}{\varepsilon} \|\widehat{\Udhn{n}}\|_{h,t_n}^2\\
        &+ \frac{\tau \theta}{ \varepsilon} \intm(I_h \fd(U_h^n) - I_h f(U_h^n), \widehat{\Udhn{n}}).
        \end{aligned}\label{defectpf3}
	\end{align}
    By mirroring the proof of \cite[Proposition 3.1]{elliott2024fully} (modulo the use of mass-lumping) one can show that
    \[ \left|\inormh{U_{h,+}^{n-1,\delta}}^2 - \inormh{\widehat{\Udhn{n-1}}}^2\right| \leq C \tau \inormh{\widehat{\Udhn{n-1}}}^2,\]
    for a constant, $C$, independent of $h, \tau, \delta$.
    Likewise we observe that
    \[\frac{2\tau}{\varepsilon} \|\widehat{\Udhn{n}}\|_{h,t_n}^2 = \frac{2\tau}{\varepsilon}a_h(\widehat{\Udhn{n}}, \intinvsh \widehat{\Udhn{n}}) \leq \varepsilon \tau \|\gradgh \widehat{\Udhn{n}}\|_{L^2(\Gamma_h(t))}^2 + \frac{\tau}{\varepsilon^3} \inormh{\widehat{\Udhn{n}}}^2.\]
	Now one argues as in Theorem \ref{u delta error theorem}, where we omit the details, to find that \eqref{defectpf3} yields
	\begin{multline}
		\inormh{\widehat{\Udhn{n}}}^2 - \inormh{\widehat{\Udhn{n-1}}}^2 + \inormh{\widehat{\Udhn{n}} - \widehat{U_{h,+}^{n-1,\delta}}}^2 + 2 \varepsilon \tau \|\gradgh \widehat{\Udhn{n}}\|_{L^2(\Gamma_h^n)}^2 \leq \frac{\tau}{\varepsilon^3} \inormh{\widehat{\Udhn{n}}}^2
		\\+ C \tau
		\inormh{\widehat{\Udhn{n-1}}}^2 + C \tau \delta\|I_h f(U_h^n)\|_{L^2(\Gamma_h^n)}^2 . \label{defectpf4}
	\end{multline}
	Hence by summing over $n$, noting that $\frac{\tau}{\varepsilon^3} \leq \frac{1}{2}$ and using (a $\delta \rightarrow 0$ analogue of) \eqref{fdiscdpot2}, we may use a discrete Gr\"onwall inequality (noting that $\widehat{\Udhn{0}} = 0$) to find
	\[ \varepsilon \tau \sum_{n=1}^{N_T} \| \gradgh \widehat{\Udhn{n}} \|_{L^2(\Gamma_h^n)}^2 \leq C \delta, \]
	and we use this to obtain the desired bound.
	One finds from this and \eqref{timenorm2}
	\begin{align*}
		\varepsilon \int_{0}^T \| \gradgh (\underline{U_h} - \underline{\Udh}) \|_{L^2(\Gamma_h(t))}^2 &= \varepsilon \sum_{n=1}^{N_T} \int_{t_{n-1}}^{t_n} \| \gradgh \Phi_t^h \Phi^h_{-t_n}\widehat{\Udhn{n}} \|_{L^2(\Gamma_h(t))}^2 \\
		&\leq C \varepsilon \tau \sum_{n=1}^{N_T} \| \gradgh \widehat{\Udhn{n}} \|_{L^2(\Gamma_h^n)}^2 \leq C \delta.
	\end{align*}
\end{proof}

There is now one more error bound to show, from which we infer our final error bound.
This is the content of the following lemma.
\begin{lemma}
	Let $\Udh$ be the solution from \eqref{discregch}, \eqref{discregch2}, and $\underline{\Udh}, U_h^{\tau,\delta}$ the piecewise functions defined above.
	Then for sufficiently small $ h, \tau, \delta$,
	\begin{align}
		\varepsilon \int_{0}^T \| \gradgh (\Udh - \underline{\Udh}) \|_{L^2(\Gamma_h(t))}^2 + \sup_{t \in [0,T]} \inormh{\Udh - U_h^{\tau, \delta}}^2 \leq C \left( \tau +  \frac{\tau^2}{\delta^2} \right) ,\label{deltaerror7}
	\end{align}
	for some constant $C$ independent of $h, \tau, \delta$.
\end{lemma}
\begin{proof}
As usual we define some shorthand notation to be used throughout,
\begin{gather*}
	E_{u,h}^{\tau,\delta} := \Udh - U_h^{\tau,\delta}, \qquad E_{u,h}^\delta := \Udh - \underline{\Udh}, \qquad {E_{w,h}^\delta} := \Wdh - \underline{\Wdh},
\end{gather*}
from which it is clear that $E_{u,h}^{\tau,\delta} - E_{u,h}^\delta = \underline{\Udh} - U_h^{\tau,\delta}$.
Unlike the error analysis on a stationary domain (see for instance \cite{barrett1995error}) these functions \emph{do not} have vanishing mean value.
As such we also define the functions
\begin{gather*}
    \widetilde{E_{u,h}^{\tau,\delta}} := E_{u,h}^{\tau,\delta} - \mval{E_{u,h}^{\tau,\delta}}{\Gamma_h(t)}, \qquad \widetilde{E_{u,h}^\delta} := E_{u,h}^\delta - \mval{E_{u,h}^\delta}{\Gamma_h(t)},
\end{gather*}
to which we may apply the inverse Laplacians from Appendix \ref{invlaps}.
Then subtracting the regularised version of \eqref{defect equation} from \eqref{discregch}, \eqref{discregch2} we obtain
\begin{multline}
	\label{defect equation 1}
	\intm(\matdev_h E_{u,h}^{\tau, \delta}, \phi_h) + g_h(I_h(E_{u,h}^{\tau, \delta}\phi_h),1) + a_h(E_{w,h}^{\delta}, \phi_h) =
	g_h\left(I_h\left(\underline{\Udh}\phi_h\right),1\right) - g_h(I_h(U_h^{\tau,\delta} \phi_h),1)\\
	- \sum_{k=1}^3 D_k^\delta(\phi_h),
\end{multline}
and
\begin{multline}
	\label{defect equation 2}
	\intm(E_{w,h}^{\delta}, \phi_h) = \varepsilon a_h(E_{u,h}^{\delta}, \phi_h) + \frac{\theta}{2 \varepsilon} \intm \left(I_h \fd(\Udh) - I_h\fd(\underline{\Udh}), \phi_h \right) - \frac{1}{\varepsilon}\intm(E_{u,h}^\delta, \phi_h)\\
	- \sum_{k=4}^7 D_k^\delta(\phi_h).
\end{multline}
We then test \eqref{defect equation 1} with $\phi_h = \intinvsh \widetilde{E_{u,h}^{\delta}}$, so that
\begin{align}
	\intm\left(\matdev_h E_{u,h}^{\tau, \delta}, \intinvsh \widetilde{E_{u,h}^{\tau, \delta}}\right) + g_h\left(I_h\left(E_{u,h}^{\tau, \delta} \intinvsh \widetilde{E_{u,h}^{\tau, \delta}}\right),1\right) + a_h\left(E_{w,h}^{\delta}, \intinvsh \widetilde{E_{u,h}^{\delta}}\right) = \sum_{k=1}^4 J_k, \label{defectpf6}
\end{align}
where
\begin{gather*}
	J_1 := \intm\left(\matdev_h E_{u,h}^{\tau,\delta},\intinvsh\left( \widetilde{E_{u,h}^{\tau, \delta}} - \widetilde{E_{u,h}^{\delta}}\right) \right)\\
	J_2 := g_h \left( I_h \left(E_{u,h}^{\tau,\delta}\intinvsh\left(\widetilde{E_{u,h}^{\tau, \delta}} - \widetilde{E_{u,h}^{\delta}}\right)  \right),1 \right),\\
	J_3 := g_h\left(I_h\left(\underline{\Udh}\intinvsh \widetilde{E_{u,h}^{\delta}}\right),1\right) - g_h\left(I_h(U_h^{\tau,\delta} \intinvsh \widetilde{E_{u,h}^{\delta}}),1\right),\\
	J_4 := -\sum_{i=1}^3 D_i^\delta \left(\intinvsh \widetilde{E_{u,h}^{\delta}}\right).
\end{gather*}
Then, as usual, we observe from \eqref{defect equation 2} and the definition of $\intinvsh$ that
\begin{multline*}
	a_h\left(E_{w,h}^{\delta}, \intinvsh \widetilde{E_{u,h}^{\delta}}\right) = \intm\left(E_{w,h}^{\delta}, \widetilde{E_{u,h}^{\delta}}\right) = \varepsilon a_h(E_{u,h}^{\delta}, E_{u,h}^{\delta})+ \frac{\theta}{2 \varepsilon} \intm \left(I_h \fd(\Udh)- I_h\fd(\underline{\Udh}),  \widetilde{E_{u,h}^{\delta}} \right)\\
    - \frac{1}{\varepsilon}\intm\left(E_{u,h}^\delta,  \widetilde{E_{u,h}^{\delta}}\right) - \sum_{i=4}^7 D_i^\delta\left( \widetilde{E_{u,h}^{\delta}}\right).
\end{multline*} 
Hence using this in \eqref{defectpf6} one has
\begin{multline}
	\intm\left(\matdev_h E_{u,h}^{\tau, \delta}, \intinvsh \widetilde{E_{u,h}^{\tau, \delta}}\right) + g_h\left(I_h(E_{u,h}^{\tau, \delta}\intinvsh \widetilde{E_{u,h}^{\tau, \delta}}),1\right) + \varepsilon \| \gradgh E_{u,h}^{\delta} \|_{L^2(\Gamma_h(t))}^2\\
	+ \frac{\theta}{2 \varepsilon} \intm \left(I_h \fd(\Udh) - I_h\fd(\underline{\Udh}), \widetilde{E_{u,h}^{\delta}} \right) = \frac{1}{\varepsilon}\intm\left(E_{u,h}^\delta, \widetilde{E_{u,h}^{\delta}}\right)+ \sum_{k=1}^5 J_k, \label{defectpf7}
\end{multline}
where
\begin{gather*}
	J_5 := \sum_{i=4}^7 D_i^\delta\left(\widetilde{E_{u,h}^{\delta}}\right).
\end{gather*}

To begin we firstly write
\begin{multline*}
	\intm\left(\matdev_h E_{u,h}^{\tau, \delta}, \intinvsh\widetilde{E_{u,h}^{\tau, \delta}}\right) + g_h\left(I_h(E_{u,h}^{\tau, \delta}\intinvsh \widetilde{E_{u,h}^{\tau, \delta}}),1 \right) = 	\intm\left(\matdev_h \widetilde{E_{u,h}^{\tau, \delta}}, \intinvsh\widetilde{E_{u,h}^{\tau, \delta}}\right) + g_h\left(I_h(\widetilde{E_{u,h}^{\tau, \delta}}\intinvsh \widetilde{E_{u,h}^{\tau, \delta}}),1 \right)\\
	+ \intm\left(\matdev_h \mval{E_{u,h}^{\tau, \delta}}{\Gamma_h(t)}, \intinvsh\widetilde{E_{u,h}^{\tau, \delta}}\right) + g_h\left(\mval{E_{u,h}^{\tau, \delta}}{\Gamma_h(t)}, \intinvsh \widetilde{E_{u,h}^{\tau, \delta}} \right),
\end{multline*}
where we note that $I_h (\intinvsh \widetilde{E_{u,h}^{\tau, \delta}}) = \intinvsh \widetilde{E_{u,h}^{\tau, \delta}}$.
Now from Proposition \ref{transport3}, Lemma \ref{transport5}, and the definition of $\intinvsh$, it is a straightforward calculation to see that
\begin{multline*}
	\intm\left(\matdev_h E_{u,h}^{\tau, \delta}, \intinvsh\widetilde{E_{u,h}^{\tau, \delta}}\right) + g_h\left(I_h(E_{u,h}^{\tau, \delta}\intinvsh \widetilde{E_{u,h}^{\tau, \delta}}),1 \right) = \frac{1}{2} \frac{d}{dt} \inormh{ \widetilde{E_{u,h}^{\tau, \delta}}}^2 - \frac{1}{2} b_h\left( \intinvsh \widetilde{E_{u,h}^{\tau, \delta}}, \intinvsh \widetilde{E_{u,h}^{\tau, \delta}} \right)\\
	+ \intm\left(\matdev_h \mval{E_{u,h}^{\tau, \delta}}{\Gamma_h(t)}, \intinvsh\widetilde{E_{u,h}^{\tau, \delta}}\right) + g_h\left(\mval{E_{u,h}^{\tau, \delta}}{\Gamma_h(t)}, \intinvsh \widetilde{E_{u,h}^{\tau, \delta}} \right)
\end{multline*}
We use this, and the monotonicity of $\fd$, in \eqref{defectpf7} to finally see that
\begin{multline}
		\frac{1}{2} \frac{d}{dt} \inormh{ \widetilde{E_{u,h}^{\tau, \delta}}}^2 + \varepsilon \| \gradgh E_{u,h}^{\delta} \|_{L^2(\Gamma_h(t))}^2 \leq \frac{\theta}{2 \varepsilon} \intm \left(I_h \fd(\Udh) - I_h\fd(\underline{\Udh}), \mval{E_{u,h}^{\delta}}{\Gamma_h(t)} \right)\\
		+ \frac{1}{\varepsilon}\intm\left(E_{u,h}^\delta, \widetilde{E_{u,h}^{\delta}}\right) + \frac{1}{2} b_h\left( \intinvsh \widetilde{E_{u,h}^{\tau, \delta}} , \intinvsh \widetilde{E_{u,h}^{\tau, \delta}} \right)\\
		+ \left| \intm\left(\matdev_h \mval{E_{u,h}^{\tau, \delta}}{\Gamma_h(t)}, \intinvsh\widetilde{E_{u,h}^{\tau, \delta}}\right) \right| + \left| g_h\left(\mval{E_{u,h}^{\tau, \delta}}{\Gamma_h(t)}, \intinvsh \widetilde{E_{u,h}^{\tau, \delta}} \right) \right| + \sum_{i=1}^5 |J_i|,
	\label{defectpf8}
\end{multline}
and all that remains is to bound these terms accordingly.

We begin with the useful observation that $\left| \mval{E_{u,h}^\delta}{\Gamma_h(t)}\right| \leq C \tau$.
To see this is true we can calculate
\[ \left|\mval{E_{u,h}^\delta}{\Gamma_h(t)}\right| = \frac{1}{|\Gamma_h(t)|} \left|\int_{\Gamma_h(t)} \Udh(t) - \int_{\Gamma_h(t)} \underline{\Udh}(t) \right| = \frac{1}{|\Gamma_h(t)|} \left|\int_{\Gamma_h(0)} U_{h,0} - \int_{\Gamma_h(t)} \underline{\Udh}(t) \right| \leq C \tau, \]
where we have used (a $\delta \neq 0$ version of) \eqref{defectbound2}.
A similar calculation yields $\left| \mval{E_{u,h}^{\tau,\delta}}{\Gamma_h(t)}\right| \leq C \tau$.
Hence using this observation, \eqref{phidbound2}, and the Poincar\'e inequality we find
\begin{align}
	\begin{aligned}
\left| \frac{\theta}{2 \varepsilon} \intm \left(I_h \fd(\Udh) - I_h\fd(\underline{\Udh}), \mval{E_{u,h}^{\delta}}{\Gamma_h(t)} \right) \right| &\leq \frac{C \tau}{\delta} \| E_{u,h}^\delta\|_{L^2(\Gamma_h(t))}\\
&\leq \frac{C \tau}{\delta} \| \gradgh E_{u,h}^\delta\|_{L^2(\Gamma_h(t))} + \frac{C \tau}{\delta} \left| \mval{E_{u,h}^\delta}{\Gamma_h(t)} \right|\\
& \leq \frac{\varepsilon}{4} \| \gradgh E_{u,h}^\delta\|_{L^2(\Gamma_h(t))}^2 + \frac{C \tau^2}{\delta^2},
\end{aligned}
\label{defectpf9}
\end{align}
where we have used Young's inequality for the final inequality.
Similarly one finds that
\begin{align}
	\left| g_h\left(\mval{E_{u,h}^{\tau, \delta}}{\Gamma_h(t)}, \intinvsh \widetilde{E_{u,h}^{\tau, \delta}} \right) \right| \leq C \tau \inormh{\widetilde{E_{u,h}^{\tau, \delta}}}, \label{defectpf10}
\end{align}
where we have also used \eqref{intmbound0} and the Poincar\'e inequality.
Since $\matdev_h \mval{E_{u,h}^{\tau, \delta}}{\Gamma_h(t)}$ depends only on $t$ we find
\[\intm\left(\matdev_h \mval{E_{u,h}^{\tau, \delta}}{\Gamma_h(t)}, \intinvsh \widetilde{E_{u,h}^{\tau, \delta}}\right) = 0.\]
Next we use the definition of $\intinvsh$ and \eqref{invlapineq3} to find
\begin{align}
	\begin{aligned}
	\frac{1}{\varepsilon}\intm\left(E_{u,h}^\delta, \widetilde{E_{u,h}^{\delta}}\right) &=\frac{1}{\varepsilon}a_h\left(E_{u,h}^\delta, \intinvsh \widetilde{E_{u,h}^{\delta}}\right)\\
	&\leq \frac{\varepsilon}{4} \|\gradgh E_{u,h}^\delta\|_{L^2(\Gamma_h(t))}^2 + C \inormh{\widetilde{E_{u,h}^{\delta}}}^2,
	\end{aligned}
	\label{defectpf11}
\end{align}
and we observe that
\[ \inormh{\widetilde{E_{u,h}^{\delta}} } \leq \inormh{\widetilde{E_{u,h}^{\tau,\delta}}} + \inormh{\underline{\Udh} - U_h^{\tau,\delta} - \mval{\left( \underline{\Udh} - U_h^{\tau,\delta} \right)}{\Gamma_h(t)}}, \]
where this rightmost term will be bounded by using (a $\delta \neq 0$ analogue of) \eqref{defectderivative2}.
From our smoothness assumptions on $V$ we find that, as usual,
\begin{align}
	\frac{1}{2} b_h\left( \intinvsh \widetilde{E_{u,h}^{\tau, \delta}}, \intinvsh \widetilde{E_{u,h}^{\tau, \delta}} \right) \leq C \inormh{\widetilde{E_{u,h}^{\tau, \delta}}}^2.\label{defectpf12}
\end{align}
All that remains is to bound the $J_i$ terms.
To bound $J_1$ we first split $\matdev_h E_{u,h}^{\tau,\delta} = \matdev_h\Udh - \matdev_h U_{h}^{\tau,\delta}$ and observe that
\begin{align*}
	J_1 &= \intm\left(\matdev_h \Udh,\intinvsh (\widetilde{E_{u,h}^{\tau, \delta}} - \widetilde{E_{u,h}^{\delta}}) \right) + \intm\left(\underline{\matdevtau \Udhn{n}}(t),\intinvsh (\widetilde{E_{u,h}^{\tau, \delta}} - \widetilde{E_{u,h}^{\delta}}) \right)\\
	&= \intm\left(\matdev_h \Udh,\intinvsh (\widetilde{E_{u,h}^{\tau, \delta}} - \widetilde{E_{u,h}^{\delta}}) \right) + \intm\left(\underline{\matdevtau \Udhn{n}}(t) - \mval{\underline{\matdevtau \Udhn{n}}(t)}{\Gamma_h(t)},\intinvsh (\widetilde{E_{u,h}^{\tau, \delta}} - \widetilde{E_{u,h}^{\delta}}) \right),
\end{align*}
where we have used the fact that $\matdev_h U_h^{\tau, \delta} = \underline{\matdevtau \Udhn{n} }(t)$.
Now we use the definitions of $\|\cdot\|_{H^{-1}(\Gamma_h(t))}$ and $\intinvsh$ to see
\begin{align*}
	|J_1| \leq C \left( \|\matdev_h\Udh\|_{H^{-1}(\Gamma_h(t))} + \inormh{\underline{\matdevtau \Udhn{n}}(t) - \mval{\underline{\matdevtau \Udhn{n}}(t)}{\Gamma_h(t)}} \right) \inormh{\underline{\Udh} - U_h^{\tau,\delta} - \mval{\left( \underline{\Udh} - U_h^{\tau,\delta} \right)}{\Gamma_h(t)}}
\end{align*}
where we now use H\"older's inequality, Lemma \ref{semi discrete derivative lemma}, and ($\delta \neq 0$ analogues of) \eqref{defectderivative2}, \eqref{pw linear derivative bound} to see that
\begin{align}
	\int_0^T |J_1| \leq C\tau \left( \int_0^T \|\matdev_h\Udh\|_{H^{-1}(\Gamma_h(t))}^2 + \int_0^T \inormh{\underline{\matdevtau \Udhn{n}}(t) - \mval{\underline{\matdevtau \Udhn{n}}(t)}{\Gamma_h(t)}}^2 \right)^{\frac{1}{2}} \leq C \tau. \label{defectpf13}
\end{align}
To bound $J_2$ we firstly write
\begin{align*}
	J_2 = g_h \left( I_h \left(\widetilde{E_{u,h}^{\tau,\delta}}\intinvsh\left(\widetilde{E_{u,h}^{\tau, \delta}} - \widetilde{E_{u,h}^{\delta}}\right) \right),1 \right) + g_h \left(  \mval{E_{u,h}^{\tau,\delta}}{\Gamma_h(t)},\intinvsh\left(\widetilde{E_{u,h}^{\tau, \delta}} - \widetilde{E_{u,h}^{\delta}}\right) \right).
\end{align*}
Now using the smoothness of $V$, the bound on $\mval{E_{u,h}^{\tau,\delta}}{\Gamma_h(t)}$, and Poincar\'e's inequality we have
\begin{align*}
	|J_2| \leq C\|\widetilde{E_{u,h}^{\tau,\delta}}\|_{L^2(\Gamma_h(t))} \inormh{\widetilde{E_{u,h}^{\tau, \delta}} - \widetilde{E_{u,h}^{\delta}}} +C\tau \inormh{\widetilde{E_{u,h}^{\tau, \delta}} - \widetilde{E_{u,h}^{\delta}}}
\end{align*}
Hence we now use \eqref{invlapineq3}, Young's inequality, \eqref{defectderivative2}, and our CFL condition to see that
\begin{align}
	\begin{aligned}
	\int_0^T |J_2| &\leq C \tau^2 +  \frac{C}{h}\int_0^T\inormh{\widetilde{E_{u,h}^{\tau,\delta}}} \inormh{\widetilde{E_{u,h}^{\tau, \delta}} - \widetilde{E_{u,h}^{\delta}}}\\
	&\leq C\int_0^T\inormh{\widetilde{E_{u,h}^{\tau,\delta}}}^2 + C \tau^2 + \underbrace{\frac{C\tau^2}{h^2}}_{\leq C\tau}.
	\end{aligned}
	\label{defectpf14}
\end{align}

To bound $|J_3|$ we find that using Poincar\'e's inequality
\begin{align*}
	|J_3|\leq C \|\underline{U_h^\delta} - U^{\tau,\delta}_h\|_{L^2(\Gamma_h(t))} \inormh{\widetilde{E_{u,h}^{\delta}}}
\end{align*}
from which we find that
\begin{align*}
	|J_3|\leq C \|\underline{U_h^\delta} - U^{\tau,\delta}_h\|_{L^2(\Gamma_h(t))} \inormh{\widetilde{E_{u,h}^{\tau,\delta}}} + C \|\underline{U_h^\delta} - U^{\tau,\delta}_h\|_{L^2(\Gamma_h(t))} \inormh{\widetilde{E_{u,h}^{\tau, \delta}}-\widetilde{E_{u,h}^{\delta}}}
\end{align*}
and hence we use \eqref{defectderivative1}, \eqref{defectderivative2}, and Young's inequality to see that
\begin{align}
	\int_0^T |J_3| \leq C \tau^{\frac{3}{2}} + C\int_0^T \inormh{\widetilde{E_{u,h}^{\tau,\delta}}}^2. \label{defectpf15}
\end{align}
From Lemma \ref{defect bounds} one immediately has that (after using the Poincar\'e inequality where appropriate)
\begin{align*}
	|J_4| &\leq C\tau \left( \|\Udhn{n} \|_{L^2(\Gamma_h^n)} + \|\matdevtau \Udhn{n} \|_{L^2(\Gamma_h^n)} + \|\gradgh \Wdhn{n} \|_{L^2(\Gamma_h^n)}\right)\inormh{\widetilde{E_{u,h}^\delta}}\\
	&\leq C\tau \left( \|\Udhn{n} \|_{L^2(\Gamma_h^n)} + \|\matdevtau \Udhn{n} \|_{L^2(\Gamma_h^n)} + \|\gradgh \Wdhn{n} \|_{L^2(\Gamma_h^n)}\right)\inormh{\widetilde{E_{u,h}^{\tau,\delta}}}\\
	&+ C\tau \left( \|\Udhn{n} \|_{L^2(\Gamma_h^n)} + \|\matdevtau \Udhn{n} \|_{L^2(\Gamma_h^n)} + \|\gradgh \Wdhn{n} \|_{L^2(\Gamma_h^n)}\right)\inormh{\widetilde{E_{u,h}^{\tau,\delta}} - \widetilde{E_{u,h}^\delta}},
\end{align*}
and similarly that
\begin{align*}
	|J_5| \leq C \tau \left( \|\Udhn{n} \|_{H^1(\Gamma_h^n)} + \|I_h \fd(\Udhn{n})\|_{L^2(\Gamma_h^n)} + \| \Wdhn{n} \|_{L^2(\Gamma_h^n)}\right) \|\gradgh E_{u,h}^\delta \|_{L^2(\Gamma_h(t))}.
\end{align*}
Hence we integrate over $[0,T]$ and use \eqref{fdiscregstability}, \eqref{fdisc derivative bound1}, \eqref{fdiscdpot2}, \eqref{defectderivative2}, and Young's inequality to see that
\begin{gather}
	\int_0^T |J_4| \leq C \tau^\frac{3}{2} + C \int_0^T\inormh{\widetilde{E_{u,h}^{\tau, \delta}}}^2,\label{defectpf16}\\
	\int_0^T |J_5| \leq C\tau^2 + \frac{\varepsilon}{4} \int_0^T \|\gradgh E^\delta_{u,h}\|_{L^2(\Gamma_h(t))}^2. \label{defectpf17}
\end{gather}

One then concludes by integrating in time, using \eqref{defectpf9}--\eqref{defectpf17} in \eqref{defectpf8} and  Gr\"onwall's inequality.
We omit further details.
\end{proof}

We are now in a position to prove our final error bound.

\begin{theorem}
	\label{fullydiscerrortheorem}
	Let $(u,w)$ be the unique solution of \eqref{cheqn1},\eqref{cheqn2} and $\underline{U_h}$ be defined as above.
	Then for $h, \tau$ sufficiently small we have
	\begin{align}
		\label{fullydiscerror}
		\varepsilon \int_0^T \| \gradgh (u^{-\ell} - \underline{U_h}) \|_{L^2(\Gamma_h(t))}^2 \leq C\left(\tau + h^{\frac{4}{3}} \log\left(\frac{1}{h}\right) + \frac{\tau^2}{h^\frac{8}{3}}\right),
	\end{align}
	for $C$ a constant independent of $h, \tau.$
\end{theorem}
\begin{proof}
	This follows by using the splitting
	\[u^{-\ell} - \underline{U_h} = [u^{-\ell} - (u^{\delta})^{-\ell}] + [(u^{\delta})^{-\ell} - U_h^\delta]+ [\Udh - \underline{\Udh}]- [\underline{\Udh}-\underline{U_h}],\]
	and the error bounds \eqref{deltaerror1}, \eqref{deltaerror3}, \eqref{deltaerror4}, \eqref{deltaerror6}, \eqref{deltaerror7} where we have chosen $\delta = C(p)h^p$, for a small constant $C(p)$, and for maximal order we take $p = \frac{4}{3}$.
	Note that taking $\tau \leq Ch^2$, in accordance with our CFL condition, we obtain the same order error bound as in in the semi-discrete case.
\end{proof}
\begin{remark}
	\begin{enumerate}
		\item As discussed in the semi-discrete error, this will hold for any $\tilde{U}_{h,0} \in \mathcal{I}_{h,0}$ such that
		\[\int_{\Gamma_h(0)} \tilde{U}_{h,0} = \int_{\Gamma_0} u_0, \quad \mathrm{ and } \quad \|\Pi_h u_0 - \tilde{U}_{h,0}\|_{L^2(\Gamma_h(0))} \leq C h^\frac{2}{3}\log\left(\frac{1}{h}\right)^\frac{1}{2},\]
		for some constant $C$ independent of $h$.
		Moreover, from \eqref{defectderivative1} we also find that
		\[\varepsilon \int_0^T \| \gradgh (u^{-\ell} - U_h^\tau) \|_{L^2(\Gamma_h(t))}^2 \leq C\left(\tau + h^{\frac{4}{3}} \log\left(\frac{1}{h}\right)+ \frac{\tau^2}{h^\frac{8}{3}}\right).\]
		\item It may be possible to adapt arguments from \cite{barrett1999improved} to improve this bound, but we do not consider this here.
	\end{enumerate}	
\end{remark}

\section{Numerical experiments}
\label{section: logch numerics}
In this section we present some numerical results for an implementation of the fully discrete scheme \eqref{fullydisceqn1}, \eqref{fullydisceqn2}.
We may express the fully discrete scheme in a block matrix form as
\[\begin{pmatrix}
	\bar{M}^n & \tau A^n\\
	-\varepsilon A^n + \frac{1}{\varepsilon}\bar{M}^n  & \bar{M}^n
\end{pmatrix}
\begin{pmatrix}
	\alpha^n\\
	\beta^n
\end{pmatrix}
- \frac{\theta}{2 \varepsilon} \begin{pmatrix}
	0\\
	\bar{M}^n f(\alpha^n)
\end{pmatrix} =
\begin{pmatrix}
	\bar{M}^{n-1} \alpha^{n-1}\\
	0
\end{pmatrix},
\]
as we have throughout the paper.
Here $\alpha^n, \beta^n$ are such that
\[U_h^n = \sum_{j=1}^{N_h} \alpha_j^n \phi_j^n, \quad W_h^n = \sum_{j=1}^{N_h} \beta_j^n \phi_j^n.\]

One cannot solve this scheme immediately by use of standard Newton methods, as one has the constraint that $\alpha^n \in (-1,1)^{N_h}$.
One could remedy this by considering a Newton scheme with variable stepsizes, see for instance \cite{deuflhard2005newton}, or in our case by using a variant of the nonsmooth Newton solver in \cite{graser2014numerical}.
The corresponding linear system to solve is solved by an exact solver based on LU decomposition with pivoting.
We omit further details on the implementation.

\subsection{Dynamics on an expanding torus}
Here we consider the torus given by the level set equation
\begin{align*}
	\left( \sqrt{x^2 + y^2} - 0.75 - t \right)^2 + z^2 - 0.25^2 = 0,
\end{align*}
which one can verify has $\gradg \cdot V > 0$.
We choose the initial data to be $u_0(x,y,z) = 0.9 x\cos\left(\frac{\pi y}{2}\right)$, which we observe has vanishing mean value.
Moreover we take $\varepsilon = 0.1$, $\theta = 0.4$ and $T=0.6$.
The mesh here consists of 6016 elements, and $\tau = 5\cdot 10^{-5}$.

\begin{figure}[H]
	\centering
	\includegraphics[width=\textwidth]{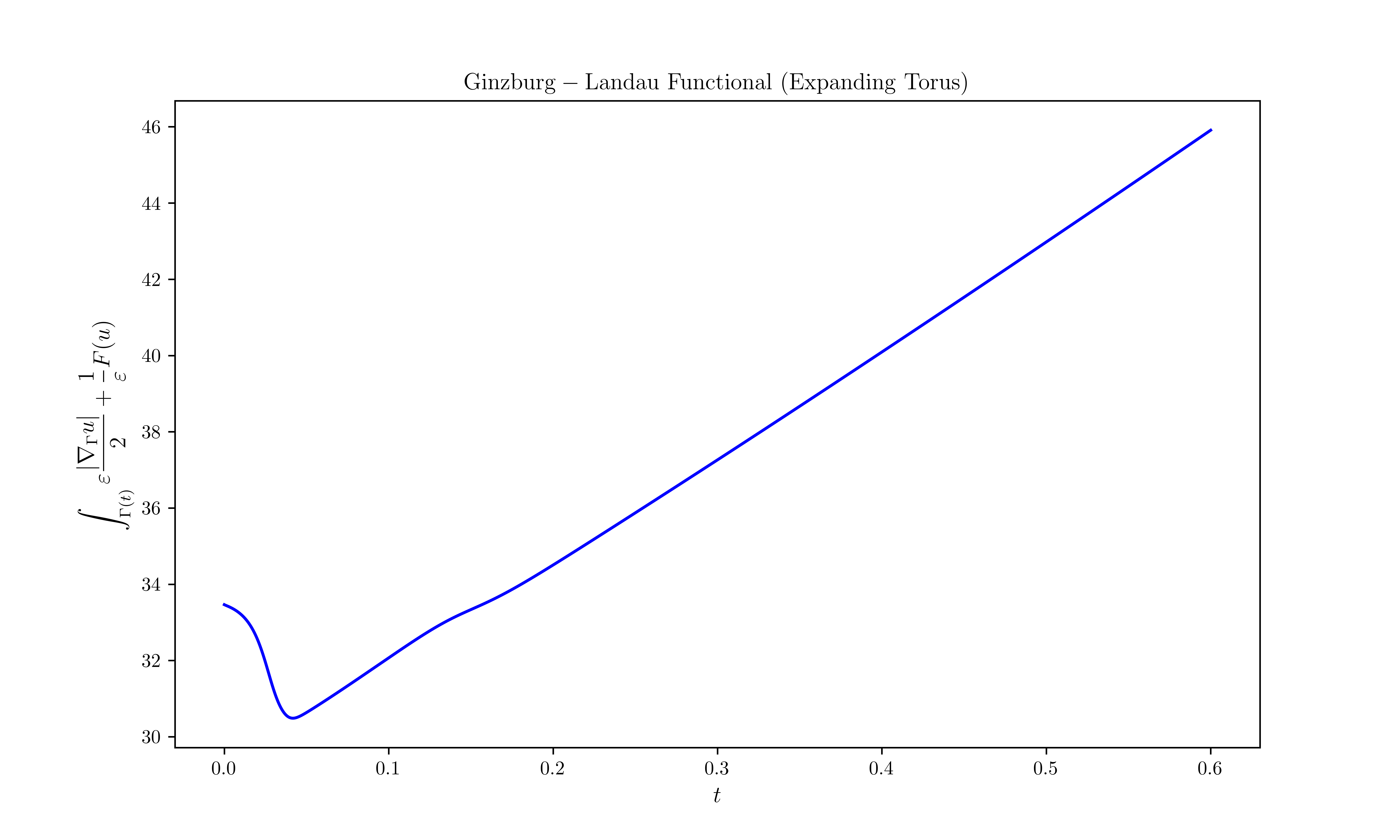}
	\caption{Plot of the Ginzburg-Landau functional on an expanding torus over $[0,0.6]$.}
	\label{fig:ExpTorusGL}
\end{figure}
We note that the Ginzburg-Landau functional as plotted in Figure \ref{fig:ExpTorusGL} is non-monotonic.
Heuristically can explain this as a competition of two contrary behaviours.
Firstly for small $\varepsilon$ the Ginzburg-Landau functional \eqref{glfunctional} is approximately the length of some curve $\gamma(t) \subset \Gamma(t)$ which evolves via some evolving surface analogue of the Mullins-Sekerka flow obtained in the sharp interface limit $\varepsilon \rightarrow 0$ --- we refer the reader to a similar discussion in \cite[Section 6.3]{elliott2024fully}.
This should, and for a stationary domain does, shrink the curve till it converges to some curve of a fixed length.
On the other hand, the surface $\Gamma(t)$ is expanding, and so sufficiently fast expansion of $\Gamma(t)$ would lead to the length of $\gamma(t)$ increasing.
Thus the two regimes seen in Figure \ref{fig:ExpTorusGL} are justified by these two behaviours respectively.

\begin{figure}[H]
	\begin{subfigure}{.45\linewidth}
		\includegraphics[width=1.2\textwidth]{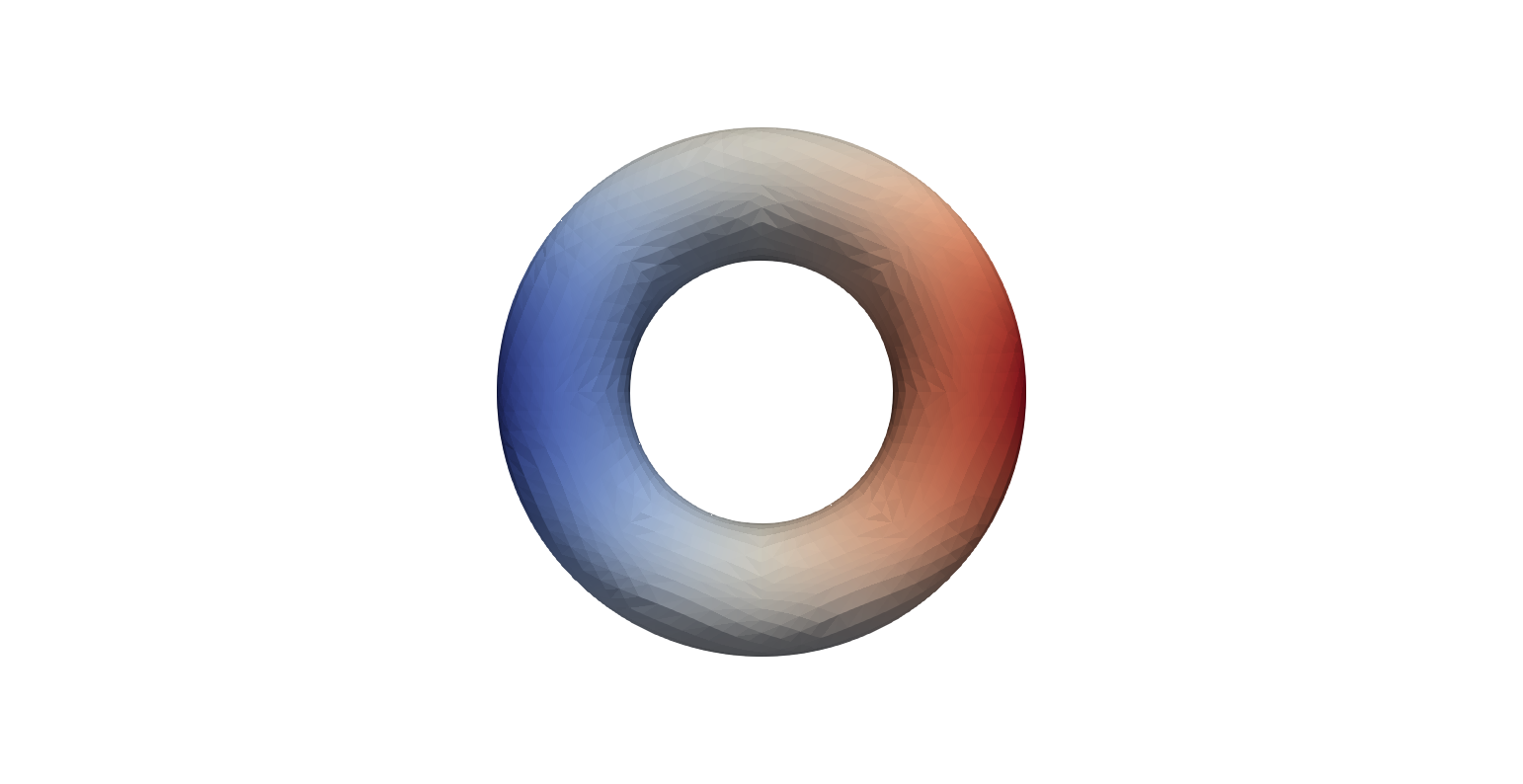}
		\caption{$t=0.$}
	\end{subfigure}%
	\begin{subfigure}{.45\linewidth}
		\includegraphics[width=1.2\textwidth]{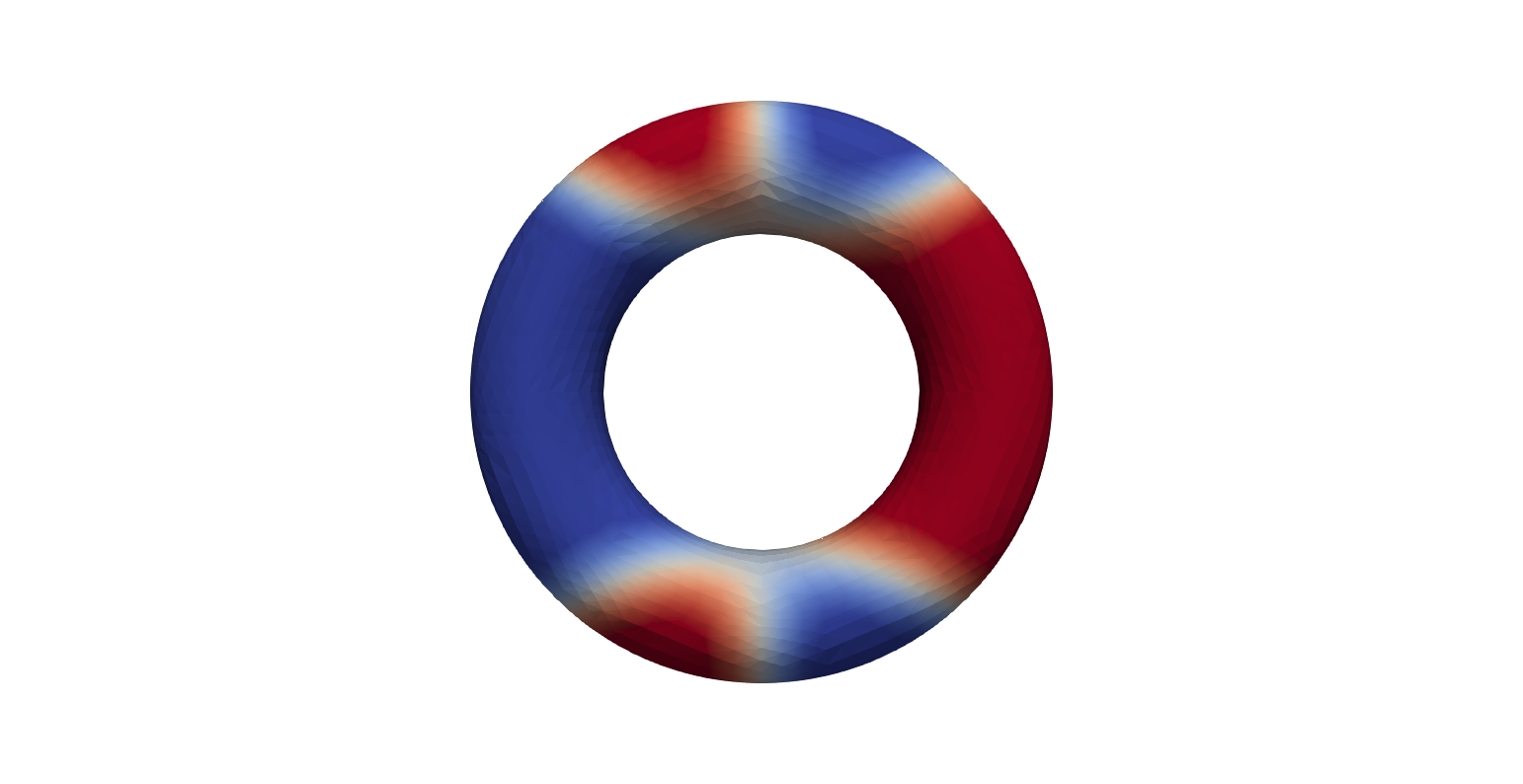}
		\caption{$t=0.1.$}
	\end{subfigure}%
	\newline
	\begin{subfigure}{.45\linewidth}
		\includegraphics[width=1.2\textwidth]{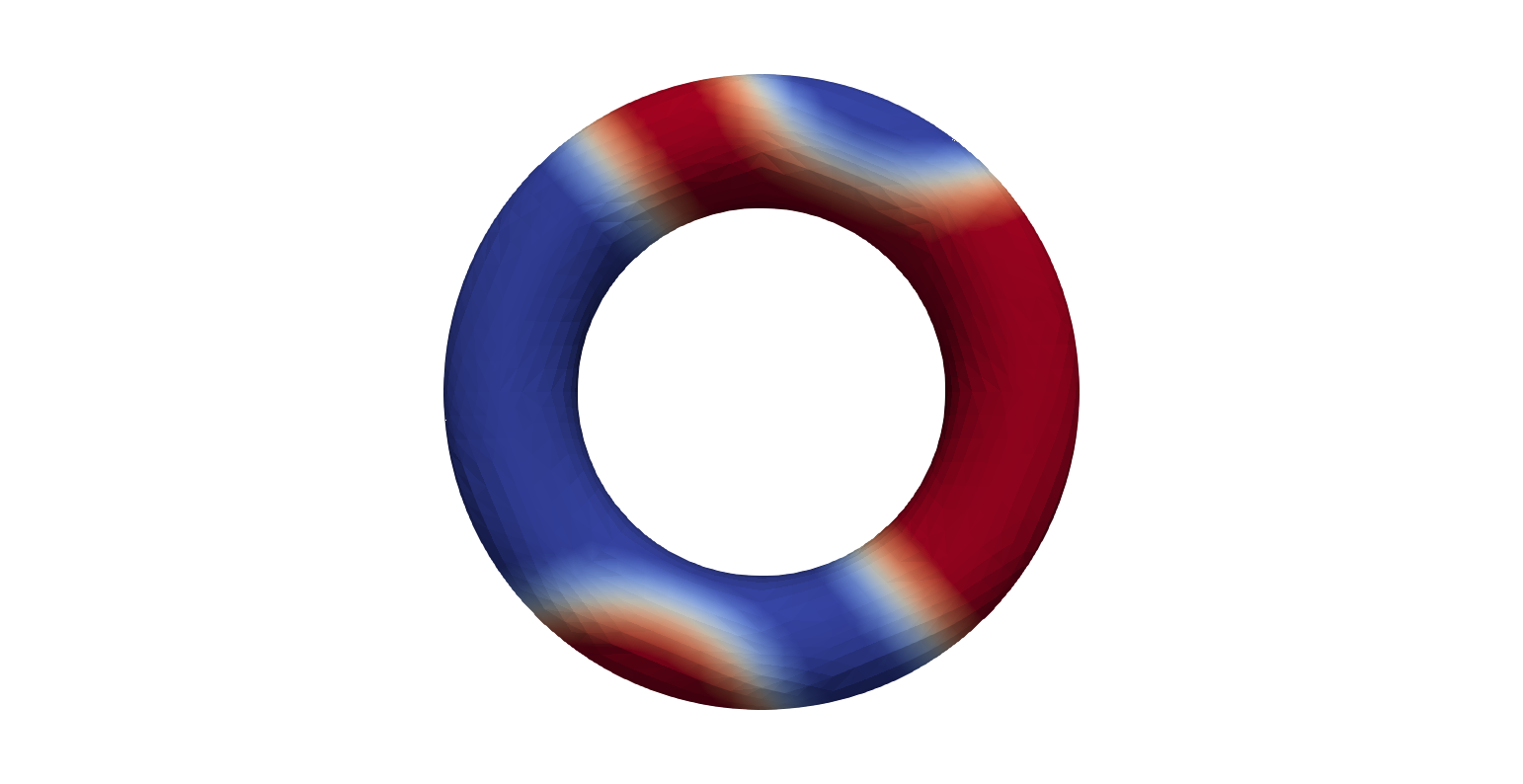}
		\caption{$t=0.2.$}
	\end{subfigure}%
	\begin{subfigure}{.45\linewidth}
		\includegraphics[width=1.2\textwidth]{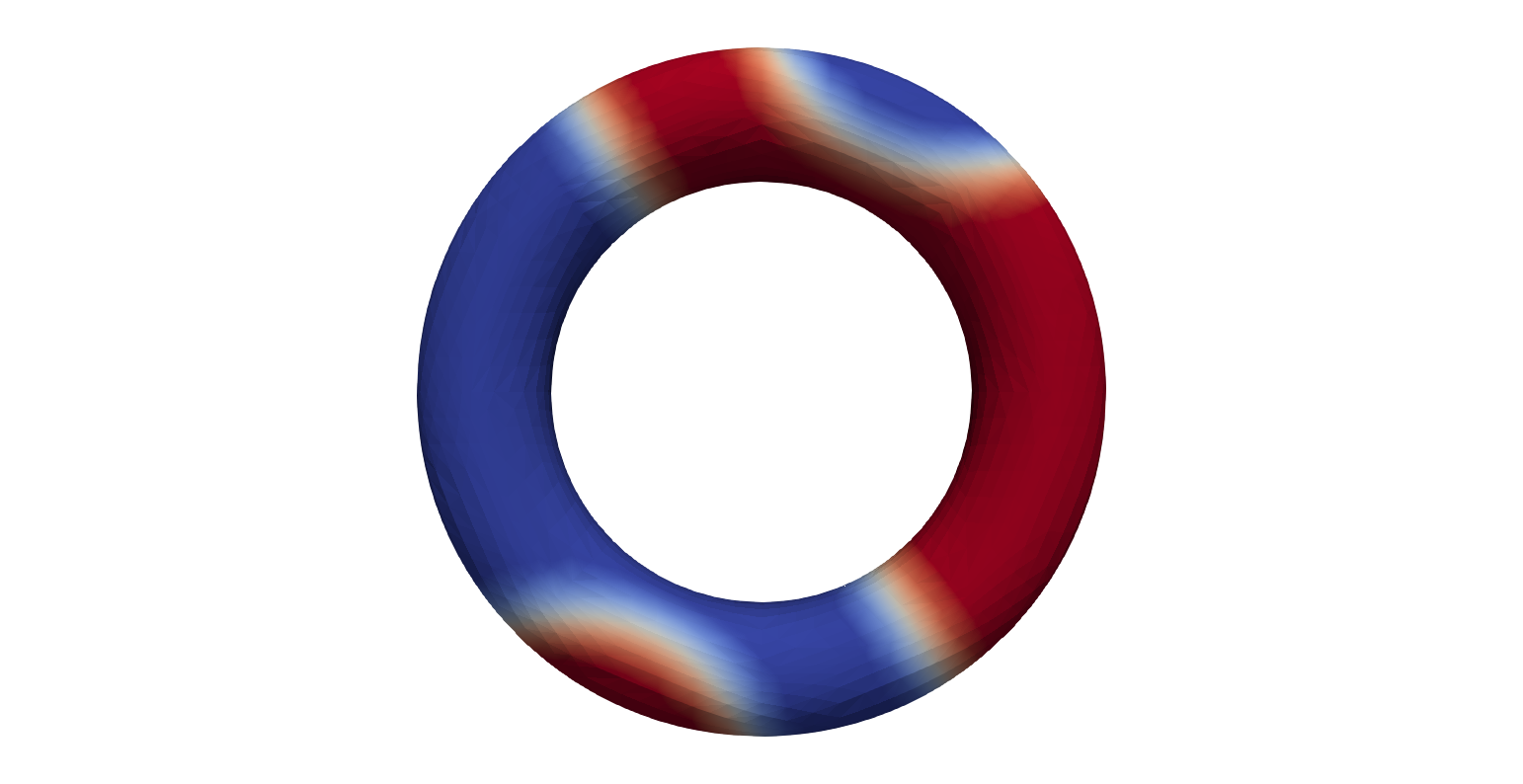}
		\caption{$t=0.3.$}
	\end{subfigure}%
	\caption{Evolution of $u$ on an expanding torus.
		Regions of blue correspond to a negative quantity, and red a positive quantity.}
	\label{fig:ExpTorusEvolution}
\end{figure}

\subsection{Dynamics on a shrinking torus}
Here we consider a shrinking torus, given by the level set equation
\[\left( \sqrt{x^2 + y^2} - 0.75 \right)^2 + z^2 - (0.25-0.25t)^2 = 0,\]
to investigate the necessity of our assumption that $\gradg \cdot V \geq 0$.
It is a straightforward calculation to verify that this level set is such that $\gradg \cdot V < 0$.
We choose the same initial data and values for $\varepsilon, \theta, \tau$ as in the previous experiment.
The theory from \cite{caetano2021cahn,caetano2023regularization} shows that a solution exists for $t \in [0,1)$, and this experiment investigates whether any numerical issues arise when the true solution is defined.
We find that there seems to be no numerical issues on this small time interval, which indicates some hope for extending our numerical analysis to evolving surfaces without the condition that $\gradg \cdot V > 0$. 
We observe in Figure \ref{fig:ShrTorusGL} that the Ginzburg-Landau functional is monotonically decreasing here, and this follows the same heuristic argument as we saw for the expanding case.

\begin{figure}[ht]
	\centering
	\includegraphics[width=\textwidth]{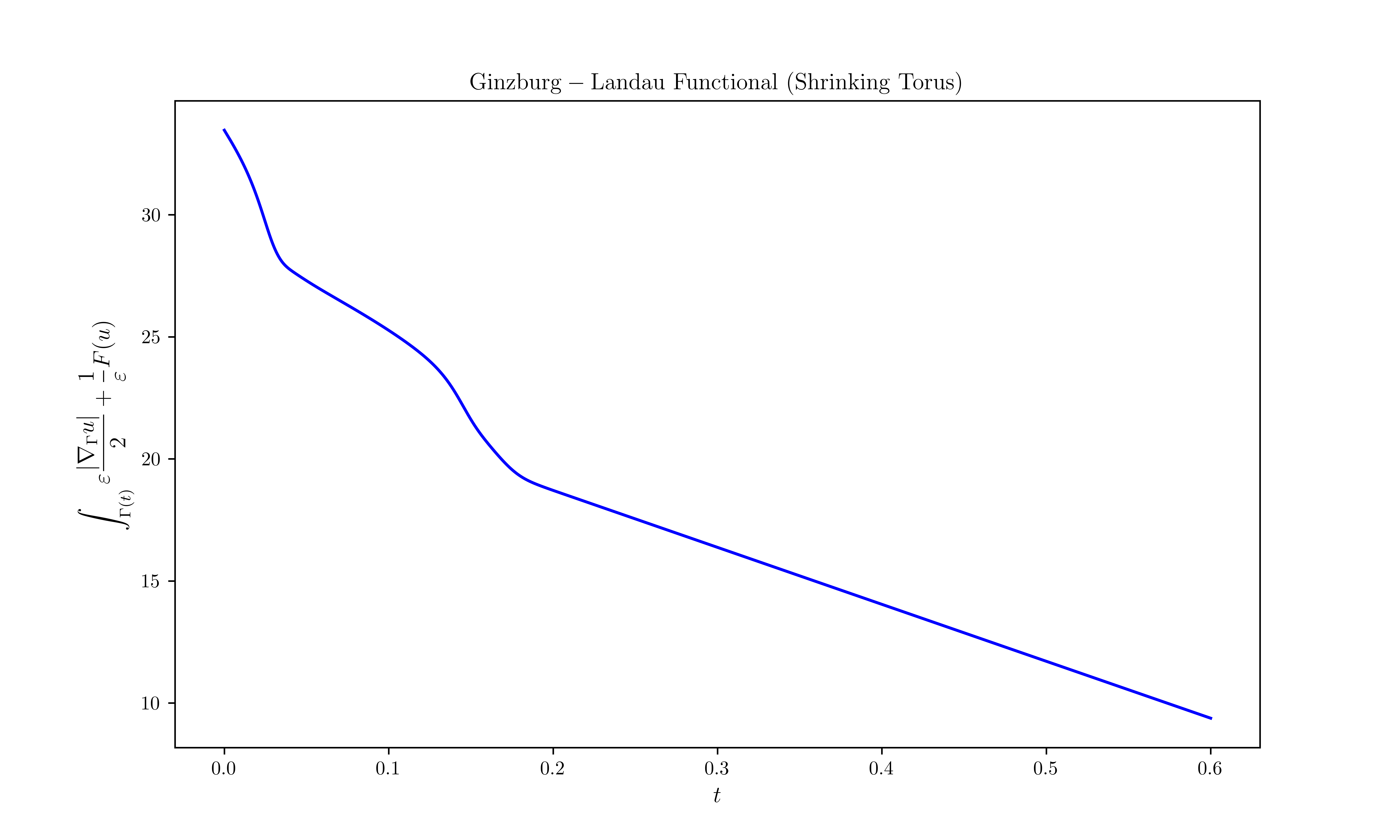}
	\caption{Plot of the Ginzburg-Landau functional on a shrinking torus over $[0,0.6]$.}
	\label{fig:ShrTorusGL}
\end{figure}

We also include some examples of the evolution of $u$ on this shrinking domain, Figure \ref{fig:ShrTorusEvolution}, which shows how the behaviour of the solution is vastly different to an expanding domain --- see Figure \ref{fig:ExpTorusEvolution}.

\begin{figure}[H]
	\begin{subfigure}{.45\linewidth}
		\includegraphics[width=1.2\textwidth]{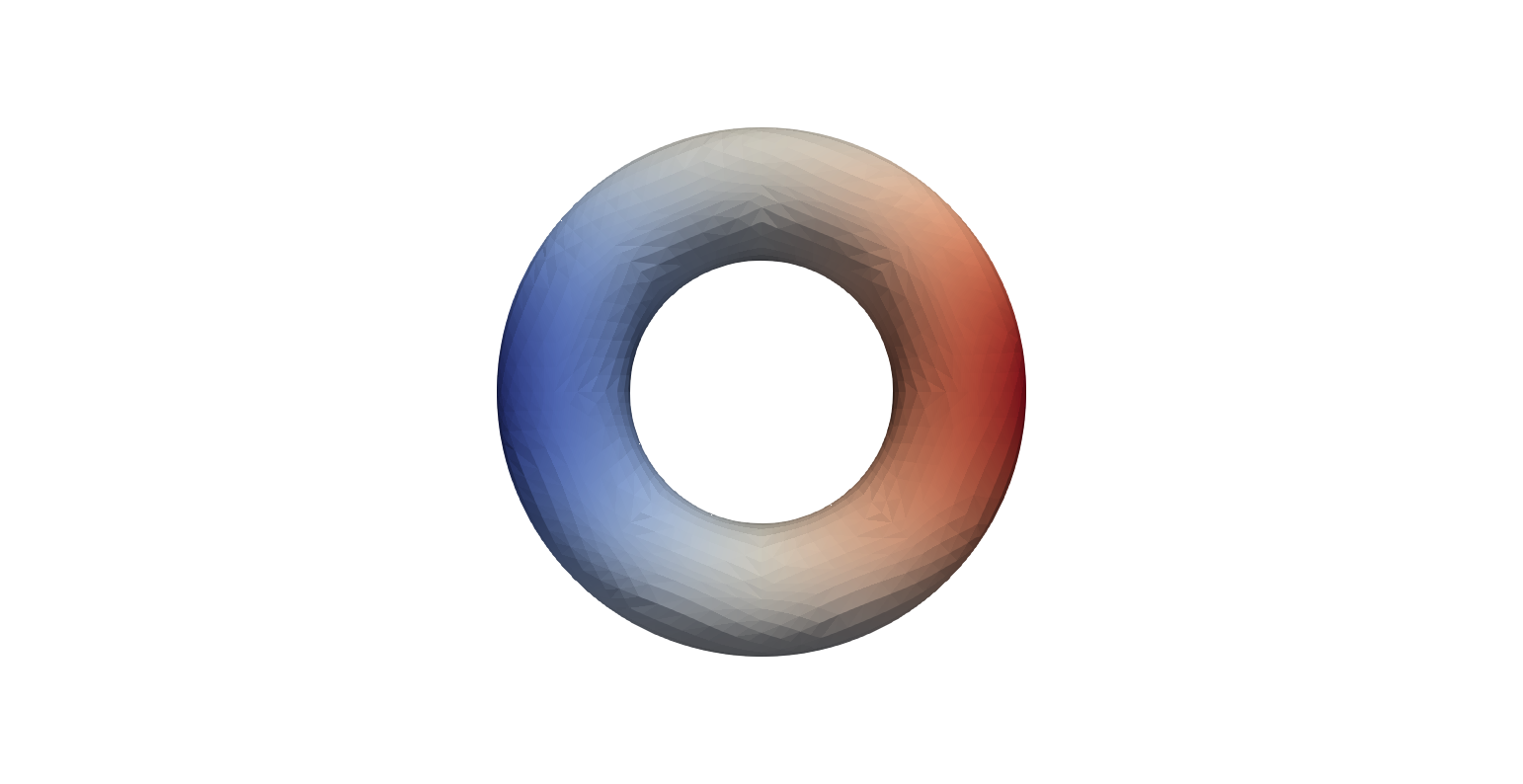}
		\caption{$t=0.$}
	\end{subfigure}%
	\begin{subfigure}{.45\linewidth}
		\includegraphics[width=1.2\textwidth]{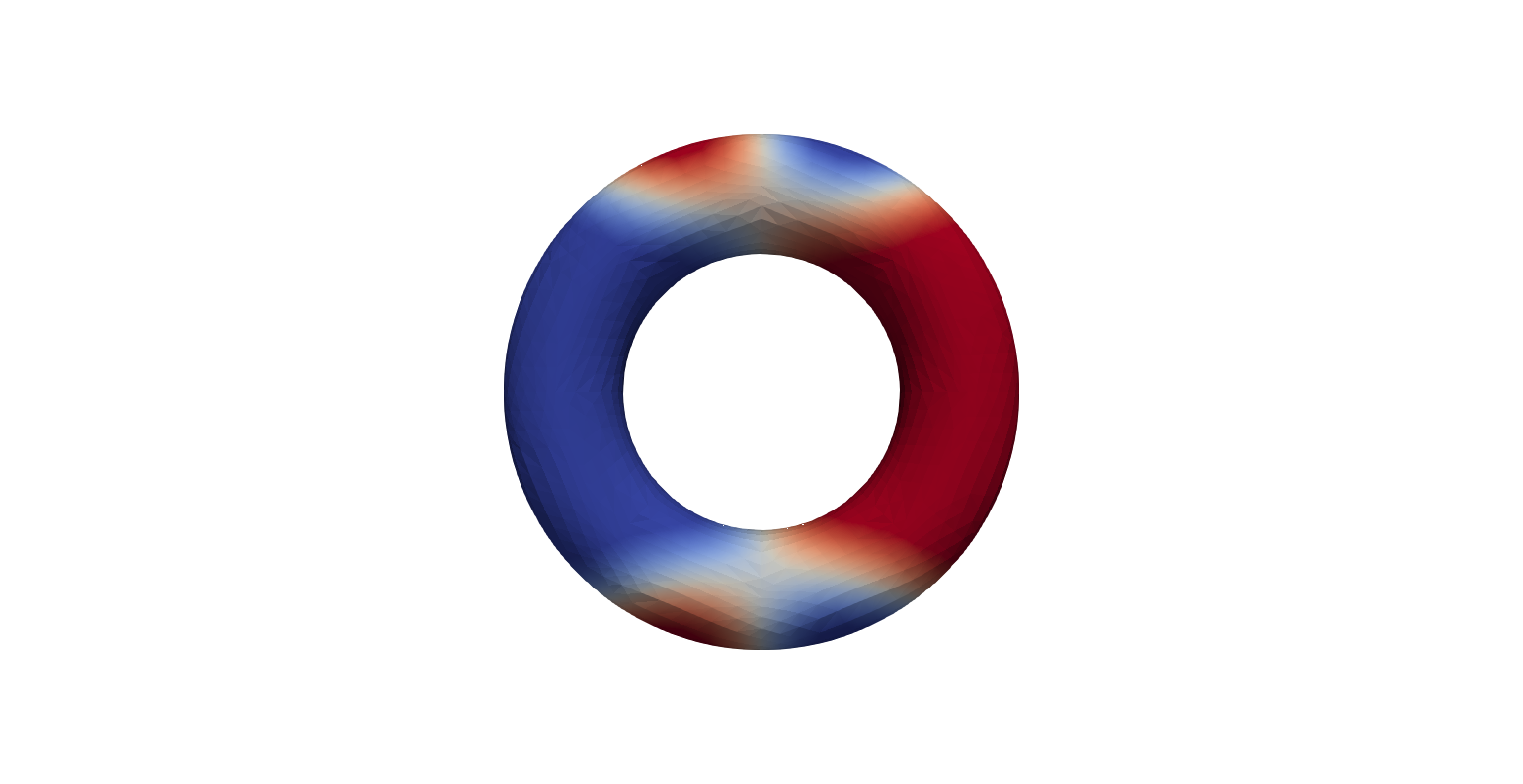}
		\caption{$t=0.1.$}
	\end{subfigure}%
	\newline
	\begin{subfigure}{.45\linewidth}
		\includegraphics[width=1.2\textwidth]{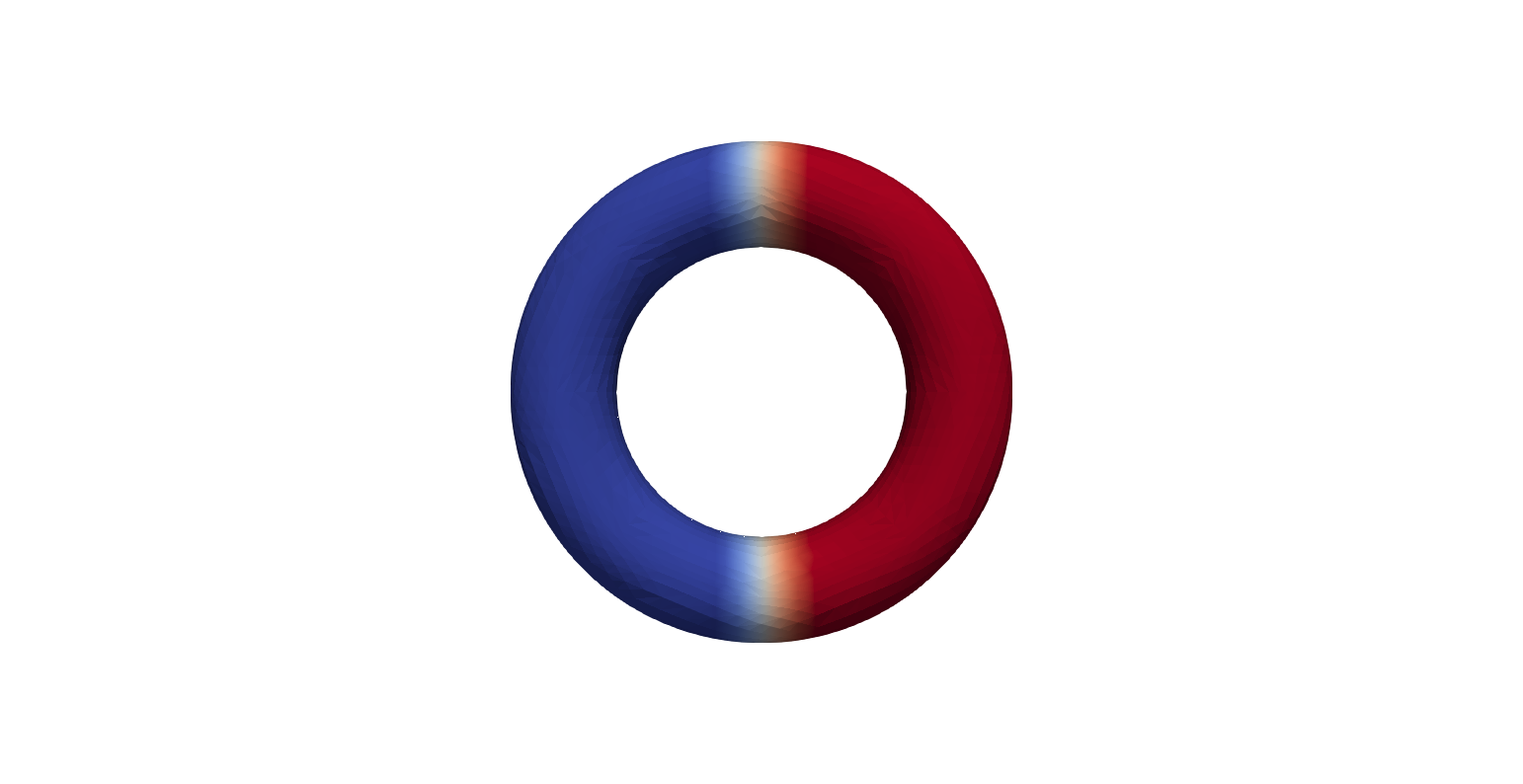}
		\caption{$t=0.2.$}
	\end{subfigure}%
	\begin{subfigure}{.45\linewidth}
		\includegraphics[width=1.2\textwidth]{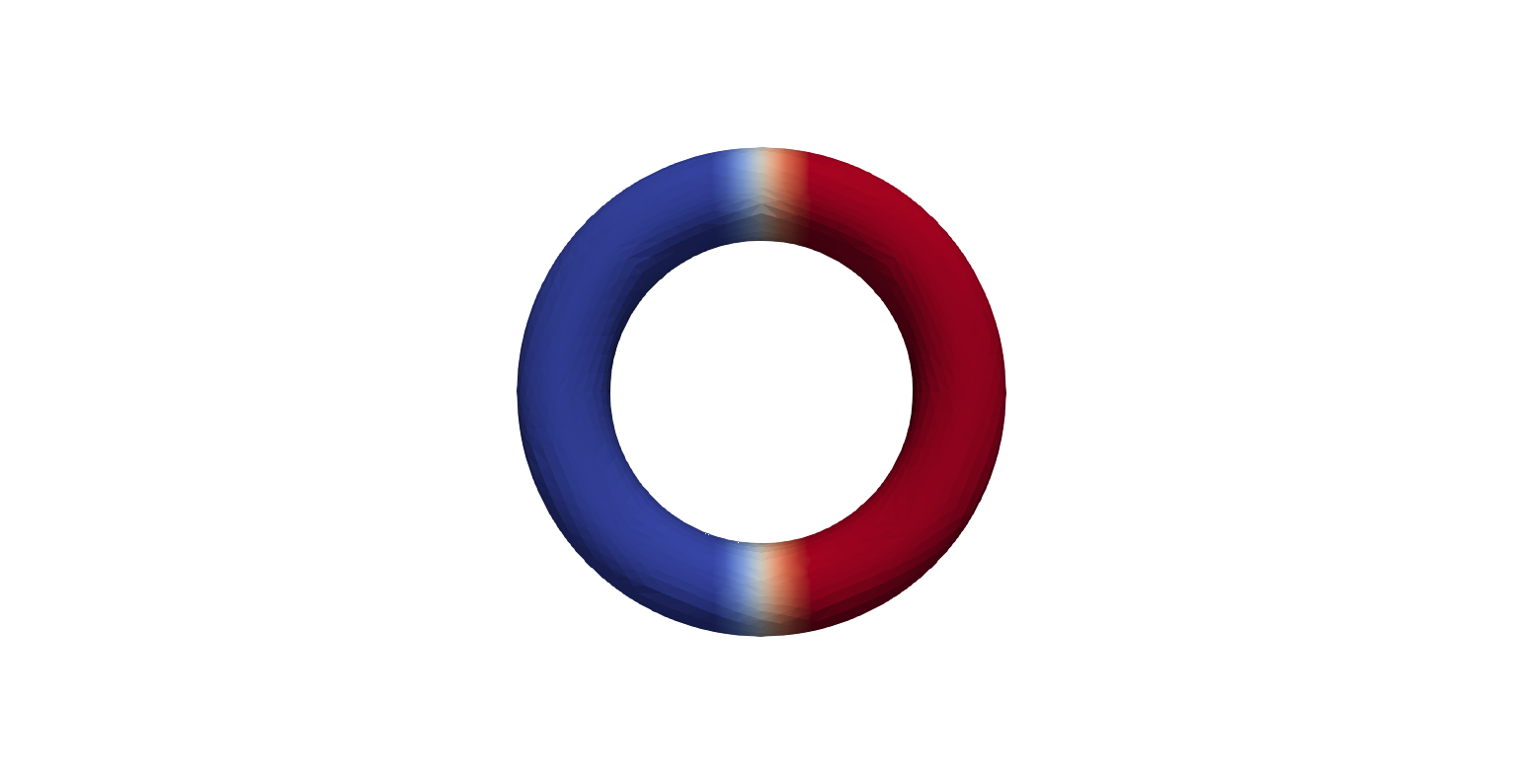}
		\caption{$t=0.3.$}
	\end{subfigure}%
	\caption{Evolution of $u$ on a shrinking torus.
		Regions of blue correspond to a negative quantity, and red a positive quantity.}
	\label{fig:ShrTorusEvolution}
\end{figure}

\subsection{Experimental order of convergence on an expanding sphere}
Here we compute an experimental order of convergence (EOC) for the $H^1$ error of the scheme \eqref{fullydisceqn1}, \eqref{fullydisceqn2} on an expanding sphere\footnote{This example is the unit sphere evolving by inverse mean curvature flow ($V_N = \frac{1}{H}$) which fits our assumptions on the expansion of the surface as discussed previously.}, given by the level set equation
\[ x^2 + y^2 + z^2 - e^t = 0. \]
The EOC is computed by solving the equation on a coarse mesh, and prolonging this coarse solution onto a finer mesh.
This process is analogous to lifting.
We obtain an error by approximating the true solution with a fine solution ($\tau = 10^{-5}, h \approx 4.023559 \cdot 10^{-2}$), as the exact solution is not known.
We choose the timestep sizes to be $\mathcal{O}(h^2)$, in accordance with our CFL condition, and this also avoids a bottleneck in the error.
We observe that the EOC is larger than predicted by Theorem \ref{fullydiscerrortheorem}.
This is not surprising as our proof relies on using the regularised potential, which introduces a bottleneck that practical schemes will not be limited by.

Here the parameters $\varepsilon, \theta, T$ are chosen as $\varepsilon = 0.1, \theta = 0.5, T = 0.1$, and the initial data is given by $u_0(x,y,z) = 0.5 x$.

\begin{table}[h]
	\centering
	\begin{tabular}{ |c|c|c| } 
		\hline
		$h$ & $\|\gradg(u - (U_h^{N_T})^\ell)\|_{L^2(\Gamma(T))}$ & EOC\\
		\hline
		$6.437694 \cdot 10^{-1}$ & $4.052072$ & - \\ 
		\hline
		$3.218847\cdot 10^{-1}$ & $2.016871$ & $1.006541$ \\ 
		\hline
		$1.609424\cdot 10^{-1}$ & $9.449989 \cdot 10^{-1}$ & $1.0937343$ \\ 
		\hline
		$8.047118 \cdot 10^{-2}$ & $3.987348 \cdot 10^{-1}$ & $1.244883$ \\ 
		\hline
	\end{tabular}
	\caption{Table of EOC for the expanding sphere.\label{table:SphereTable}}
\end{table}

We note that in this numerical experiment the mesh started acute (for all levels of refinement) but immediately lost this property.
Despite this the solution remained numerically stable and appears to be have an optimal order of convergence in the $H^1$ seminorm.

\subsection*{\textbf{Acknowledgements}}
The authors would like to thank the anonymous referees for their helpful comments, and spotting some mistakes in the original draft of this paper.
Thomas Sales is supported by the Warwick Mathematics Institute Centre for Doctoral Training, and gratefully acknowledges funding from the University of Warwick and the UK Engineering and Physical Sciences Research Council (Grant number:EP/TS1794X/1).
For the purpose of open access, the author has applied a Creative Commons Attribution (CC BY) licence to any Author Accepted Manuscript version arising from this submission.

\appendix
\section{Inverse Laplacians}
\label{invlaps}

In this appendix we discuss several notions of inverse Laplacians and related inequalities.

\begin{definition}
	Let $z \in H^{-1}(\Gamma(t))$ with $m_*(z,1)= 0$.
	We define the inverse Laplacian of $z$, $\mathcal{G}z$, to be the unique solution of
	\[a(\mathcal{G}z, \eta) = m_*(z, \eta), \qquad \int_{\Gamma(t)} \mathcal{G}z = 0,\]
	for all $\eta \in H^1(\Gamma(t))$.
\end{definition}
We also recall the following result from \cite{caetano2021cahn,elliott2015evolving}.
\begin{lemma}
	If $z \in L^2_{H^1}\cap H^1_{H^{-1}}$ and $m_*(z,1)= 0$, then $\mathcal{G}z$ is well defined and $\mathcal{G} z \in H^{1}_{H^1}$.
\end{lemma}
We also require three different notions of a discrete inverse Laplacian on $\Gamma_h(t)$.
\begin{definition}
	Let $z_h \in L^2(\Gamma_h(t))$ such that $\int_{\Gamma_h(t)} z_h = 0$.
	\begin{enumerate}
		\item We define the discrete inverse Laplacian, $\invh z_h \in H^1(\Gamma_h(t))$, to be the unique solution of
		$$ a_h(\invh z_h, \eta_h) = m_h(z_h, \eta_h), \qquad \int_{\Gamma_h(t)} \invh z_h = 0,$$
		for all $\eta_h \in H^1(\Gamma_h(t))$.
		\item Similarly for $z_h \in S_h(t)$, we define the inverse Laplacian on shape functions to be the unique solution, $\invsh z_h \in S_h(t)$, of
		$$ a_h(\invsh z_h, \phi_h) = m_h(z_h, \phi_h), \qquad \int_{\Gamma_h} \invsh z_h = 0,$$
		for all $\phi_h \in S_h(t)$.
		\item Lastly for $z_h \in S_h(t)$, we define a mass-lumped inverse Laplacian as the unique solution, $\intinvsh z_h \in S_h(t)$, of
		\begin{equation*}
			\label{intinvlap}
			a_h \left(\intinvsh z_h, \phi_h \right) = \intm(z_h, \phi_h), \qquad \int_{\Gamma_h(t)}\intinvsh z_h, = 0,
		\end{equation*}
		for all $\phi_h \in S_h(t)$, provided that $z_h \in C^0(\Gamma_h(t))$.	\end{enumerate}
\end{definition}
Each of these operators gives rise to an corresponding norm, given by
\begin{gather*}
	\| z \|_{-1}^2 := a(\mathcal{G}z,\mathcal{G}z), \qquad \|z_h\|_{-1,h}^2 := a_h(\invh z_h, \invh z_h),\\
	\normsh{z_h}^2 := a_h(\invsh z_h, \invsh z_h), \qquad \inormh{z_h}^2 := a_h \left( \intinvsh z_h, \intinvsh z_h \right).
\end{gather*}

It is clear that these are well defined norms due to the Poincar\'e inequality and the condition on the mean value.
We will use the inverse Laplacian to establish bounds in $H^{-1}(\Gamma_h(t))$, since
\begin{equation}\label{invlapduality}
	\| z_h \|_{H^{-1}(\Gamma_h(t))}  = \sup_{\eta_h \in H^1(\Gamma_h(t)) \backslash \{0\} } \frac{|m_h(z_h, \eta_h)|}{\|\eta_h\|_{H^1{(\Gamma_h(t))}}} \leq \|z_h\|_{-1,h},
\end{equation}
for $z_h \in L^2(\Gamma_h(t))$ with $\int_{\Gamma_h(t)} z_h = 0$.

We relate the second and third inverse Laplacians through the inequality,
\begin{align}
	\label{invlapineq1}
	\left\|(\invh - \invsh)z_h \right\|_{L^2(\Gamma_h(t))} \leq Ch^2 \|z_h\|_{L^2(\Gamma_h(t))}.
\end{align}
This can be seen as an error bound for linear finite elements solving Laplace's equation posed on $\Gamma_h(t)$ (see \cite{dziuk2013finite}).
Furthermore we have the following inequalities.
\begin{lemma}
	Let $z_h \in S_h(t)$ with $\int_{\Gamma_h(t)} z_h = 0$, then we have,
	\begin{align}
		\label{invlapineq2}
		C_1 h^2 \| \gradgh z_h \|_{L^2(\Gamma_h(t))} \leq C_2 h \| z_h \|_{L^2(\Gamma_h(t))} \leq \normsh{z_h} \leq \| z_h \|_{-1,h} \leq C_3 \normsh{z_h}.\\
		C_4 h^2 \| \gradgh z_h \|_{L^2(\Gamma_h(t))} \leq C_5 h \normh{t}{z_h} \leq C_6 \inormh{z_h} \leq C_7 \normsh{z_h} \leq C_8 \inormh{z_h}. \label{invlapineq3}
	\end{align}
\end{lemma}
\begin{proof}
	The first inequality in the chain follows by Poincar\'e's inequality and an inverse inequality (see \cite{brenner2008mathematical}).
	The second inequality comes from
	\begin{multline*}
	    h^2 \|z_h\|_{L^2(\Gamma_h(t))}^2 = h^2 a_h\left( \invsh z_h, z_h \right) \leq \frac{1}{\lambda} \normsh{z_h}^2 + \frac{\lambda h^4}{4}\| \gradgh z_h \|_{L^2(\Gamma_h(t))}^2\\
     \leq \frac{1}{\lambda}\normsh{z_h}^2 + \frac{C \lambda h^2}{4}\| z_h \|_{L^2(\Gamma_h(t))}^2,
	\end{multline*}
	where we have used Young's inequality and the first inequality.
	The inequality follows by taking $\lambda$ sufficiently small so that $\frac{C \lambda}{4} < 1$.
	The third inequality comes from
	$$ a_h \left( \invsh z_h, \invsh z_h \right) = m_h\left( z_h, \invsh z_h \right) = a_h \left( \invh z_h, \invsh z_h \right).$$
	The final inequality then follows from
	$$ \|z_h\|_{-1,h}^2 - \normsh{z_h}^2 = m_h\left( z_h,  (\invh - \invsh) z_h \right) \leq Ch^2 \| z_h \|_{L^2(\Gamma_h(t))}^2 \leq C \normsh{z_h}^2, $$
	where we have used \eqref{invlapineq1} and the second inequality.
	The proof of \eqref{invlapineq3} is more or less identical, and hence omitted.
\end{proof}

We note that \eqref{invlapineq1}, \eqref{invlapineq2}, \eqref{invlapineq3} show that for $z_h \in S_h(t)$ with $\int_{\Gamma_h(t)} z_h = 0$ we may control $\|z_h \|_{H^{-1}(\Gamma_h(t))}$ by $\normsh{z_h}$ or $\inormh{z_h}$.

Lastly we note that each of these operators changes in time as the geometry of $\Gamma(t)$ varies.
As such, we will need the following commutator bound which compares two notions of an inverse Laplacian on $\Gamma_h(t)$.
This result will be required in the proof of Lemma \ref{defectderivative lemma}.
\begin{lemma}
	\label{invlap commutator}
	Let $z_h^n \in S_h^n$ and $t \in [t_{n-1}, t_n]$.
	Then for
        \begin{gather*}
            z_{h,0}^n := z_h^n - \mval{z_h^n}{\Gamma_h^n} \in S_h^n,\\
            z_{h,0}^t := \Phi_t^h \Phi_{-t_{n}}^h z_h^{n}  - \mval{\Phi_t^h \Phi_{-t_{n}}^h z_h^{n}}{\Gamma_h(t)} \in S_h(t),
        \end{gather*}
        one has that
	\begin{align*}
		\left\| \gradgh \left( \intinvsh z_{h,0}^t - \Phi_t^h \Phi_{-t_{n-1}}^h \intinvsh z_{h,0}^{n} \right) \right\|_{L^2(\Gamma_h(t))}\leq C \tau \|z_h^{n}\|_{L^2(\Gamma_h^{n})}, 
	\end{align*}
	for a constant $C$ independent of $\tau,h$.
\end{lemma}
\begin{proof}
	We begin by noting that
	\begin{align*}
		\left\| \gradgh \left( \intinvsh z_{h,0}^t - \Phi_t^h \Phi_{-t_{n}}^h \intinvsh z_{h,0}^{n} \right) \right\|_{L^2(\Gamma_h(t))}^2&= \intm \left( z_{h,0}^t, \intinvsh z_{h,0}^t - \Phi_t^h \Phi_{-t_{n}}^h \intinvsh z_{h,0}^{n} \right)\\
		&-\intm\left( z_{h,0}^{n}, \overline{\intinvsh z_{h,0}^t - \Phi_t^h \Phi_{-t_{n}}^h \intinvsh z_{h,0}^{n}}(t_{n}) \right)\\
    &+ a_h\left( \intinvsh z_{h,0}^{n},  \overline{\intinvsh z_{h,0}^t - \Phi_t^h \Phi_{-t_{n}}^h \intinvsh z_{h,0}^{n}}(t_{n}) \right)\\
		&- a_h \left( \Phi_t^h \Phi_{t_{n}}^h \intinvsh z_{h,0}^{n}, \intinvsh z_{h,0}^t - \Phi_t^h \Phi_{-t_{n}}^h \intinvsh z_{h,0}^{n}  \right),
	\end{align*}
    where we have used the definition of $\intinvsh$ and introduced extra terms
    \[ \intm\left( z_{h,0}^{n}, \overline{\intinvsh z_{h,0}^t - \Phi_t^h \Phi_{-t_{n}}^h \intinvsh z_{h,0}^{n}}(t_{n}) \right)\\
    = a_h\left( \intinvsh z_{h,0}^{n},  \overline{\intinvsh z_{h,0}^t - \Phi_t^h \Phi_{-t_{n}}^h \intinvsh z_{h,0}^{n}}(t_{n}) \right). \]
	Next it is a straightforward analogue of \eqref{timeperturb1} to see that, since $\Phi_t^h \Phi_{-t_{n}}^h z_{h}^{n} = \underline{z_h^n}(t)$,
    \begin{align*}
    \left| \mval{z_h^n}{\Gamma_h^n} - \mval{\Phi_t^h \Phi_{-t_{n}}^h z_{h}^{n}}{\Gamma_h(t)} \right| &\leq \left| \frac{1}{|\Gamma_h^n|} - \frac{1}{|\Gamma_h(t)|} \right| \left| \int_{\Gamma_h^n} z_h^n \right| + \frac{1}{|\Gamma_h(t)|}\left| \int_{\Gamma_h^n} z_h^n - \int_{\Gamma_h(t)}\Phi_t^h \Phi_{-t_{n}}^h z_{h}^{n} \right|\\
    &\leq C \tau \|z_{h}^{n}\|_{L^2(\Gamma_h^{n})},
    \end{align*}
    where we have used that
    \[ \left| \frac{1}{|\Gamma_h^n|} - \frac{1}{|\Gamma_h(t)|} \right| \leq \left| |\Gamma_h^n| - |\Gamma_h(t)| \right| \leq C \tau, \]
    for the second inequality.
    Hence one can observe that from \eqref{timeperturb3}
	\begin{multline*}
		 \left|\intm \left( z_{h,0}^t, \intinvsh z_{h,0}^t - \Phi_t^h \Phi_{-t_{n}}^h \intinvsh z_{h,0}^{n} \right)-\intm\left( z_{h,0}^{n}, \overline{\intinvsh z_{h,0}^t - \Phi_t^h \Phi_{-t_{n}}^h \intinvsh z_{h,0}^{n}}(t_{n}) \right)\right|\\
		  \leq C \tau \|z_h^{n}\|_{L^2(\Gamma_h^{n})} \left\| \intinvsh z_{h,0}^t - \Phi_t^h \Phi_{-t_{n}}^h \intinvsh z_{h,0}^{n} \right\|_{L^2(\Gamma_h(t))}.
	\end{multline*}
	By a similar argument, using \eqref{timeperturb2}, one also obtains
	\begin{multline*}
		 \left|a_h\left( \intinvsh z_{h,0}^{n},  \overline{\intinvsh z_{h,0}^t - \Phi_t^h \Phi_{-t_{n}}^h \intinvsh z_{h,0}^{n}}(t_{n}) \right)- a_h \left( \Phi_t^h \Phi_{-t_{n}}^h \intinvsh z_{h,0}^{n}, \intinvsh z_{h,0}^t - \Phi_t^h \Phi_{-t_{n}}^h \intinvsh z_{h,0}^{n}  \right)\right|\\
		\leq C \tau \inormh{z_{h,0}^{n}} \left\| \gradgh \left( \intinvsh z_{h,0}^t - \Phi_t^h \Phi_{-t_{n}}^h \intinvsh z_{h,0}^{n} \right) \right\|_{L^2(\Gamma_h(t))}.
	\end{multline*}
	Lastly we note that $\inormh{z_{h,0}^{n}} \leq C \|{z_{h,0}^{n}}\|_{L^2(\Gamma_h^n)}  \leq C \|{z_{h}^{n}}\|_{L^2(\Gamma_h^n)}$, and by using Poincar\'e's inequality one finds
    \begin{align*}
        \left\| \intinvsh z_{h,0}^t - \Phi_t^h \Phi_{-t_{n}}^h \intinvsh z_{h,0}^{n} \right\|_{L^2(\Gamma_h(t))} &\leq C\left\| \gradgh \left( \intinvsh z_{h,0}^t - \Phi_t^h \Phi_{-t_{n}}^h \intinvsh z_{h,0}^{n} \right) \right\|_{L^2(\Gamma_h(t))}\\
        &+ C \tau\|z_{h,0}^{n}\|_{L^2(\Gamma_h^n)}.
    \end{align*}
    Here the rightmost term comes from the fact that
    \[ \left| \mval{\intinvsh z_{h,0}^t - \Phi_t^h \Phi_{-t_{n}}^h \intinvsh z_{h,0}^{n}}{\Gamma_h(t)}\right| = \left|\mval{\Phi_t^h \Phi_{-t_{n}}^h \intinvsh z_{h,0}^{n}}{\Gamma_h(t)}\right| \leq C \tau \inormh{z_{h,0}^n} \leq C \tau \| z_{h,0}^n \|_{L^2(\Gamma_h(t))}. \]
    This first equality follows from the fact that $\int_{\Gamma_h(t)}\intinvsh z_{h,0}^t = 0$, and the subsequent inequalities follow from \eqref{timeperturb1} and the Poincar\'e inequality along with the fact that $\int_{\Gamma_h^n}{\intinvsh z_{h,0}^n} = 0$.
    The result now follows.
\end{proof}

\bibliographystyle{acm}
\bibliography{bibliography.bib}

\end{document}